\documentclass[twoside,10pt]{article}
\usepackage[a4paper,scale=.79745,centering]{geometry} 
\usepackage{mathrsfs,amsmath,amssymb,amsthm,mathabx,graphicx}
\usepackage[mathlines]{lineno}
\usepackage[dvipsnames]{xcolor}
\title{On the profile of trees with a given degree sequence}
\usepackage{subfig}
\date{}
 \theoremstyle{plain}
\newtheorem{theorem}{Theorem}

\newtheorem{lemma}{Lemma}
\newtheorem{corollary}{Corollary}
\newtheorem{proposition}{Proposition}

\theoremstyle{definition}
\newtheorem{example}{Example}
\newtheorem*{definition}{Definition}
\newtheorem*{remark}{Remark}
\newcommand{\dunderline}[1]{\underline{\underline{#1}}}
\newcommand{\bo}[1]{\mathbf{#1}}
\newcommand{\dispand}{\quad\text{and}\quad}
\newcommand{\barbare}[1]{\ensuremath{ \overline{\overline{#1}} }}

\author{Osvaldo Angtuncio\thanks{Universidad Nacional Aut\'onoma de M\'exico.\hfill \texttt{osvaldo.angtuncio@matem.unam.mx}} \qquad Ger\'onimo Uribe Bravo\thanks{Universidad Nacional Aut\'onoma de M\'exico.\hfill \texttt{geronimo@matem.unam.mx}}}

\usepackage[pagebackref,colorlinks,citecolor=Plum,urlcolor=Periwinkle,linkcolor=DarkOrchid]{hyperref}

\usepackage[all,cmtip]{xy} 

\newcommand{\p}{\ensuremath{ \mathbb{P}  } }
\newcommand{\mc}[1]{\ensuremath{\mathscr{#1}}}
\newcommand{\na}{\ensuremath{\mathbb{N}}}
\newcommand{\set}[1]{\ensuremath{\left\{ #1\right\} }}
\newcommand{\proba}[1]{\ensuremath{\p\! \left( #1 \right)}}
\newcommand{\imf}[2]{\ensuremath{#1\!\paren{#2}}}
\newcommand{\paren}[1]{\ensuremath{\left( #1\right) }}
\newcommand{\fun}[3]{\ensuremath{#1:#2\to #3}}

\newcommand{\re}{\ensuremath{\mathbb{R}}}
\newcommand{\cadlag}{c\`adl\`ag}
\newcommand{\defin}[1]{{\bf #1}}
\DeclareMathOperator{\id}{Id} %
\newcommand{\esp}[1]{\ensuremath{\se\! \left( #1 \right)}}
\newcommand{\se}{\ensuremath{\mathbf{E}}}
\newcommand{\bra}[1]{\ensuremath{\left[ #1\right] }}
\newcommand{\floor}[1]{\ensuremath{\lfloor #1\rfloor}}
\newcommand{\abs}[1]{\hspace{.25mm}\left|#1\right|\hspace{.25mm}}
\newcommand{\eps}{\ensuremath{ \varepsilon}}
\newcommand{\indi}[1]{\mathbf{1}_{#1}}
\newcommand{\F}{\ensuremath{\mathscr{F}}}
\newcommand{\ceil}[1]{\ensuremath{\lceil #1 \rceil}}

\begin{document}
\maketitle
\begin{abstract}
		A \emph{degree sequence} is a sequence $\bo{s}=(N_i,i\geq 0)$ of non-negative integers 
		satisfying $1+\sum_i iN_i=\sum_i N_i<\infty$. 
		We are interested in the uniform distribution  $\p_{{\bf s}}$  on rooted plane trees 
		whose degree sequence equals ${\bf s}$, 
		giving conditions for the convergence of the profile 
		(sequence of generation sizes) 
		as the size of the tree 
		goes to infinity. 
		This provides a more general  formulation 
		and a probabilistic proof 
		of a conjecture due to Aldous \cite{MR1166406}. 
		Our formulation contains and extends results in this direction 
		obtained previously by Drmota and Gittenberger \cite{MR1608230} and Kersting \cite{kersting2011height}. 
		A technical result is needed to ensure that trees with law $\p_{{\bf s}}$ 
		have enough individuals in the first generations, 
		and this is handled 
		through novel path transformations and fluctuation theory of exchangeable increment processes. 
		As a consequence, 
		we obtain
		a boundedness criterion for  the inhomogeneous continuum random tree 
		introduced by Aldous, Miermont and Pitman  \cite{MR2063375}. 
\end{abstract}
\medskip

\noindent \emph{{Keywords:}}  
		Configuration model; 
		exchangeable increment processes; 
		Vervaat transform; Lamperti transform.

\medskip

\noindent \emph{{AMS subject classifications:}}   
		05C05; 
		34A36; 
		60F17;  
		60G09; 
		60G17; 
		60J80

\section{Introduction and statement of the results}\label{sectionIntroProfileOfTGDS}
Trees are an important concept in both pure and applied mathematics, 
appearing (for example) 
in biology to represent genealogies, 
in computer science as a fundamental data structure, 
as well as an important example of a combinatorial class or species 
(cf. \cite{MR3077154}, \cite{MR2251473}, \cite{MR633783},  \cite{MR2483235}, \cite{MR2484382}). 
Random trees, on the other hand, 
have been shown to be useful in analyzing the asymptotic behavior of certain families of deterministic trees. 
One of the most widely studied classes of random trees is 
that of Galton-Watson (GW) trees conditioned to have size $n$ (denoted CGW($n$)), 
some of which have been shown to be uniform in classes of trees of size $n$, 
like plane trees, binary plane trees or Cayley trees (cf. \cite{MR1630413}). 
Hence, asymptotic counting problems associated to such trees 
can be solved using the limiting continuum random tree introduced by Aldous 
(cf. \cite{MR1085326,MR1166406,MR1207226,MR2203728}).

We will analyze the class of trees with a given degree sequence because of two reasons. 
First,  for several real-world networks that have been analyzed, 
their degree sequence might have a certain feature such as having power law decay 
(see for example \cite{MR2091634,MR2857452}). 
Then, the simplest way to build an associated model of random trees 
is through the uniform distribution on trees whose degree sequence has the observed feature 
(in the random graph setting, this corresponds to the configuration model). 
Second, trees with a given degree sequence are more general than the widely studied CGW($n$) trees, 
since the latter laws can be obtained as \emph{mixtures} of the former. 
Part of the success in the study of Galton-Watson trees comes from their link to random walks; 
we wish to to argue that similar success can be had for trees with a given degree sequence thanks to their link with exchangeable increment processes. 

One way to understand the shape of rooted trees is through their profile, 
which counts the quantity of elements in the successive generations. 
(The introductions in \cite{MR2291961,MR2925941} 
summarize certain applications and references on the profile of random trees. )
A conjecture due to Aldous \cite{MR1166406} 
(for CGW($n$) having a finite variance offspring distribution)  
states that the rescaled profile converges in distribution to a multiple of the total local time process of the normalized Brownian excursion (NBE). 
Aldous's conjecture was proved in \cite{MR1608230} as a complex application of analytic combinatorics.
The latter work was generalized  in \cite{kersting2011height} 
to the case where the offspring distribution is in the domain of attraction of a stable law. 
In this paper, we state and prove a much more general version of Aldous's conjecture 
in the setting of trees with a given degree sequence.

Let us turn to the  
statements of our results. 
We define rooted plane trees following \cite{MR850756} and \cite{MR2203728}. 
Let $\mathbb{Z}_+=\{1,2,\ldots \}$ be the set of positive integers, 
and define $\mc{U}=\bigcup_{n=0}^{\infty}\mathbb{Z}_+^n$ 
as the set of all labels, 
using the convention $\mathbb{Z}^0_+=\{\varnothing\}$. 
An element of $\mc{U}$ is a sequence $u=u_1\cdots u_n$ of positive integers, 
where $|u|=n$ represents the \emph{generation or height} of $u$. 
If $u=u_1\cdots u_i$ and $v=v_1\cdots v_j$ belong to $\mc{U}$, 
write $uv=u_1\cdots u_iv_1\cdots v_j$ for the concatenation of $u$ and $v$. 
By convention $u\varnothing=\varnothing u=u$. 
For any $n\in \na$, let $[n]=\{1,\ldots, n\}$ with $[0]=\varnothing$.
\begin{definition}
	A \emph{rooted plane tree} $\tau$ is a finite subset of $\mc{U}$ such that:
	\begin{enumerate}
		\item $\varnothing\in \tau$,
		\item if $v\in \tau$ and $v=uj$ for some $j\in \mathbb{Z}_+$, then $u\in \tau$,
		\item for every $u\in \tau$, there exists a number $\chi(u)\in \na$, such that $uj\in \tau$ iff $j\in [\chi(u)]$.
	\end{enumerate}
\end{definition}
In the previous definition, the value $\chi(u)$ represents the number of children of $u$ in $\tau$.
The size of a tree $\tau$ (the number of individuals) will be denoted by $|\tau|$.
In the following, by a tree we mean a rooted plane tree. 
See Figure \ref{FigureOfTreeDFWBFW} for a graphical representation.

\begin{figure}
	\subfloat[Depth-first walk.]{	
		\includegraphics[width=.4\textwidth]{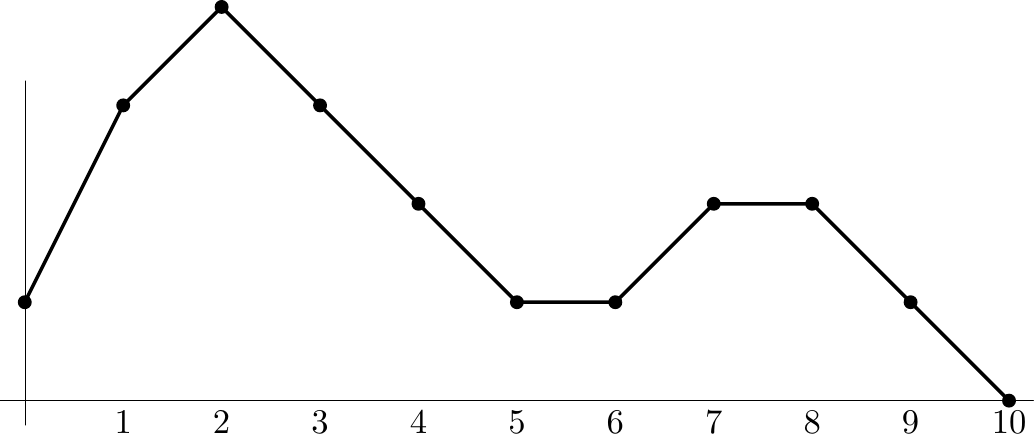}\label{figexcursionDFW}}
	\subfloat[
	Visual representation, where $u<v$ implies $u$ is to the left or below $v$
	]{	
		\includegraphics[width=.2\textwidth]{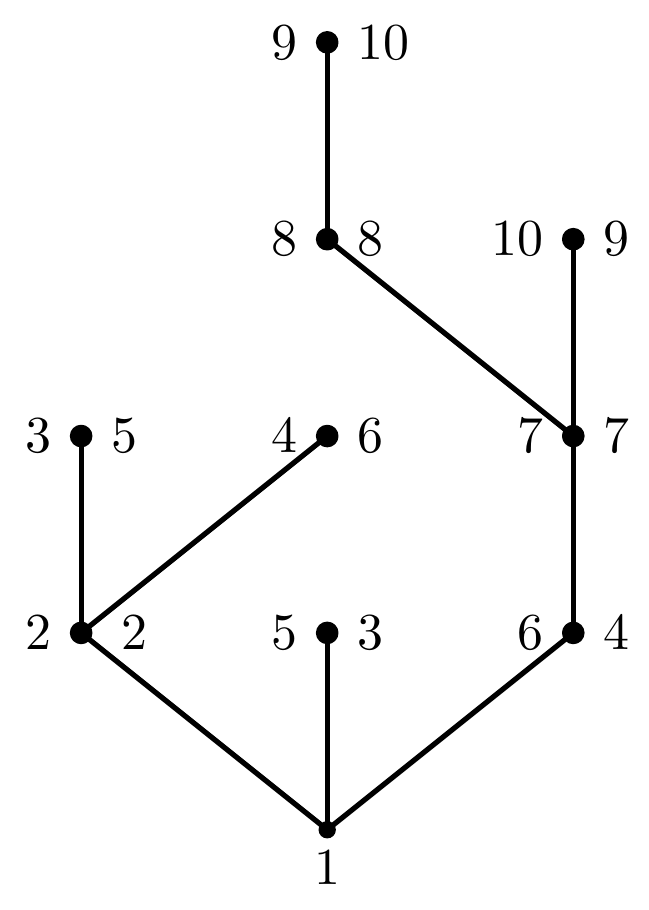}\label{figtreeDFW}
	 }
	\subfloat[Breadth-first walk and profile]{
		\includegraphics[width=.4\textwidth]{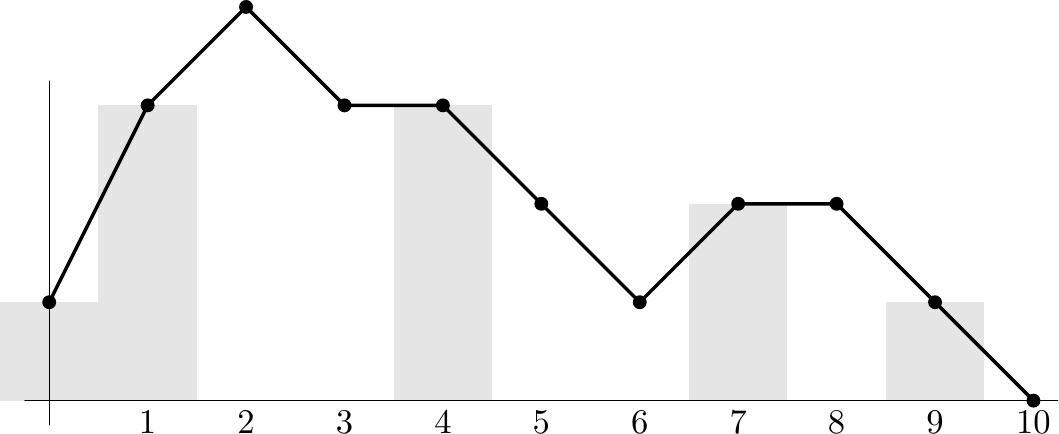}\label{figexcursionBFW}
	 }
	\caption{The tree 
	$\set{\emptyset<1<11<12<2<3<31<311<3111<312}$ 
	with profile $1,3,3,2,1$, labeled by pairs consisting of depth-first and breadth-first indices (left and right respectively). }
	\label{FigureOfTreeDFWBFW}
\end{figure}

Let us introduce the class of trees analyzed in this paper. 
\begin{definition}
A \emph{degree sequence} ${\bf s}$ is an integer sequence $(N_i,i\geq 0)$  satisfying $\sum N_i=1+\sum iN_i<\infty$. 
A random \emph{tree with a given degree sequence ${\bf s}$}  
is one whose law is uniform on the set of trees having degree sequence ${\bf s}$; 
we denote its law by $\p_{\bf s}$. 
The \emph{size} of the tree (or of the associated degree sequence) is the constant $s=\sum_i N_i$. 
\end{definition}
The integer $N_i$ represents the number of vertices with $i$ children of some rooted plane tree. 

\begin{example}[Trees with restricted degrees]
\label{Example_RestrictedDegreeTree}
Let $S\subset \set{1,2,\ldots}$ be a finite set of possible degrees. 
For every $s\in S$, choose $n_s\geq 0$ (but not all zero!). 
Then, set $N_i=n_i$ for $i\in S$ and choose $N_0=1+\sum_i(i-1)N_i$, 
so that $(N_i,i\geq 0)$ is a degree sequence. 

In particular, we can choose $S=\set{k}$ to obtain $k$-ary trees 
and if $n_k=n$ 
then the degree sequence is $N_k=n$, $N_0=1+(k-1)n$ 
(and $N_i=0$ for all $i\neq 0,k$) and its size  equals $1+nk$. 
\end{example}

\begin{example}[Galton-Watson trees]
\label{ExampleOfCGW}
Let $\mu$ be a distribution on the integers with mean in $(0,1]$; 
we think of $\mu_k$ as the probability that individuals in the population have $k$ offspring. 
A Galton-Watson tree with offspring distribution $\mu$ 
is a random rooted plane tree $\Theta$ such that, 
for any (finite) rooted plane tree $\tau$ of degree sequence ${\bf s}=(N_i)$, 
\[
	\proba{\Theta=\tau}
	=\prod_i \mu_i^{N_i}. 
\]Hence, conditionally on their degree sequence being ${\bf s}$ 
(which already conditions on their size being $s$), 
Galton-Watson trees have law $\p_{\bf s}$. 
In other words, a Galton-Watson tree (even if conditioned by its size) 
can be obtained by mixing the laws $\p_s$ 
with respect to the law of the degree distribution of $\Theta$. 
 
 A CGW$(n)$ tree has the law of $\tau$ conditioned on $|\tau|=n$. 
 In particular, if $\mu_0,\mu_k>0$ and $\mu_0+\mu_k=1$, 
 then the GW$(1+nk)$ is uniform on $k$-ary trees with $n$ inner vertices. 
 When $\mu$ has finite variance 
 or $\overline \mu(x)\sim x^{-\alpha}\imf{L}{x}$ for some $\alpha\in (1,2)$ 
 (where $L$ is a so-called slowly varying function which could be a constant), 
 \cite{MR1964956} has proved that CGW$(n)$ have scaling limits 
 which have been termed the Continuum Random Tree (CRT) of Aldous in the first case, 
 and  stable trees in the second case. 
\end{example}

Note that Example \ref{Example_RestrictedDegreeTree} is more combinatorial in nature 
than Example \ref{ExampleOfCGW}, 
even though they share the class of $k$-ary degree sequences. 
We will show more elaborate examples of degree sequences 
(having for example finite size versions of power law decay or showing applicability of our main results) in Section \ref{section_ExampleSection}. 
We would like to mention an example which is not inside our framework but that inspired it: $p$-trees. 
This family of growing random trees is introduced in \cite{MR1741774}. 
Fix $n$ and a probability measure $p=(p_i)_{1\leq i\leq n}$. 
A $p$-tree is a random tree labelled by $\set{1,\ldots, n}$, denoted $\Theta$,  
such that, for any labelled tree $\tau$ with $n$ vertices: 
\[
	\imf{\p}{\Theta=\tau}=\prod_{i=1}^n p_i^{\chi_i(\tau)},
\]where $\chi_i(\tau)$ is the number of children of the vertex with label $i$ in $\tau$ 
(cf. formula (5) in \cite{MR1741774}). 
If now $(p^n)_n=((p^n_i)_{i\leq n})_{n\geq 1}$ is a sequence of probability measures 
such that $p^n_{k+1}<p^n_k$, 
\cite{MR1741774} and \cite{MR2063375} are interested in the following asymptotic regime:
\[
	p^n_1\to 0\quad\text{and}\quad \frac{p^n_i}{s_n}\to \beta_i\quad\text{where}\quad
	s_n=\sqrt{\sum_{i}(p^n_i)^2}. 
\]In a sense, $p^n_i$ represents the size of the $i$-th individual with most children in the $n$-th tree 
(individuals with a big quantity of children are also referred to as hubs in the tree); 
then $s_n$ can be interpreted as the characteristic hub size in the $n$-th tree. 
In this case, the tree $\Theta_n$ with law $\p_{p^n}$ 
can be scaled to converge 
(in the sense of sampling, also called the Gromov-weak sense in \cite{MR3522292}) 
to the so-called inhomogeneous continuum random tree (or ICRT$(\theta)$). 
When $\beta_i=0$ for all $i$, the ICRT corresponds to Aldous' CRT. 
A natural conjecture is that the scaling limit of trees with a given degree sequence is also the ICRT$(\beta)$; 
this has been proved under some conditions (which imply $\beta_i=0$ for all $i\geq 1$) in \cite{MR3188597}. 
One open question in \cite{MR2063375} is whether the height of the ICRT is finite; 
by analogy to the case of L\'evy trees (cf. \cite{MR1954248}) 
this might be equivalent to the ICRT being compact. 
We will be able to give a partial answer to boundedness as an application of our techniques. 
We expect this to be relevant to proving convergence of trees with a given degree sequence to the ICRT. 

The main characteristic we will study in a plane tree are the profiles. 
\begin{definition}
Let $\tau$ be a plane tree. 
Define $\fun{c_\tau}{\na}{\na}$ so that 
$c_\tau(j)$ is the total number of vertices of $\tau$ up to generation $j$. 
Then, $c_\tau$ is called the \emph{cumulative profile} or \emph{cumulative population process} of the tree $\tau$. 

Let $\fun{z_\tau}{\na}{\na}$ be such that 
$z_{\tau}(j)$ is the number of vertices in $\tau$ in generation $j$. 
Then $z_\tau$ is called the  \emph{profile} of the tree $\tau$. 
We write $z$ and $c$ when there is no risk of confusion. 
\end{definition}

Our study the profile exploits its connection to breadth-first walks as follows. 
We can order the vertices of the tree according to the \emph{lexicographical order} 
(e.g., $\varnothing<1<21<22$), 
and assign label $i$ to the $i$th vertex, for $i\in [|T|] $. 
The \emph{depth-first walk} (DFW) of the tree 
will be the walk with $i$th increment $c(i)-1$, started at one.
See Figure \ref{FigureOfTreeDFWBFW} for an example. 
This ordering is also called \emph{depth-first order} and, 
together with the associated depth-first walk,  
is fundamental in understanding distances in the tree 
(cf. \cite{MR2203728,MR1954248}). 
Another useful labeling of the tree is the \emph{breadth-first order}. 
To define this, assign label 1 to the root. 
Suppose the first generation (offspring of the root) has size $z_1$. 
Order the first generation in lexicographical order, 
and assign label $i$ to the $i$th vertex, for $i\in \{2,\ldots, 1+z_1 \}$, 
continuing this way for each consecutive generation. 
The \emph{breadth-first walk} (BFW) of the tree will be 
the walk with $i$th increment $c(i)-1$, started at one; 
denote it by $x$. 
See Figure \ref{FigureOfTreeDFWBFW} for an example. 
As we now discuss, the breadth-first walk $x$ is the key to understanding the profile of a tree. 
Indeed, a simple counting argument of a proof by induction 
show us that $z_n$ can be recursively obtained by setting  
\begin{equation}
	\label{equation_DiscreteLampertiTransform}
	z_0=1,\quad
	c_k=z_0+\cdots + z_k
	\quad\text{and}\quad
	z_{k+1}=x\circ c_k
\end{equation}(cf. Chapter 9 of \cite{MR838085} and the introduction in \cite{MR3098685}). 
Also, breadth-first and depth-first walks have the same law; 
this can be seen as a manifestation 
of their link with exchangeable increment processes (cf. proof of \cite[Lemma 7]{MR3188597}).

For every $n\in\na$, let ${\bf s_n}=(N_i^n ,i\geq 0)$ be a \emph{degree sequence}. 
Consider a tree $\Theta_n$ with law $\p_{{\bf s_n}}$. 
The BFW of $\Theta_n$ will be denoted by $X^{n}$, its cumulative profile by $C^n$ and its profile by $Z^n$ and recall that $Z^{n}(k+1)=X^n(C^n(k))$. 
By analogy with the continuous-time case (cf. \cite{MR0208685}), 
$Z^{n}$ is called the \emph{discrete Lamperti transform} of $W^{n}$. 
Figure \ref{FigureOfTreeDFWBFW} shows an example of a BFW and its Lamperti transform. 
For certain classes of trees 
(such as finite variance CGW$(n)$ trees), 
the breadth-first walks $X^n$ have a scaling limit 
(in the sense that $X^n(n\cdot)/\sqrt{n}$ converges weakly to a stochastic process $X$, 
which is the so-called normalized Brownian excursion for finite variance CGW$(n)$ trees.)
To study the scaling limit of the profile, 
the general idea is 
to prove, for some sequence $b_n\to\infty$,  
that the scaling ($Z^n(s_n \cdot/b_n)/b_n$) of the profile has a limit 
which is a particular solution to
\begin{equation}
	\label{eqnLampertiTransformIntro}
	Z(t)=X\circ C(t) \quad\text{with}\quad C(t)=\int_0^tZ(s)\,ds,
\end{equation}in analogy to \eqref{equation_DiscreteLampertiTransform}, 
where $X$ is the limit of the rescaled BFWs. 
Before our main result, 
we will consider scaling limits of BFWs and then a characterization of solutions to \eqref{eqnLampertiTransformIntro}. 
BFWs have a simple probabilistic structure: 
they correspond to the Vervaat transform of exchangeable increment processes 
as shown in \cite{MR3188597} and recalled in Section \ref{subsection_SubsequentialScalingLimits}. 

In our case, the scaling limit $X$ of the breadth-first walks of the trees $\Theta_n$ is related 
to a continuous time process $X^b$ with exchangeable increments 
(abridged EI process) 
of the form
\begin{equation}
\label{eqnEIP}
X^b(t)=\sigma b(t)+\sum_{j=1}^\infty \beta_j\paren{{\bf 1}\paren{U_j\leq t}-t}\ \ \ \ \ \ t\in[0,1],
\end{equation}where $b$ is a Brownian bridge on $[0,1]$, 
$(U_j,j\geq 1)$ are independent of $b$ and i.i.d. with uniform law on $[0,1]$, 
and with constants $\sigma \in \re^+$, $\beta_1 \geq \beta_2\geq \cdots \geq 0$ with $\sum\beta_j^2<\infty$. 
(From Kallenberg's representation \cite{MR0394842}, 
the process $X^b$ has canonical parameters $(0,\sigma, \beta)$.) 
Recall that $X^b$ can be considered as a random element of \emph{Skorohod space}, that is, of the space of functions on $[0,1]$ which are right-continuous and have left limits (abridged \cadlag). 
The ``excursion-type" process $X$ serving as a scaling limit of the breadth first walks $X_{n}$, 
is obtained from $X^b$ using the  \emph{Vervaat } transformation, 
which exchanges the pre and post minimum parts of $X^b$. 
It is formally defined as follows and visualized in Figure \ref{figtransfVervaat}.

\begin{definition}[Vervaat transformation]
Let $X=(X_t,t\in [0,1])$ be a stochastic process with \cadlag\ paths 
which reaches its infimum value uniquely and continuously at $\rho$. 
Assume that $X_0=X_1=0$. 
The \defin{Vervaat transform} of $X$ is the \cadlag\ stochastic process $V=(V_t,t\in [0,1])$ 
defined by $V_t=X_{\{t+\rho\}}-X_t$, where $\{t\}$ is the fractional part of $t$. 
\end{definition}
\begin{figure}
		\includegraphics[width=.475\textwidth]{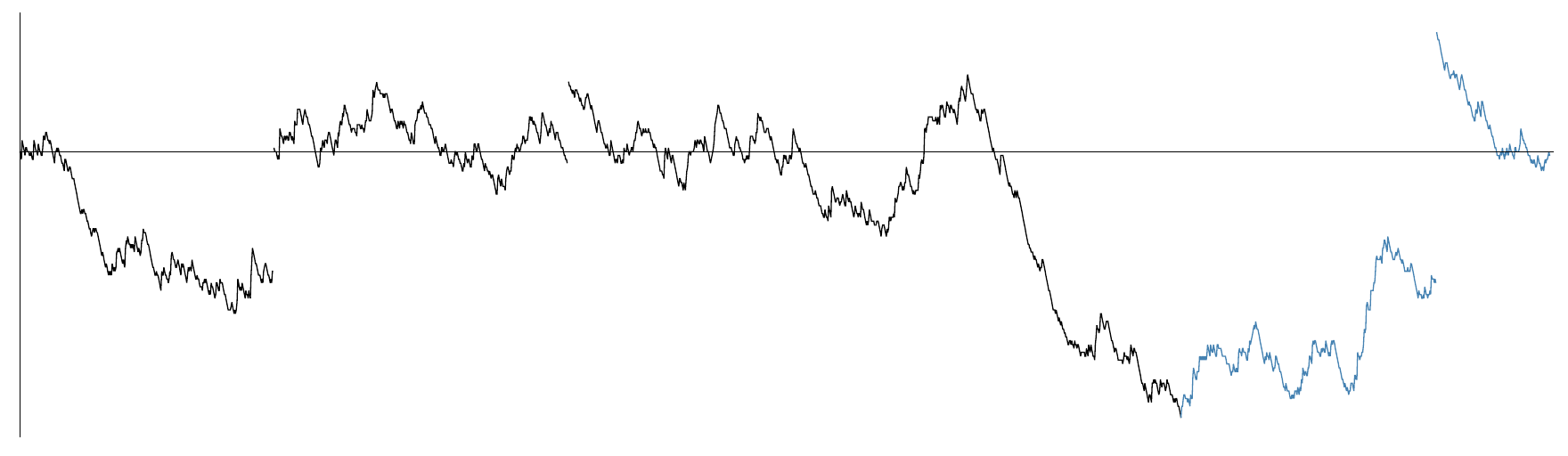}
		\hfill
		\includegraphics[width=.475\textwidth]{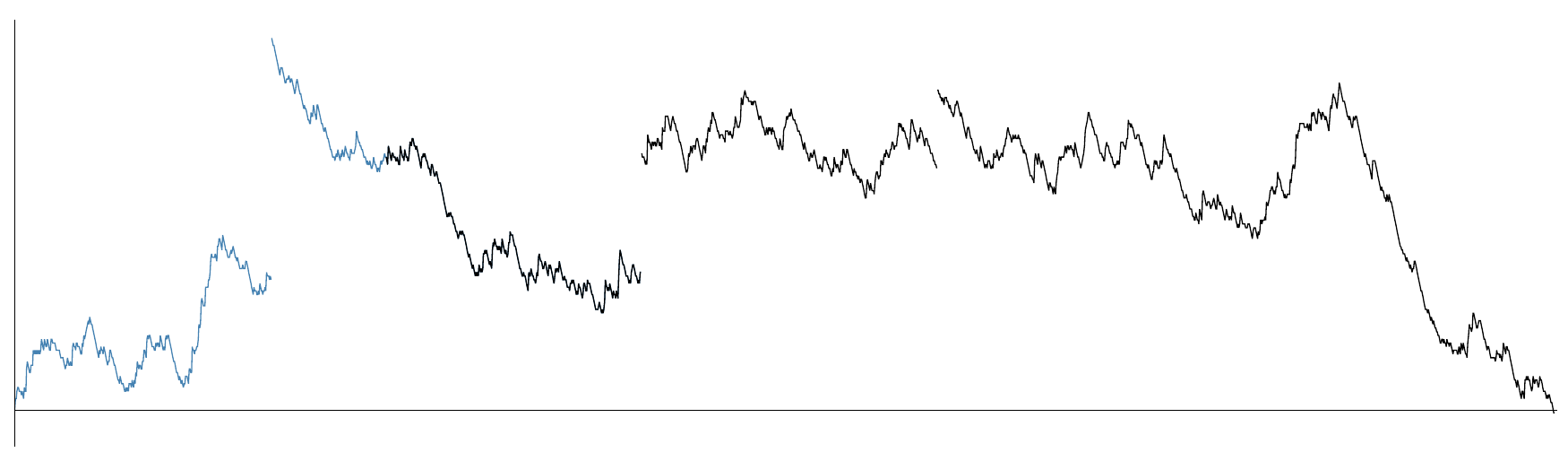}
		\caption{\footnotesize Discontinuous EI process (left) and its Vervaat transform (right). 
		The EI process attains its infimum continuously at a unique time 
		and its Vervaat transform is positive.}
		\label{figtransfVervaat}
\end{figure}

For a given degree sequence ${\bf s_n}$, 
let $(d(j),j\in [s_n])$ be the associated \emph{child sequence}, 
obtained by writing $N^n_0$ zeros, $N^n_1$ ones, etc.
, and ordering the resulting sequence decreasingly. 
The ordering shifts focus into the hubs (individuals with the most offspring) as in the case of $p$-trees. 
Note that $N^n_j=|\{i:d(i)=j \}|$. 
Recall that for \cadlag\ functions $\fun{f_n,f}{[0,1]}{\infty}$, $f_n\to f$ 
if there exists a sequence of increasing homemorphisms $\fun{\lambda_n}{[0,1]}{[0,1]}$ 
such that $f_n\circ\lambda_n\to f$ and $\lambda_n\to \id$ uniformly. 
This notion of convergence comes from several metrics 
and defines the Skorohod topology (cf. \cite{MR1700749} and \cite{MR838085}). 
Recall that convergence to $f$ in the Skorohod topology 
coincides with uniform convergence on compact sets 
whenever $f$ is continuous. 
We will use the product topology whenever the convergence of more than one \cadlag\ function is required. 
Also, if $X^n,X$ are \cadlag\ processes, the weak convergence
\[ 
	X^n\stackrel{d}{\to}X\text{ means that }\esp{\imf{F}{X^n}}\to \esp{\imf{F}{X}}
\]for any $F$ which is bounded and continuous on \cadlag\ functions. 

\begin{proposition}
\label{proposition_ConvergenceOfBFWs}
Let ${\bf s_n}$  be a degree sequence for every $n\in\na$ of size $s_n$, where ${\bf s_n}=(N^n_i,i\geq 0)$. 
Denote by $d^n$ its ordered child sequence. 
Let $\tilde X^n$ be the breadth-first walk of a uniform tree with degree sequence $\mathbf{s_n}$, extended by constancy on each interval $[i,i+1)$ for $i\in\na$. 
Assume that\begin{description}
\item[Size] $s_n\to\infty$, 
\item[Hubs] there exists $b_n\to\infty$ such that, for every $i\geq 1$,  $(d^n_i/b_n)$ is convergent to a limit $\beta_i\geq 0$, 
\item[Degree variance] there exists $\sigma\in [0,\infty)$ such that $\frac{1}{b_n^2}\sum_i (i-1)^2 N_i\to \sigma^2+\sum_i\beta_i^2$, and 
\item[Unbounded variation] either $\sigma^2>0$ or $\sum_i \beta_i=\infty$. 
\end{description}
Define the scaled breadth-first walk $X^n$ as $\tilde X^n(s_n\cdot)/b_n$. 
Then, the EI process $X^b$ given by \eqref{eqnEIP} achieves its minimum uniquely and continuously and 
the sequence $(X^n)$ converges weakly to the Vervaat transform $V(X^b)$. The process $V(X^b)$ is strictly positive on $(0,1)$. 
\end{proposition}
In the above proposition, the hypotheses \defin{size, hubs} and \defin{degree variance} are exactly those that are needed to apply the characterization and convergence results for EI processes of \cite{MR0394842} and \cite[Theorem 3.13]{MR2161313} and obtain convergence of the bridge-like processes. To obtain the convergence of the breadth-first walks, we need results concerning the continuity properties of the Vervaat transformation, like the identification of conditions to ensure that EI processes reach their minimum uniquely and continuously, obtained recently in \cite{2019arXiv190304745A}. 
This is where hypothesis \defin{unbounded variation} is relevant. 

Regarding solutions to \eqref{eqnLampertiTransformIntro}, we now characterize them in terms of a very special one, the Lamperti transform of $X$. Randomness is not needed for the result. 
\begin{definition}[Lamperti transformation]
For a given function $\fun{f}{[0,1]}{\re_+}$, 
which is right-continuous and admits left limits 
(abridged \cadlag), 
let 
\[
	i(t)=\int_0^t \frac{1}{\imf{f}{s}}\, ds
\]and define the right-continuous inverse of $i$, denoted $c^0$, by
\[
	c^0(t)=\inf\set{s\geq 0: i(s)>t},
\]where, by convention $\inf\emptyset=1$. 
The Lamperti transform of $f$ is $h^0=f\circ c^0$. 
We call the pair $(h^0,c^0)$ the Lamperti pair associated to $f$.
\end{definition}

In the next result, $D_+h$ denotes the right-hand derivative of the function $h$. 
\begin{proposition}
\label{proposition_deterministicIVP}
Let $\fun{f}{\re_+}{\re_+}$ be a \cadlag\ function with non-negative jumps $\Delta f(t)=f(t)-f(t-)\geq 0$. 
Assume that $f=0$ on $\set{0}\cup[1,\infty)$ 
 and $f>0$ on $(0,1)$. 
Let $(h^0,c^0)$ be the Lamperti pair associated to $f$. 
Then, solutions to the equation
\begin{equation}
\label{deterministicIVP}
c_0=0,\quad
D_+ c=f\circ c
\end{equation}can be characterized as follows:
\begin{enumerate}
\item If $\int_{0+}\frac{1}{\imf{f}{s}}\, ds=\infty$ then $h,c=0$ and $0$ is the unique solution to \eqref{deterministicIVP}. 
\item If $\int_{0+}\frac{1}{\imf{f}{s}}\, ds<\infty$, then $c^0$ is not identically zero and $D_+c^0=h^0$, 
so that $c^0$ solves \eqref{deterministicIVP}. 
Furthermore, solutions to \eqref{deterministicIVP} conform a one-parameter family $(c^\lambda)_{\lambda\in [0,\infty]}$, 
\[
	c^\lambda(t)=c^{0}( \bra{\lambda-t}^+  )
	\text{ for $\lambda>0$ and}\quad
	c^\infty=0.
\]Finally, assume that $\int_{0+}\frac{1}{\imf{f}{s}}\, ds<\infty$.  \begin{enumerate}
\item If $\int^{1-}\frac{1}{\imf{f}{s}}\, ds=\infty$ then $c^0$ is strictly increasing on $\re_+$, with $c^0(\infty)=1$.
\item If $\int^{1-}\frac{1}{\imf{f}{s}}\, ds<\infty$ then $c^0$ is strictly increasing until it reaches the value $1$ at the finite time $\int_0^1 \frac{1}{\imf{f}{s}}\, ds$. 
\end{enumerate}
\end{enumerate}
\end{proposition}
Note that the Lamperti transformation is strictly increasing, plateauing if it reaches $1$. 
Hence, under the conditions of Proposition \ref{proposition_deterministicIVP}, 
solutions $c$ to \eqref{deterministicIVP} 
cannot have constancy intervals except when they take the values zero or one. 
This is a key fact which will allow us to prove our main theorem. 

In \cite{MR2063375}, it is proved that boundedness of the ICRT is equivalent to $\int_0^1 1/X_s\, ds<\infty$. 
Having seen that this is related to (non-triviality and finite time absorption of) 
the Lamperti transformation of the scaling limits of breadth-first walks, 
we offer the following sufficient conditions which are simple to check. 
For a sequence $\beta_i\downarrow 0$, 
define $\overline \beta(x)=\#\{i: \beta_i>x\}$. 
\begin{proposition}
\label{proposition_CompactnessOfICRT}
Let $X$ be the Vervaat transform of the EI process $X^b$ given in \eqref{eqnEIP}. 
If $\sigma^2>0$ then $\int_0^1 1/X_s\, ds<\infty$. 
Otherwise, if $\sigma^2=0$  and 
\begin{enumerate}
\item \label{lowerBGAssumption}
$\lim_{x\to 0}x^\alpha \overline \beta(x)\to \infty$ for some $\alpha\in (1,2)$ then $\int^{1-} 1/X_s\, ds<\infty$. 
\item If furthermore  there exists $\tilde \alpha<1/(2-\alpha)$ such that $\sum_i \beta_i^{\tilde \alpha}<\infty$ then $\int_{0}^1 1/X_s\, ds<\infty$. 
\end{enumerate}
\end{proposition}
In the above proposition, recall that  $\sum_i \beta_i^2<\infty$. 
Hence,  $\int_0^1 1/X_s\,ds<\infty$ if condition \ref{lowerBGAssumption} above holds for some $\alpha>3/2$. 
We also obtain the integrability of $1/X$ when $\beta$ has power law type decay, or more broadly, when 
\[
	1<\sup\set{\alpha: \lim_{x\to 0}x^\alpha \overline\beta(x) =\infty}	  
	=\inf\set{\alpha: \lim_{x\to 0} x^\alpha\overline\beta(x)=0}. 
\]Based on the case of scaling limits of Galton-Watson trees, 
\cite{MR2063375} conjectured a necessary and sufficient condition in terms of $\sigma$ and $\beta$ for finitude of the integral of $1/X$, 
which is required for compactness of the ICRT. 
Our condition is only sufficient, 
but it is acknowledged in \cite{MR3748328} that ``checking compactness...turns out to be quite intractable". 
In that article, one obtains 
annealed results where $\beta$ is random and satisfies almost surely the last display. 
See also \cite{2020arXiv200502566B, 2018arXiv180405871B,2020arXiv200202769B} for more up to date accounts on annealed criteria for the compactness of the ICRT coming from the study of a different random graph model, which gives further evidence for the boundedness conjecture in \cite{MR2063375}. 

Our main theorem is the following.

\begin{theorem}
\label{teoUnicidadEnFnalGralIntro}
Let $({\bf s}_n)$ be a sequence of degree sequences of sizes $(s_n)$ 
and let $\tilde X^n$ be the breadth-first walk of a uniform tree with degree sequence ${\bf s}_n$. 
Assume the existence of constants $b_n\to\infty$ 
such that $X^n=\tilde X^n(s_n \cdot)/b_n$ converges weakly to 
the Vervaat transform $X$ of an unbounded variation exchangeable increment process. 
Let $(Z,C)$ be the Lamperti pair associated to $X$. 
Define $\tilde C^{n}$ as the cumulative population 
and $\tilde Z^{n}$ the population profile of a uniform tree with degree sequence ${\bf s_n}$, as well as
\[
	C^n=\frac{1}{s_n} \imf{\tilde C^n}{\frac{s_n}{b_n}\cdot}
	\quad\text{and}\quad
	 Z^n= \frac{1}{b_n}\imf{\tilde Z^n}{\frac{s_n}{b_n}\cdot}. 
\]Then, under the hypotheses $s_n/b_n\to\infty$ and 
\begin{equation}\label{enoughIndividualsAtTheTopOfTheTreeHypothesis}
\int^{1-}\frac{1}{X_s}ds<\infty\quad a.s.,
\end{equation}we have the joint convergence	
$\paren{X^n,Z^n, C^n}\stackrel{d}{\to} \paren{X,Z,C}$ 
under the product Skorohod topology. 
\end{theorem}

\begin{figure}
\includegraphics[width=\textwidth]{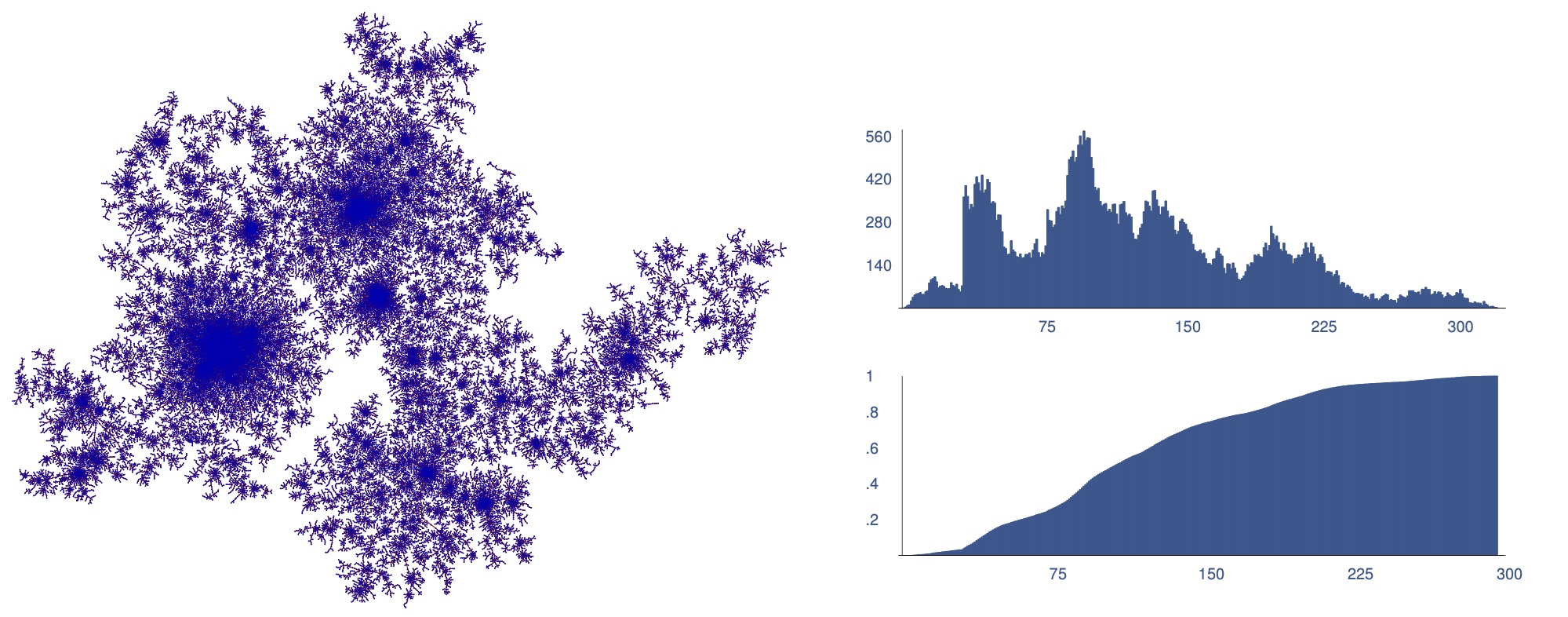}
\caption{Uniformly sampled tree with a given degree sequence of size 58000 whose BFW approximates an EI process with parameters $(0,2,(1/i)_{i
\geq 1})$ together with its population and cumulative profiles. 
}
\label{figTree58000Pareto}
\end{figure}

As a simple application, 
consider the case  where ${\bf s_n}=(N^n_i,i\geq 0)$ is the $k$-ary degree sequence 
with $n$ inner (non-leaf) vertices.
Then,  the sequence $({\bf s_n})$ satisfies the conditions of Proposition \ref{proposition_ConvergenceOfBFWs}. 
Indeed, obviously the sizes $1+nk$ grow to infinity, 
\[
	\sum (i-1)^2N^n_i= 1+k(k-1)n
\]so that we can take $b_n=\sqrt{n}$ 
to see that hypothesis \defin{degree variance} holds and that \defin{hubs} holds with $\beta=0$. 
The limiting EI process then has parameters $(0,\sqrt{k(k-1)},0)$. 
Hence, the limiting EI process $X^b$ is a (non-zero) multiple of the Brownian bridge 
(which satisfies Hypothesis \defin{unbounded variation}); 
its Vervaat transform $X$ is called the normalized Brownian excursion (cf. \cite{MR515820}). 
Proposition \ref{proposition_CompactnessOfICRT} tells us that $1/X$ is integrable; 
since $s_n/b_n\sim k\sqrt{n}\to \infty$, the hypotheses of Theorem \ref{teoUnicidadEnFnalGralIntro} therefore hold 
and we conclude a limit theorem for the profiles of $k$-ary trees with $n$ inner vertices as $n\to\infty$. 
The scaling limit of the profile is the Lamperti transform of the normalized Brownian excursion; 
this has the same law as the total local time process of the normalized Brownian excursion thanks to a celebrated theorem of Jeulin (cf. \cite{MR884713} and \cite{MR2063375}). 
We could similarly treat the case of restricted degree trees, with the same scaling sequences and scaling limit. 

More generally, one can easily build degree sequences satisfying the hypotheses of Theorem \ref{teoUnicidadEnFnalGralIntro}; see Section \ref{section_ExampleSection}. 
Also, the theorem can also be applied to mixtures of trees with a given degree sequence. 
This is exemplified in the following more complex application to the the convergence of the rescaled profile of CGW($n$) trees of Example \ref{ExampleOfCGW}.

Recall that a distribution $\mu=(\mu_n,n\geq 0)$ is called \emph{critical} if $\sum n\mu_n=1$, and \emph{aperiodic} if the greatest common divisor of all $n$ with $\mu_n>0$ is one. 

\begin{corollary}
	\label{corolarioConvergenciaCGWVarFinIntro}
	Consider a CGW($n$) tree with offspring distribution $\mu$, which is critical
	 and aperiodic. 
	Assume also that either $\mu$ has finite variance $\sigma^2$ 
	or that its tails satisfy $\imf{\overline \mu}{k}=k^{-\alpha}\imf{L}{k}$ for some slowly varying function $L$. 
	Denote by $X^n$, $C^n$ and $Z^n$ its rescaled breadth-first walk, cumulative profile and profile 
	as in Theorem \ref{teoUnicidadEnFnalGralIntro}, 
	where $s_n=n$ and $b_n$ equals  $\sqrt{n}$ in the finite variance case 
	and $n^{1/\alpha}\tilde L(n)$ 
	for some  slowly varying function $\tilde L$ 
	otherwise. 
	Then, we have the joint convergence
	\begin{equation}\label{eqnPerfilReescaladoVariFinIntro}
	\paren{X^n,Z^n,C^n}
	\stackrel{d}{\to}\paren{X,Z,C},
	\end{equation}where $X$ is the normalized Brownian excursion (multiplied by $\sigma^2$) in the first case and otherwise is the normalized stable excursion and $(Z,C)$ is the Lamperti pair associated to $X$. 
\end{corollary}
The normalized stable excursion equals the Vervaat transform of the bridge of a stable process. 
More information about it can be found in \cite{MR1465814}. 
This gives another proof of Aldous's conjecture \cite{MR1166406} in the finite variance case, 
first proved in \cite{MR1608230}. 
Kersting analyzed the second case in \cite{kersting2011height}. 
We would be able to give a further extension to the case where $\mu$ is allowed to vary with $n$. 
What is lacking are local limit theorems for triangular arrays, 
which are relevant to obtain the scaling limit of L\'evy bridges, 
as can be seen from the proof. 

We conclude this section with an overview of the proof of Theorem \ref{teoUnicidadEnFnalGralIntro}. 
First, we perform a deterministic analysis of the Lamperti transformation and its probabilistic counterpart. 
In particular, we obtain subsequential compactness and, through Proposition \ref{proposition_deterministicIVP}, 
that all subsequential limits in the context of Theorem \ref{teoUnicidadEnFnalGralIntro} are of the form $(X,Z^\Lambda, C^\Lambda)$ for some random $\Lambda$. 
(Recall that $C^\Lambda$ is a shift of $C$ and that $Z^\Lambda$ is its right-hand derivative). 
In order to prove convergence, we must prove that all subsequential limits agree: this is done by showing that
$\Lambda=0$ almost surely, 
which can be interpreted as an asymptotic thickness near the root (or at the base) of our random tree. 
A novel path transformation for discrete EI process, called the \emph{213 transformation}, is introduced to show that, under the conditions of Theorem \ref{teoUnicidadEnFnalGralIntro}, our random tree sequence is asymptotically thick. 
\begin{definition}[213 Transformation]\label{defi213Transformation}
Let $\tau$ be a plane tree labeled in depth-first order. 
	Let $v\in \{2,\ldots,|\tau| \}$ be a vertex in $\tau$ and let $u$ be any strict ancestor of $v$.
	Cut $\tau$ at $u$ and $v$, obtaining three subtrees keeping their original labels: 
	$\tau(1,u)$ has the original root, 
	$\tau(u,v)$ with root $u$, 
	and $\tau(v)$ with root $v$.
	 (If $v$ is a leaf, then $\tau(v)$ is empty) 
	Construct a new tree $\Psi_{u,v}(\tau)$ by grafting (or pasting) the root of $\tau(v)$ at $u$, 
	and then further grafting this structure at the leaf $v$ of $\tau(u,v)$. 
	If $w$ is the depth-first walk of $\tau$, denote by $\Psi_{u,v}(w)$ the depth-first walk of $\Psi_{u,v}(\tau)$. 
\end{definition}
An example is shown in Figure \ref{figTrans213}. 
We obtain in this way a new tree with the same degree sequence; 
more importantly, when choosing $u$ and $v$ adequately, this transformation preserves the law $\p_{\bf s}$. 
\begin{figure}
	\centering
	\subfloat[]{
	\includegraphics[width=.3\textwidth]{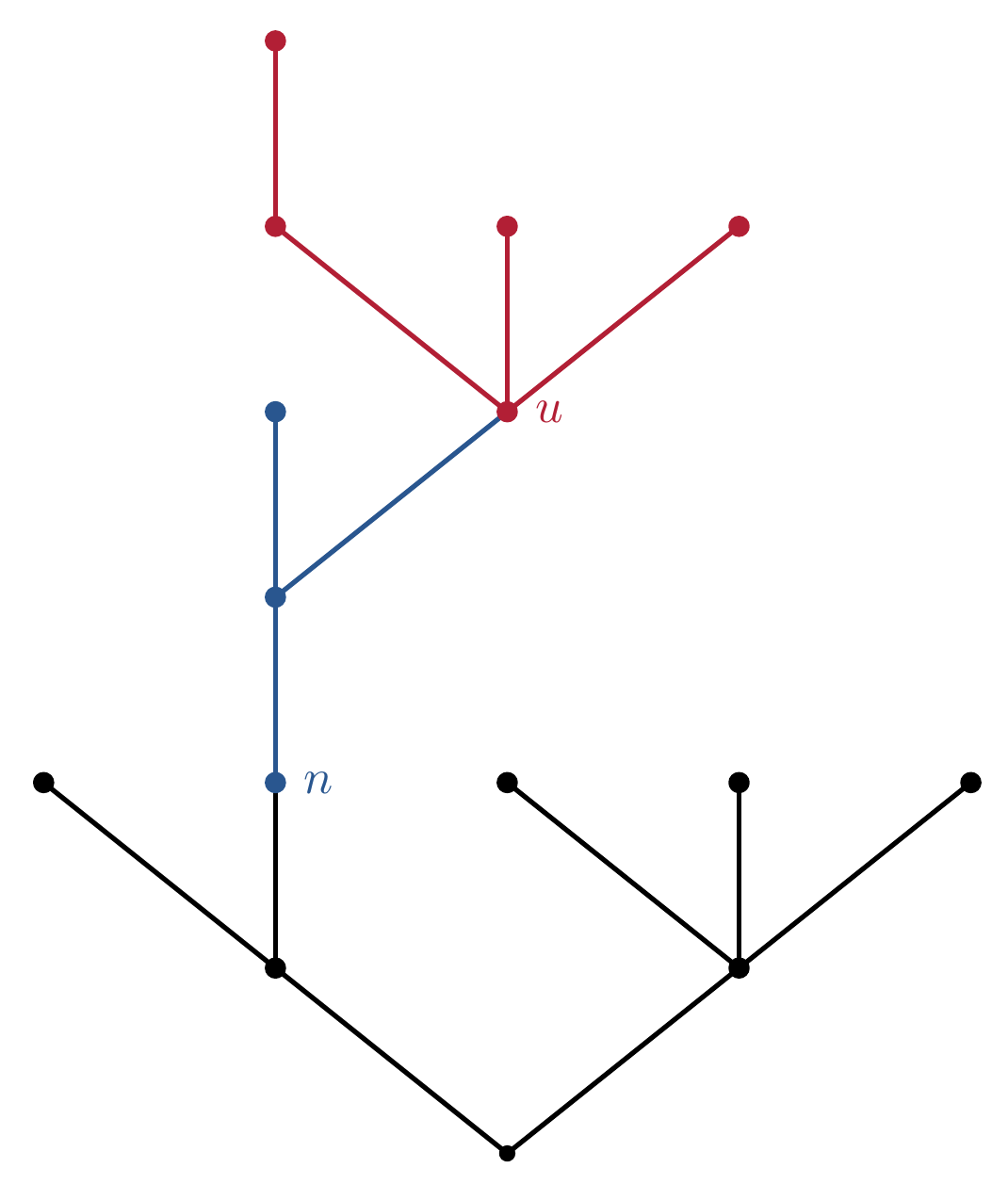}}
	\quad\quad\quad\quad\quad\quad\quad
	\subfloat[]{\includegraphics[width=.3\textwidth]{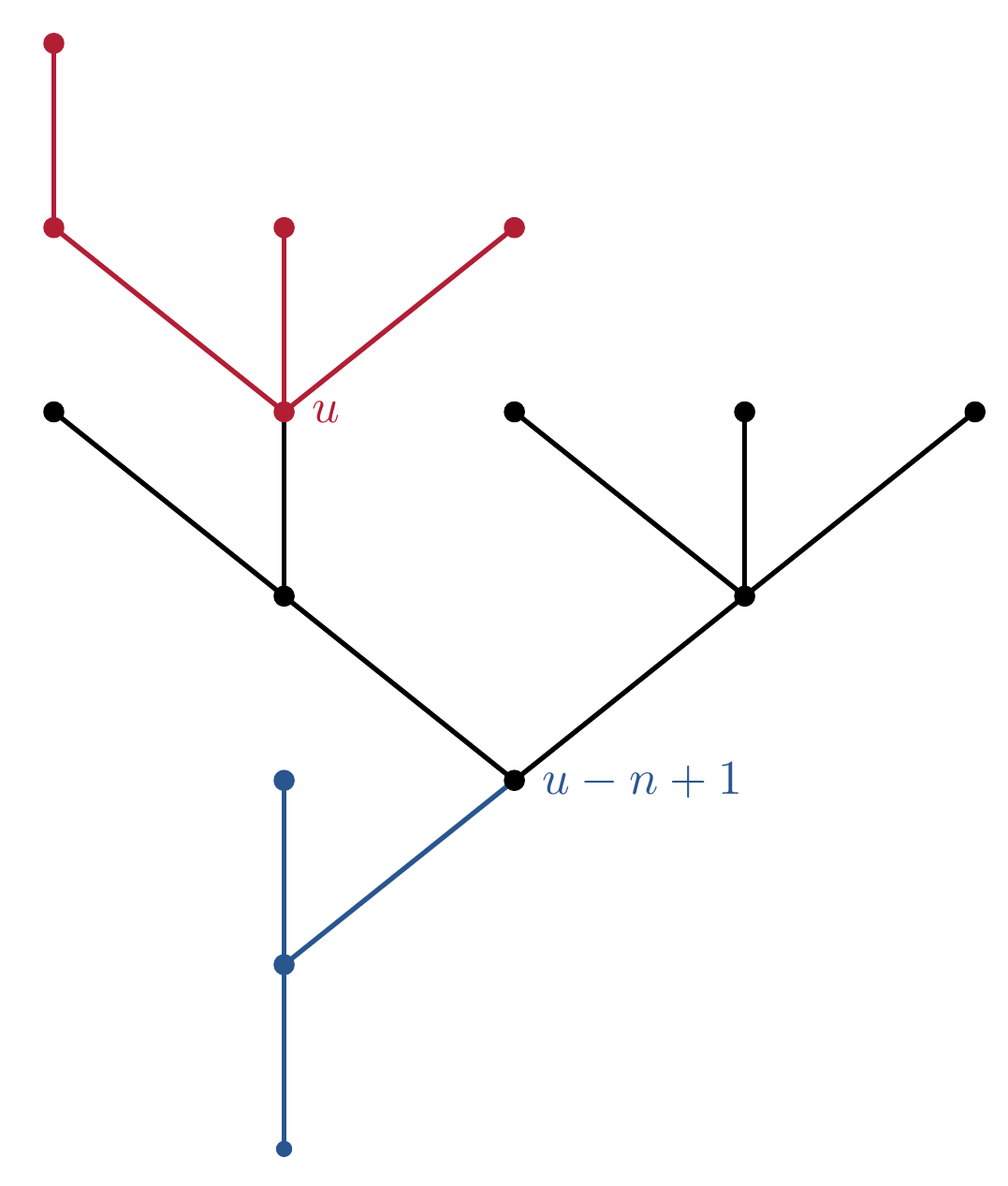}}
	\caption{
	The effect of the 213 transformation on left tree at vertex $u=7$ and its ancestor at height $3$. 
	} 
	\label{figTrans213}
\end{figure}
\begin{proposition}\label{propoIgualdadDistribucionalIntro}
	Let $W$ be the DFW of a tree $\Theta$ with law $\p_{{\bf s}}$. 
	Consider an independent uniform r.v. $V$ on $\{2,\ldots, |{\bf s}| \}$, and a positive integer $h$. 
	When the vertex $V$ in depth-first order has height greater than $h$, 
	let $U$ be the only ancestor of $V$ at distance $h$ from $V$, 
	and define $\tilde{W}:=\Psi_{V,U}(W)$. 
	If the height of $U$ is $\leq h$, set $\tilde{W}=W$. 
	Then $(\tilde W, V)$ and $(W,V)$ have the same law. 
\end{proposition}

To apply this in the context of Theorem \ref{teoUnicidadEnFnalGralIntro}, 
consider an arbitrary $\lambda\in \re$ and a sequence of trees $\tau_n$ with law $\p_{{\bf s_n}}$; 
we keep the same notation for a subsequence where the triplet $(X_n,C_n,Z_n)$ converges.  
Then, if $\proba{\Lambda>\lambda}>0$, then there exists a \emph{giant subtree} at the top of $\tau_n$ 
(one having a positive proportion of vertices) and of height smaller than $\lambda_1<\lambda/2$ 
(all of this with non-vanishing probability). 
Consider then the effect of applying the 213 transformation. 
When the uniform random variable $V$ falls in this giant subtree, with $U$ being its ancestor at distance $\lambda_1$, the transformed subtree will become thick \emph{before} height $2\lambda_1$, 
thin \emph{between} heights $\{2\lambda_1,\ldots, \lambda \}$ and thick \emph{after} height $\lambda$ 
(as in Figure \ref{figLambdaIsZero}). 
In other words, the transformed cumulative population (which has the same law) will have a scaling limit which is increasing before $2\lambda_1$ and after $\lambda$ but that will have a constancy interval in between. 
From the characterization of solutions to \eqref{eqnLampertiTransformIntro}, this represents a contradiction, showing that $\proba{\Lambda=0}=1$.

Regarding the organization of the paper, 
Section \ref{section_DeterministicResultsOnLampertiTransformation} 
contains the deterministic analysis around the Lamperti transformation, 
including the characterization of solutions of Proposition \ref{proposition_deterministicIVP}, 
a result on subsequential (scaling) limits of the discrete Lamperti transformation 
and results on convergence of hitting times of the cumulative population process 
which are useful for the analysis of giant subtrees. 
Then, Section \ref{section_EIProcesses} contains results on exchangeable increment processes  
which give  the convergence of breadth-first walks of Proposition \ref{proposition_ConvergenceOfBFWs}, 
weak subsequential compactness of profiles, 
the compactness criteria for the ICRT of Proposition \ref{proposition_CompactnessOfICRT} 
and invariance under the 312 transformation of Proposition \ref{propoIgualdadDistribucionalIntro}. 
Our main result (Theorem \ref{teoUnicidadEnFnalGralIntro}) is proved in Section \ref{section_AsymptoticThickness}. 
Finally, Section \ref{section_ExampleSection} contains the application of the main theorem to Galton-Watson trees (Corollary \ref{corolarioConvergenciaCGWVarFinIntro}) and the construction of degree sequences which show the general applicability of Theorem \ref{teoUnicidadEnFnalGralIntro}. 

\section{Deterministic results on the Lamperti transformation}
\label{section_DeterministicResultsOnLampertiTransformation}
In this subsection we will collect results on the (deterministic) analysis of solutions to the initial value problem \eqref{deterministicIVP} and its discrete counterpart \eqref{equation_DiscreteLampertiTransform}. 
First of all, the latter can be seen as a discretization of the former: 
as argued in \cite{MR3098685}, 
it corresponds to applying a Euler method of span $1$ 
to the initial value problem. 
Indeed, a discretized initial value problem of the form \eqref{deterministicIVP} can be written as
\[
	c^\sigma(0)=0,\quad
	h^\sigma(t):=D_+ c^\sigma(t)= f\circ c^\sigma(\sigma \floor{t/\sigma}).
\]Note that as for the discrete Lamperti transformation, the (unique) solution can be obtained recursively (see \cite[p. 1594]{MR3098685}). 
Also, note that if $f$ is the breadth-first walk of some plane tree $\tau$ 
(extended by constancy on each $[n,n+1)$ and on $[|\tau|, \infty)$), then $h^1$ is the Lamperti transform of $f$. 
Finally, the effect of scaling can also be incorporated: 
if $f_n$ is the breadth-first walk of a plane tree $\tau_n$ 
and there exist scaling constants $s_n,b_n$ such that $f_n(s_n\cdot)/b_n$ converges on Skorohod space 
to some \cadlag\ function $f$, our method will be based on showing that
\[
	h^1(w_n)(b_n\cdot /s_n) 
	=h^{b_n/s_n}(f_n)
	\to  h^0(f)
\](the Lamperti transform of $f$) whenever the discretization parameter converges to zero. This highlights the role of $b_n/s_n$ as a discretization parameter and explains the hypothesis $s_n/b_n\to\infty$ in Theorem \ref{teoUnicidadEnFnalGralIntro}. 
This method was introduced in \cite{MR3098685}, 
although in the setting there \eqref{deterministicIVP} has a unique solution. 

We first establish Proposition \ref{proposition_deterministicIVP}.
Then we will examine the convergence of hitting times of cummulative population processes, 
which are needed to ensure that our random trees have enough individuals near the top of the tree.

Let us proceed to the proof of Propositon \ref{proposition_deterministicIVP}. 
\subsection{Characterization of solutions to the ODE}
\label{subsectionSolutionsToODE}

Consider then an excursion type function $f$ as in the statement and let $(c^0,h^0)$ be the Lamperti pair of $f$. 
If $\int_{0+}1/f(s)\, ds=\infty$, then, by definition, $i=\infty$ on $(0,\infty)$ and then $c^0=h^0=0$. 
Let $f_-(t)=f(t-)$ be the left-continuous version of $f$ (where $f(0)=0$). 
We will actually prove the characterization result for any solution to the differential inequality
\begin{equation}
\label{eqnDesgFnal}
 	\int_s^t f_-\circ c(r)\, dr \leq \imf{c}{t}-\imf{c}{t}\leq \int_s^t f\circ c(r)\, dr
\end{equation}which is more useful when discussing scaling limits. 
Any solution to the above equation (or to \eqref{deterministicIVP}) is non-decreasing, 
since $f$ and $f_-$ are non-negative. 
Note that, because $\Delta f(t)\geq 0$, 
any solution to \eqref{deterministicIVP} actually solves \eqref{eqnDesgFnal}. 
Assume there exists a non-zero solution $\tilde c$ to \eqref{eqnDesgFnal} and define
\[
		\lambda=\inf\{t>0:\tilde c(t)>0 \}<\infty.
\]Then, for $t>0$
\[
		\int^{\lambda+t}_\lambda f_-\circ c(r)\,dr\leq c(\lambda+t)-c(\lambda)\leq \int^{\lambda+t}_\lambda f\circ c(r)\,dr.
\]The function $c_\lambda(\cdot)=\tilde c(\lambda+\cdot)$ satisfies \eqref{eqnDesgFnal} and is positive on $(0,\infty)$. 
Recall that $c^\lambda$ is the shift of the Lamperti transform $c^0$ by $\lambda$. 
We now show that $c_\lambda= c^0$, 
which proves that $c=c^\lambda$ and that $c^\lambda$ is a solution to \eqref{eqnDesgFnal} and therefore to \ref{deterministicIVP}.

	To ease notation, we write $\tilde c$ instead of $c_\lambda$. 
	Let $a=\inf\set{t>0:\tilde c(t)=1}$. 
	Then $\tilde c$ is constant on $(a,\infty)$ because $f$ is absorbed at zero at time $1$. 
	By construction, $\tilde c$ is positive on the interval $(0,\infty)$ and, since $f>0$ on $(0,1)$, 
	$\tilde c$ is strictly increasing on $(0,a)$ 
	Hence, $f_-\circ \tilde c=f\circ \tilde c$ except at a countable number of points on $(0,a)$, 
	so that we have equalities in \eqref{eqnDesgFnal}.
	Hence, $\tilde c$ is also a solution to \eqref{deterministicIVP} which is strictly increasing on $(0,a)$. 
	
	Let $\tilde\iota$ be the inverse of $\tilde c$ on $[0,1)$. 
	On $[0,1)$, $\tilde\iota$ is increasing, continuous, and with values on $[0,\infty)$. 
	Let $0<r<1$. 
	From the definition of the IVP, 
	$D_+ \tilde c(\tilde\iota(r))=f\circ \tilde c(\tilde\iota(r))=f(r)>0$. 
	Hence, by the formula for the derivative of an inverse function and the fundamental theorem of calculus 
\[
		\infty>\tilde\iota(t)-\tilde\iota(s)=\int_s^tD_+\tilde\iota(r)\,dr=\int_s^t\frac{dr}{f(r)}
		\quad \text{ for }0<s<t<1.
\]Because $\tilde c$ is continuous at 0, then $\tilde\iota(s)\to 0$ as $s\downarrow 0$. 
	Therefore, we obtain
	\begin{equation}\label{eqni_tIsFinite}
	\int_0^t\frac{dr}{f(r)}=\tilde\iota(t)<\infty
	\quad\text{ for all } t\in [0,1).
	\end{equation}We conclude that the Lamperti transform $c^0$ is not zero 
	and that its inverse, $i$, equals $\tilde\iota$ on $[0,1)$. 
	It remains to prove that $a=i(1)$. 
	But this is clear from \eqref{eqni_tIsFinite}, because $\lim_{t\uparrow 1} \tilde\iota=\lim_{t\uparrow 1} i$.
	This implies that $\tilde c= c^0$ and that, therefore, $c^0$ is a solution to \ref{deterministicIVP}. 
	We also deduce that the existence of a non-zero solution to \ref{deterministicIVP} implies finitude of $i$, 
	which also says that $0$ is the only solution if $i(0+)=\infty$. 
	
	Finally, applying the above paragraph to $c$, we have shown that $c$ is strictly increasing on $(0,i(1-))$ 
	with inverse $i$. 
	Hence, $c$ reaches $1$ in finite time if and only if $i(1-)<\infty$. 
	Otherwise, $c$ is strictly increasing, never reaching the value $1$. 
	
	Recall from the proof that $\lambda=\inf\set{t\geq 0: \tilde c(t)>0}$ 
	when $\int_{0+}1/f(s)\, ds<\infty$ and $\tilde c$ is a non-trivial solution to \eqref{deterministicIVP}. 
	However, 
	we can define $\lambda$ even when $\tilde c$ is a trivial solution or $\int_{0+}1/f(s)\, ds=\infty$. 
	With the usual convention $\inf\emptyset=\infty$, we see that $\lambda=\infty$ in this case, 
	but that the equality $\tilde c=c^\lambda$ holds in both cases. 
	
	The preceding argument is based on the analysis of continuity properties of the Lamperti transformation of \cite[\S 4.1]{MR3098685}, where it is noted that the zero sets of $f$ are responsible for the non-uniqueness of the ODE $c'=f\circ c$. We will also use the fact from the latter paper that if $f>0$ on $[0,1)$ and $f=0$ on $[1,\infty)$, then the ODE has a unique solution $c$ which is strictly increasing until it reaches the value $1$. 
	
\subsection{The composition operation and convergence of profiles}
\label{subsection_SubsequentialScalingLimits}

In this subsection, we note that, whenever we have the joint convergence of the (scaled) breadth-first walk and the cumulative profile (on the product Skorohod space), we can already deduce convergence together with the profile. 
The result follows from continuity considerations regarding the composition operation on Skorohod space studied in \cite{MR561155}, \cite{MR1876437} and \cite{MR2479479}. 

\begin{lemma}\label{lemmaConvPopulationProfileDeterministic}
	Let $f_n,f$  be non-negative \cadlag\ functions on $\re_+$ and such that $\imf{f_n,f}{0}\geq 0$, and $f_n,f=0$ on $[1,\infty)$. 
	Also, let $c_n,c$ be non-decreasing continuous functions such that
	\begin{enumerate}
	\item there exist $\lambda_n\in [0,\infty)$ and $\lambda\in[0,\infty]$ such that $c_n$ and $c$ equal $0$  on $[0,\lambda_n]$ and $[0,\lambda]$, 
	\item there exist $\mu_n\in (\lambda_n,\infty)$ and $\mu\in (\lambda,\infty)$ such that $c_n$ and $c$ equal $1$  on $[\mu_n,\infty)$ 
	and  $[\mu,\infty)$, and
	\item $c_n$ and $c$ are strictly increasing on $[\lambda_n,\mu_n]$ and $[\lambda, \mu]$. 
\end{enumerate}If $c_n\to c$ uniformly on compact sets then, $f^n\circ c^n\to f\circ c$ on Skorohod space. 
\end{lemma}
\begin{remark}
	Lemma \ref{lemmaConvPopulationProfileDeterministic} could be proved as in Theorem 3 of \cite{MR3098685}, using Theorem 1.2 of \cite{MR2479479}. Indeed the conditions of the latter hold since $f$ is continuous at $0$ and $1$, which are the only possible discontinuities of 
	the right-continuous inverse of $c$. 
	For completeness, we include the following proof, adapted from the proof of Lemma 3 of \cite{kersting2011height}.
\end{remark}
\begin{proof}
	First suppose that $\lambda<\infty$ and let $i$ be the right-continuous inverse of $c$. 
	Let $0<u<v<r$ be continuity points of $f\circ c$. 
	Using Lemma 2.2 of \cite{MR561155}, convergence of $f^n\circ c^n$ on $[0,r]$ follows from its convergence on each of the subintervals $[0,u]$, $[u,v]$ and $[v,r]$. 
	
	The simplest is the middle interval whenever $0<i(\epsilon)=u<v= i(1-\epsilon)$ for some $\epsilon\in (0,1/2)$ where $\epsilon$ and $1-\epsilon$ are continuity points of $f$. 
	Indeed, the hypotheses imply that for $n$ large enough, $c^n$ and $c$ are strictly increasing on $[u,v]$. 
	If $f$ is any \cadlag\ function on $[0,v]$ and $u<v$, let $\|f\|_u=\sup_{s\leq u}\imf{f}{s}$ and $\|f\|_{[u,v]}=\sup_{u\leq s\leq v}\imf{f}{s}$
	If $\|f^n\circ \alpha^n-f\|_{[\epsilon,1-\epsilon]}\to 0$ where $(\alpha^n,n\in\na)$ are increasing homeomorphisms on $[\epsilon,1-\epsilon]$ converging uniformly to the identity, we can define $\beta^n=i\circ \alpha^n\circ c^n$ to obtain
	\[
		\|i\circ \alpha^n\circ c^n-\id \|_{[u,v]}\to 0
	\]because $i\circ \alpha^n\circ c^n\to i\circ c=\id$ on $[u,v]$.  
	Also, 
	\[
		\|f^n\circ c^n-f\circ c\circ \beta^n\|_{[u,v]}=\|f^n\circ c^n-f\circ \alpha^n\circ c^n\|_{[u,v]}=\|f^n -f\circ \alpha^n\|_{[c^n(u),c^n(v)]}
	\]The right-hand side goes to zero because $c^n(u)\to \epsilon$, $c^n(v)\to 1-\epsilon$, and both limits are continuity points of $f$.

	Now, we choose $v$ and $r$ such that $||f^n\circ c^n-f\circ c||_{[v,r]}$ is as small as we want. 
	Let $\tau^\downarrow(\epsilon)=\inf\{t:||f||_{[t,1]}<\epsilon\}$.
	For $\epsilon \in (0,1/2)$, set $\tilde v=\tau^{\downarrow}(\epsilon)$, $v=i(\tau^{\downarrow}(\epsilon))$ and choose any $r>v$. 
	Since $\|f\|_{[\tilde v,1]}\leq \epsilon$ then $\|f^n\|_{[\tilde v,1]}\leq 2\epsilon$ for $n$ large enough. 
	Then 
	\[
		\|f\circ c\|_{[v,r]}=\|f\|_{[c(v),c(r)]}\leq \|f\|_{[\tilde v,1]}\leq \epsilon\ \ \ \ \mbox{and}\ \ \ \ 		\|f^n\circ c^n\|_{[v,r]}\leq \|f^n\|_{[c^n(v),1]}.
	\]Also, since the interval $[c^n(v),1]$ converges to $[\tilde v,1]$ and $f$ is continuous at $\tilde v$ (by the lack of negative jumps), then $\|f^n\|_{[c^n(v),1]}\leq 3\epsilon$ for large enough $n$. 
	Hence $||f^n\circ c^n- f\circ c||_{[v,r]}\leq 4\epsilon$ for large enough $n$. 
	
	A similar argument proves that $||f^n\circ c^n-f\circ c||_u$ can be made smaller than $\epsilon$ for large enough $n$ by choosing $u$  adequately. 
	The latter works also when $\lambda= \infty$ (in which case $f\circ c^\infty\equiv 0$), but for any $u>0$.
\end{proof}

\subsection{Hitting times of the cumulative population process}
\label{subsection_HittingTimesOfCumulativePopulation}

From Proposition \ref{proposition_deterministicIVP}, 
recall that cumulative populations are continuous and strictly increasing except when taking the values zero or one. 
We will now see that the hitting times of points $x\in (0,1)$ are continuous at them. 

\begin{proposition}
\label{proposition_HittingTimesOfCumulativePopulation}
Let $\fun{c}{[0,\infty)}{[0,1]}$ be continuous, non-decreasing 
and strictly increasing at $t$ if $0<\imf{c}{t}<1$. 
Assume that $c(0)=0$ and $c(\infty)=1$. 
Fix $x\in (0,1)$ and define $h_x(c)=\inf\set{t\geq 0: \imf{c}{t}=x}$. 
If $c^n\to c$ on Skorohod space, then $h_x(c^n)\to h_x(c)$. 
\end{proposition}

\begin{proof}
Note that $h_c(x)$ is finite for any $x\in (0,1)$ as $c$ is continuous, $c(0)=0<x<1=c(\infty)$. 
Since $c$ is non-decreasing, we see that $h_c(x)=\inf\set{t\geq 0: c(t)\geq x}$ and $c(h_c(x))=x$. 
Since, by hypothesis, $c$ is strictly increasing at $h_c(x)$, we see that $h_c(x)=\inf\set{t\geq 0: c(t)>x}$. 
Hence, we can apply Lemma 8 in \cite{2019arXiv191009501M}, to see that $h_x$ is continuous at $c$ on Skorohod space. 
\end{proof}

\section{First results on exchangeable increment processes}
\label{section_EIProcesses}
The aim of this section is to establish the fundamental relationship 
between trees with a given degree sequence 
and exchangeable increment processes. 
With this, we will establish Proposition \ref{proposition_ConvergenceOfBFWs} 
and to prove the following preliminary results on exchangeable increment processes: 
subsequential limits of the cumulative profile of the trees in Theorem \ref{teoUnicidadEnFnalGralIntro} 
and the path transformation of Proposition \ref{propoIgualdadDistribucionalIntro}). 
A subsection is devoted to each of these topics. 

We will mainly deal with discrete time exchangeable increment processes. 
\begin{definition}
	Fix $s\in\na$. 
	A discrete time process $(W^b(j),0\leq j\leq s)$ 
	with increments $\Delta W^b(i)=W^b(i)-W^b(i-1)$ 
	has exchangeable increments (EI) if for every permutation $\sigma$ on $[s]$
	\[
		\paren{\Delta W^b(1),\ldots, \Delta W^b(s)}\stackrel{d}{=}\paren{\Delta W^b(\sigma_1),\ldots,\Delta W^b(\sigma_{s})}.
	\]
\end{definition}

Among discrete time EI processes, extremal ones play a fundamental role. 
Extremal EI processes are constructed in terms of a deterministic sequence $d_1,\ldots, d_s$ and a uniform permutation $\pi$ of $[s]$ by setting
\[
	\Delta W^b_i=d_{\pi_i}
\]In general, if $\sigma$ is a deterministic permutation of $[s]$ and $x^\sigma_i=d_{\sigma_1}+\cdots+d_{\sigma_i}$, we see that $W^b=x^\pi$. This expression separates the deterministic and random components $\set{x^\sigma}$ and $\pi$. 

Write $\theta_i(W^b)=W^{b,(i)}$ for the \emph{cyclic shift} of $W^b$ at $i$, that is, the sequence of length $s$ whose $j$th increment is $\Delta W^b(i+j)$ with $i+j$ interpreted mod $s$. 
A path transformation, introduced by Vervaat in \cite{MR515820}, is used to code discrete random trees from EI processes. The \emph{discrete Vervaat transform} of $W^b$, denoted by $V(W^b)$, is the $\rho$-th cyclic shift of $W^b$, where $\rho=\min\{i\in [{\bf s}]:W^b(i)=\min_{j\in [{\bf s}]}W(j) \}$ is the index of the first minimum of $W^b$.

The next proposition is an easy consequence of the definitions (cf. the proof of Lemma 7 in \cite{MR3188597}). 
It provides a construction of trees with a given degree sequence in terms of a random permutation. 

\begin{proposition}\label{propoBFWIsAEI}
Let ${\bf s}$ be a degree sequence with child sequence $c$ and size $s$. 
Let $W^b$ be an independent increment process constructed by uniformly permuting $c-1$ and let $\rho$ be the index of its first minimum. 
If $W$ is the breadth-first walk (or the depth-first walk) of a uniform tree with degree sequence ${\bf s}$ and $U$ is a uniform random variable on $[s]$ independent of $W$ then $(W,U)\stackrel{d}{=} (V(W^b),\rho)$.
\end{proposition}
Note, in particular, that the index $\rho$ of the first minimum of $W^b$ is uniform and independent of $V(W^b)$.
Also, both the DFW and the BFW have the same increments, but in a different order. They have, however, the same distribution. 

\subsection{Convergence of breadth-first walks}
\label{subsection_ConvergengeOfBFWs}

This subsection is devoted to the proof of Proposition \ref{proposition_ConvergenceOfBFWs}. 
Consider a sequence of degree sequences $({\bf s_n})_{n\geq 1}=(N^n_i,i\geq 0)_{n\geq 1}$
and define $s_n=|{\bf s_n}|$. 
Assume that the sequence $({\bf s_n})_{n\geq 1}$ satisfies the assumptions of Proposition \ref{proposition_ConvergenceOfBFWs}. 
Using a uniform permutation $\pi_n$ on $[s_n]$ and a child sequence $(d^n(j),j\in [s_n])$, construct the process
\[
	W^b_{n}(t)=\frac{1}{b_n}+\frac{1}{b_n}\sum_{j\leq \floor{s_n t}}(d^n(\pi_n(j))-1), \ \ t\in [0,1],
\]where $b$ stands for \emph{bridge}. 
It has exchangeable increments on $\set{k/s_n: 0\leq k\leq s_n}$. 

First, we analyze the convergence of $W^b_{n}$ to the EI process $X^b$ given in \eqref{eqnEIP}.
Define $\xi_{ j}=(d^n(\pi_n(j))-1)/b_{s_n}$. 
	By Theorem 2.2 of \cite{MR0394842}, we know that
	\[
		W^b_n
		\stackrel{d}{\to} X^b \text{ if and only if }	\paren{\sum \xi_{j},\sum \xi^2_{j},\xi_{1},\xi_{2},\ldots}\stackrel{d}{\to} 	\paren{\alpha,\sigma^2+\sum \beta^2_j,\beta_1,\beta_2,\ldots}. 
	\]

	Since $\sum_i \xi_i=-1/b_{n}$, 
	by definition of a degree sequence, 
	we see that the limit $\alpha$ exists and equals $0$. 
	For the sum of the squares, using the definition of a child sequence as well as hypothesis \defin{degree variance}, we get
	\[
		\sum \xi^2_{j}
		=\frac{1}{b_{n}^2 } \sum
		(j-1)^2N^n_j\to \sigma^2+\sum\beta_j^2.
	\]Finally, hypothesis \defin{hubs} gives us the convergence of $\tilde{\xi}_{j}\to \beta_j$.

\begin{remark}\label{rmrkUnknownOrderOfVariance}
	We can substitute the hypotheses of Proposition \ref{proposition_ConvergenceOfBFWs} to 
	$s_n\to\infty$, $b_n=\sqrt{s_n}$, $\max\set{i: N^n_i>0}=o (b_{n})$ and $\sum(j-1)^2N^n_j/b_{n}^2\to \sigma^2\in (0,\infty)$.
	In this case, we obtain in the limit the Brownian bridge on $[0,1]$, as in \cite{MR3188597}.
\end{remark}

We now prove convergence of the scaled breadth-first walks $X_n$. 
Note that $X_n$ has the same law as the Vervaat transform of $W_n^b$. 
By the proof of Lemma 6 of \cite{MR1825153}, 
the process $X^b$ achieves its infimum in a unique time and continuously. 
The latter reference deals only with EI processes with no-negative jumps; 
the same conclusion is found in Theorem 2 of \cite{2019arXiv190304745A} for more general EI processes. 
We can therefore define the Vervaat transform of $X^b$ and note that this process is strictly positive on $(0,1)$. 
Hence, by Lemma 3 of \cite{MR1825153} (or Lemma 14 of \cite{kersting2011height})
	the Vervaat transform $X^n$ of the rescaled bridges $W_{n}^b$ converges to $X$, the Vervaat transform of $X^b$. This finishes the proof of Proposition \ref{proposition_ConvergenceOfBFWs}. 

\subsection{Subsequential compactness of profiles}
\label{subsection_SequentialCompactnessOfProfiles}

In this subsection, we give a simple argument to show that the rescaled cumulative populations $(C^n)$ (extended to $[0,\infty)$ by linear interpolation) admit subsequential (weak) limits. 
Since each $C^n$ is a continuous function, we will use the weak sequential compactness (also called tightness) criteria of \cite[Ch. 2\S 7]{MR1700749}. 
This is a step in the proof of Theorem \ref{teoUnicidadEnFnalGralIntro}, whose proof then amounts to showing uniqueness of subsequential weak limits.

\begin{proposition}\label{propoConvergenciaSubsucPerfil}
	Under the assumptions of Theorem \ref{teoUnicidadEnFnalGralIntro}, 
	let $(X^n, n\geq 1 )$ be the rescaled breadth-first walks of trees $(\tau_n,n\geq 1)$ with law $\p_{{\bf s_n}}$. 
	Let $X$ be the limit of $(X^n,n\in\na)$. 
	Consider the rescaled cumulative profile $C^n$ of $\tau_n$. 
	Then, $(C^n,n\geq 1)$ is tight, 
	and every weakly subsequential limit of $((X^n,C^n),n\geq 1)$ is of the form $\paren{X,C}$ 
	where $C$ satisfies $C_t=\int_0^t X\circ C_s\, ds $. 
\end{proposition}
\begin{proof}
	We prove tightness of $(C^n,n\in \na)$, which together with the tightness of $(X^n,n\in\na)$, 
	implies tightness of $((X^n,C^n),n\in\na)$. 
	Recall that $0\leq C^n\leq 1$, so the sequence $(C^n,n\in\na)$ is uniformly bounded.	
	Note that, for $a_n=s_n/b_n$, we have
	\[
		0\leq  D_+C^n(s)= X^n\circ C^n(\floor{a_n s}/a_n)\leq \|X^n\|:=\sup_{s\in [0,1]}\abs{X^n_s}
	\]and so the modulus of continuity\[
		\imf{\omega_n}{\delta}=\sup\set{|C^n(t)-C^n(s)|: |t-s|\leq \delta}
	\]of $C^n$ satisfies $\imf{\omega_n}{\delta}\leq  \|X^n\|\delta$. 
	Therefore
	\[
		\proba{\imf{\omega_n}{\delta}>\epsilon}\leq 
		\proba{\|X^n\|> \epsilon/\delta}. 
	\]Using Theorem 13.2 in \cite{MR1700749},
	\[
		\lim_{\delta\to 0}\sup_{n} \proba{\|X^n\|> \epsilon/\delta}.
	\]Hence $((X^n,C^n),n\in\na)$ is tight by Theorem II.7.2 of \cite{MR1700749}. 
	
	Suppose that $(X,D)$ is the limit of $((X^{n_l}, C^{n_l}),l\in\na)$. 
	By Skorohod's theorem (see for example \cite[Ch 1.\S 6]{MR1700749}), we assume the convergence takes place almost surely. 
	Suppose we have proved that for any $T>0$
	\begin{equation}\label{eqnDesigualdadesX_yX}
	X_-\circ D\leq \liminf_l\,X^{n_l}\circ C^{n_l}\leq \limsup_l\,X^{n_l}\circ C^{n_l}\leq X\circ D \ \ \ \ \ \mbox{on}\ [0,T].
	\end{equation}Then, using Fatou's lemma, for $0\leq s<t\leq T$
	\begin{equation}\label{eqnContinuidadC}
	\int_s^t X_-\circ D(r)\,dr\leq D(t)-D(s)\leq \int_s^t X\circ D(r)\,dr,
	\end{equation}and therefore, by the proof of Proposition \ref{proposition_deterministicIVP}, we see that $D=C^\Lambda$ for some $\Lambda\in [0,\infty]$. 
	
	It remains to prove \eqref{eqnDesigualdadesX_yX}. 
	We start with the equality
	\[
		X^{n_l}\circ C^{n_l}=\paren{X^{n_l}\circ C^{n_l}-X\circ \alpha^{n_l}\circ C^{n_l}}+X\circ \alpha^{n_l}\circ C^{n_l}.
	\]The difference in parenthesis is bounded above by $||X^{n_l}-X\circ \alpha^{n_l}||$, which goes to zero.
	The convergence of $\alpha^{n_l}\circ C^{n_l}\to D$ follows by adding the terms $\pm C^{n_l}$:
	\begin{equation}\label{eqnLambdanCn}
	\sup_{u\leq v}|\alpha^{n_l}\circ C^{n_l}(\floor{a_n u}/a_n)-D(u)|\leq ||\alpha^{n_l}-\id||+\sup_{u\leq v}|C^{n_l}(\floor{a_n u}/a_n)-D(u)|\to 0.
	\end{equation}Then, because $X$ is \cadlag\ and has only positive jumps
	\[
		X_-\circ D\leq \liminf_l\,X\circ \alpha^{n_l}\circ C^{n_l}\leq  \limsup_l\,X\circ \alpha^{n_l}\circ C^{n_l}\leq X\circ D.\qedhere
	\]
\end{proof}

As in Subsection \ref{subsectionSolutionsToODE}, 
we note that since the process $X^a=X_{a+\cdot}$ is positive until absorbed at zero at time $1-a$, 
there exists a unique solution $C^a$ to the ODE $D_+ C=X^a\circ C$. 
Because of uniqueness, the above tightness result actually proves weak convergence. 
Note that for $0<x<1$, if $\Lambda_x=\inf\set{t\geq 0: C^a_t=x}$ then
\[
	\Lambda_{b-a}(C^a)=\int_a^b\frac{1}{X_s}\, ds; 
\]hence, even if $\Lambda_a(X^l)\to \infty$ along a subsequence, 
Proposition \ref{proposition_HittingTimesOfCumulativePopulation} tells us that $\Lambda_b(X^l)-\Lambda_a(X^l)$ converges to the finite quantity $\Lambda_{b-a}(C^a)$. 
This will be relevant to the proof of Theorem \ref{teoUnicidadEnFnalGralIntro}. 

\subsection{The 213 transformation}
\label{subsection_213Transformation}

This subsection is devoted to the proof of invariance under the 213 transformation stated as Proposition \ref{propoIgualdadDistribucionalIntro}. 
The aforementioned transformation of a tree is easy to formalize using the associated depth-first walk. 

Let $s\in \na$ and $x_i\in \na\cup \{-1\}$ for $i\in[s]$ such that $\sum_1^{s}x_i=-1$.
Consider the discrete time excursion $w$ 
constructed by applying the Vervaat transformation to the partial sums of $x$, 
starting at zero and non-negative up to time $s$, where $w_{s}=-1$. 
Let $\tau$ be the tree whose depth-first walk is $w$. 
We now define the \emph{213 transformation} of $\tau$, 
denoted $\Psi_{u,v}(\tau)$. 
Let $\Delta w_j:=w_j-w_{j-1}$ be the $j$th increment of $w$ (number of children of $j$). 
Choose a vertex $v\in\{2,\ldots, s\}$ in the tree which is neither the root nor a leaf, and let
\[
	d_v=\inf\{j\in \{1,\ldots, s-v+1\}:w_{v-1+j}-w_{v-1}=-1 \}
\]be the length of the excursion starting at $v$ (this is equal to  $|T(v)|+1$). 
Notice that a vertex $v$ is a leaf iff $d_v=1$. 
Consider any vertex $v\in\{v+1,\ldots, v-1+d_v \}$, implying that $v$ is its ancestor. 

The tree $\Psi_{u,v}(\tau)$ is the plane tree whose depth-first walk is $\Psi_{u,v}(w)$ defined as follows:
\begin{displaymath}
\Delta\Psi_{u,v}(w)_j=\left\{
\begin{array}{lllll}
\textcolor{blue}{\Delta w_{v-1+j}} & \textrm{if } 1\leq j \leq u-v\\
\Delta w_{j-(u-v)} & \textrm{if } u-v+1\leq j \leq u-1\\
\textcolor{red}{\Delta w_{u+j-u}} & \textrm{if } u\leq j\leq u-1+d_u\\
\Delta w_{v+d_v+j-(u+d_u)} & \textrm{if } u+d_u\leq j \leq s+u+d_u-v-d_v\\
\textcolor{blue}{\Delta w_{u+d_u+j-(s+u+d_u-v-d_v+1)}} & \textrm{if } s+u+d_u-v-d_v+1\leq j \leq s.
\end{array} \right.
\end{displaymath}Figure \ref{figTransformacion2143} shows the 213 transformation of a depth-first walk.

\begin{center}
	\begin{figure}
		\includegraphics[width=.47\textwidth]{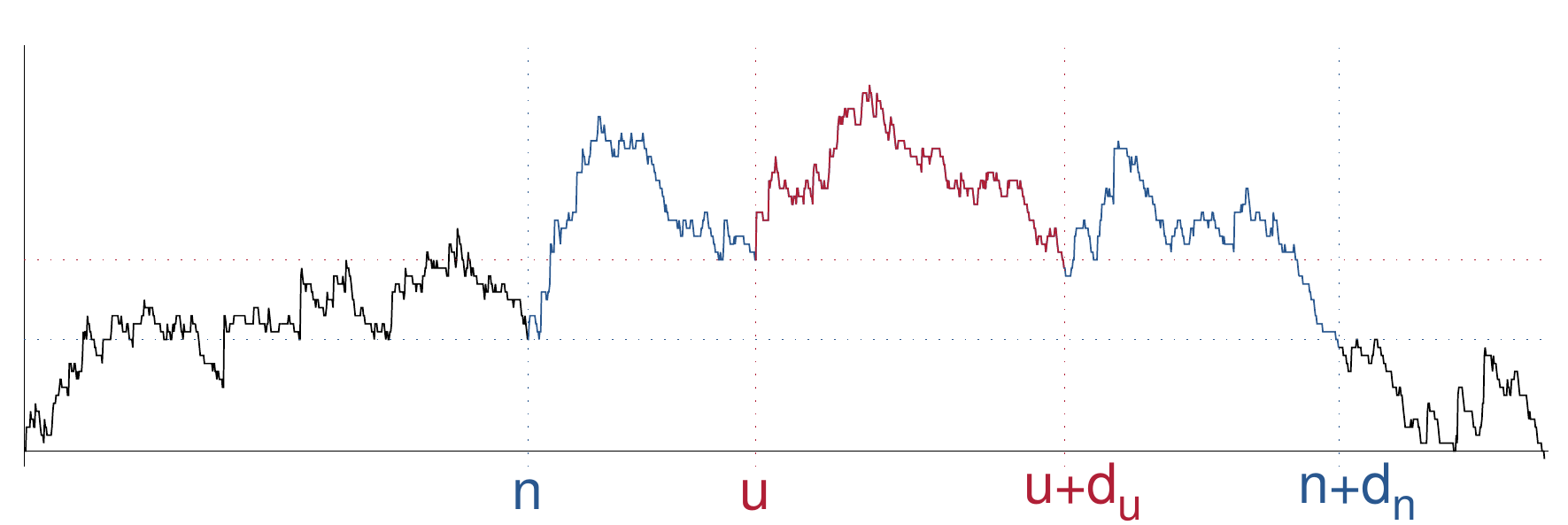}
		\hfill
		\includegraphics[width=.47\textwidth]{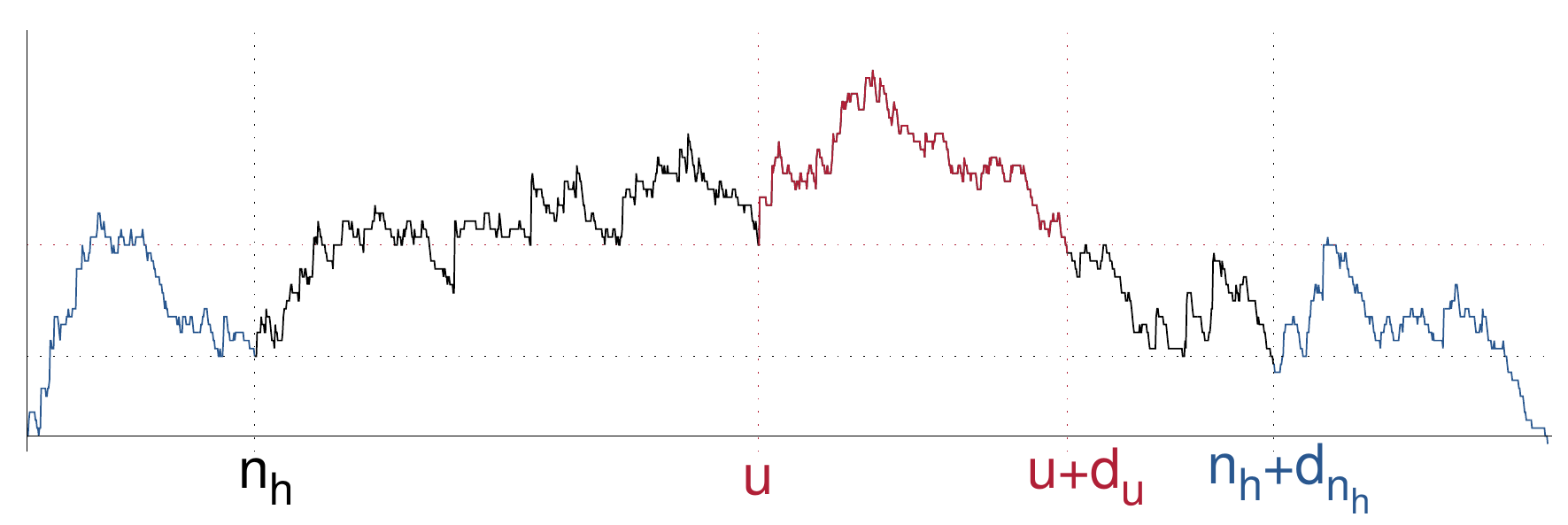}
		\caption{\footnotesize Recall the notation of the subtrees obtained after cutting at $n$ and $u$ the tree $T_n$, as in Definition \ref{defi213Transformation}. In the top excursion, the subtree $T_n(1,n)$ is represented in black, $T_n(n,u)$ in blue and $T_n(u)$ in red. The bottom excursion is $T_n$ after applying the 213 transformation. The root of the grafted subtree $T_n(1,n)$ will have label $n_h=u-n+1$ in $\tilde{T}_n$. }
		\label{figTransformacion2143}
	\end{figure}
\end{center}

Using a uniform law on the vertex $u$ and choosing $v$ in an adequate way, 
we prove the invariance of the law $\p_{{\bf s}}$ under a 213-type transformation.

Define the space of all excursions coding a tree with given degree sequence ${\bf s}$ by
\[
	\mathcal{E}_{{\bf s}}=\{w:w\mbox{ is an excursion with degree sequence } {\bf s} \}.
\]

Given an excursion $w\in \mathcal{E}_{{\bf s}}$, the natural numbers $u\in[s]\setminus \{1\}$ and $h\geq 1$, we construct $\tilde{w}$ as follows. 
In the tree generated by $w$, if $|u|$ denotes the height of the vertex $u$ and 
\begin{equation}\label{eqnCondiciones2143}
|u|>h,
\end{equation}let $v$ be the ancestor of $u$ at distance $h$ from $u$, 
and define $\tilde{w}=\Psi_{u,v}(w)$. If condition \eqref{eqnCondiciones2143} is not satisfied, define $\tilde{w}=w$.

\begin{lemma}\label{lemmaTransf2143EsBiyeccionn}
	For fixed $h\geq 1$ and $u\in\{2,\ldots, s \}$, the transformation $\Phi_{h,u}:{\mathcal{E}_{{\bf s}}}\to{\mathcal{E}_{{\bf s}}}$  sending $w$ to $\tilde w$ 
	is bijective. 
\end{lemma}
\begin{proof}
	It is easy to prove that $u$ satisfies \eqref{eqnCondiciones2143} for $w$ iff $u$ satisfies \eqref{eqnCondiciones2143} for $\tilde{w}$. 
	Being a transformation of a finite set to itself, 
	it suffices to prove it is onto.
	Consider $\tilde{w}\in \mathcal{E}_{{\bf s}}$.
	If $u$ does not satisfy \eqref{eqnCondiciones2143} for $\tilde{w}$, define $w:=\tilde{w}$.
	If $u$ satisfies \eqref{eqnCondiciones2143} for $\tilde{w}$, choose the ancestor $v_h=u-n+1$ of $u$ at height $h$ from $u$. 
	Define in this case $w=\Psi_{u,u-v+1}(\tilde{w})$. 
	In other words, 
	\[
		\tilde{w}=\Psi_{u,v}(\Psi_{u,u-v+1}(\tilde{w})).\qedhere
	\]
\end{proof}

We now prove the aforementioned equality in distribution.

\begin{proof}[Proof of Proposition \ref{propoIgualdadDistribucionalIntro}]
	Consider any excursion $w\in \mathcal{E}_{{\bf s}}$ and $u\in \{2,\ldots,s \}$.
	Using the bijection $\Phi_{h,u}$ of Lemma \ref{lemmaTransf2143EsBiyeccionn} and the independence between the tree $\Theta\sim \p_{{\bf s}}$ and $U$
	\[
		\proba{\tilde{W}=w,U=u} = \proba{W=\Phi_{h,u}^{-1}(w),U=u}
		= \proba{W=w,U=u},
	\]using that $W$ is the (uniform) excursion of the tree $\Theta$. 
\end{proof}

The 213 transformation was naturally defined on trees and has a simple definition as a path-transformation of its depth-first walks. This is possible since the subtree rooted above a vertex is encoded by a contiguous excursion of its depth-first walk. The transformation can also be seen in terms of the breadth-first walk although it does not have such a simple visualization. 
%

\section{Asymptotic thickness and uniqueness of subsequential limits}
\label{section_AsymptoticThickness}

In this section, we prove the main technical result 
responsible for the validity of Theorem \ref{teoUnicidadEnFnalGralIntro} and therefore finish its proof. 
Namely, that our random trees have an asymptotically thick base. 
We assume the hypotheses of Theorem \ref{teoUnicidadEnFnalGralIntro}. 
Since, by Proposition \ref{propoConvergenciaSubsucPerfil} 
and Lemma \ref{lemmaConvPopulationProfileDeterministic}, 
the sequence of (scaled) breadth-first walk, profile and cumulative profile satisfy is tight and 
all its subsequential limits $(X,Z,C)$ 
satisfy $D_+C=Z=X\circ C$, 
Proposition \ref{proposition_deterministicIVP} gives us the existence 
of an initial inverval of constancy of size $\Lambda=\inf\set{t\geq 0: C_t>0}$, 
such that $(Z,C)=(Z^\Lambda,C^\Lambda)$. 
When $\int_{0+}1/X_s\, ds=\infty$, there is a unique trivial solution to the ODE, 
which implies that all subsequential limits agree and that therefore there is convergence in this case. 
Hence, we now focus on the case $\int_{0+}1/X_s\, ds<\infty$. 
Having an asymptotically thick base 
is the assertion that $\Lambda=0$, in which case $(Z,C)$ is the Lamperti pair of $X$. 
Hence, subsequential limits are unique and this establishes the weak convergence asserted in Theorem \ref{teoUnicidadEnFnalGralIntro}. More specifically, the objective of this section is to prove the following result. 

\begin{theorem}
\label{teoLambdaIsZeroOrInfinity}
	Assume the hypotheses of Theorem \ref{teoUnicidadEnFnalGralIntro} are satisfied and that $\int_{0+}1/X_s\, ds<\infty$. 
	Let $C^\Lambda$ be any subsequential limit of $(C^n,n\in\na)$, and $I^\Lambda(\cdot)=\Lambda+\int_0^\cdot ds/X(s)$ its right-continuous inverse.
	Then 
	\begin{equation}\label{eqnLambda0HasProbabilityZero1}
	\proba{\Lambda\in (0,\infty]
	}=0. 
	\end{equation}
\end{theorem}

The reader can imagine the event on the left-hand side as ocurring with a subsequential limit of trees having a thin base (looking like a cord) with (scaled) height approximately $\Lambda$; after that, a giant subtree (with size proportional to the size of the tree) starts to grow. 
Attached to such cord, there can be other cords with (possibly) giant subtrees growing after height $\Lambda$ 
(as in Figure \ref{figLambdaIsZero}).

\begin{figure}
		\includegraphics[height=.4\textheight]{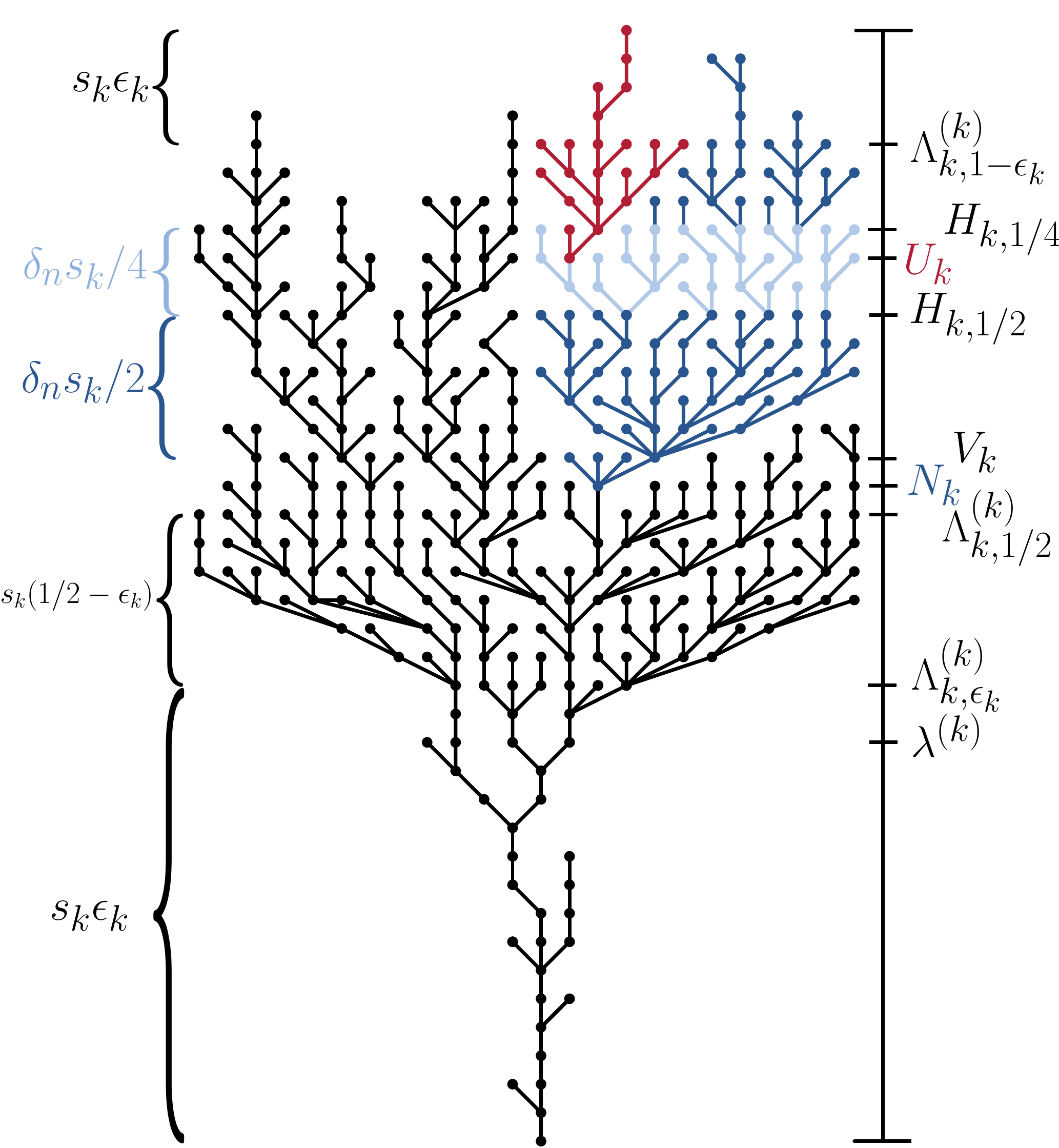}
		\vphantom{\includegraphics[height=.4\textheight]{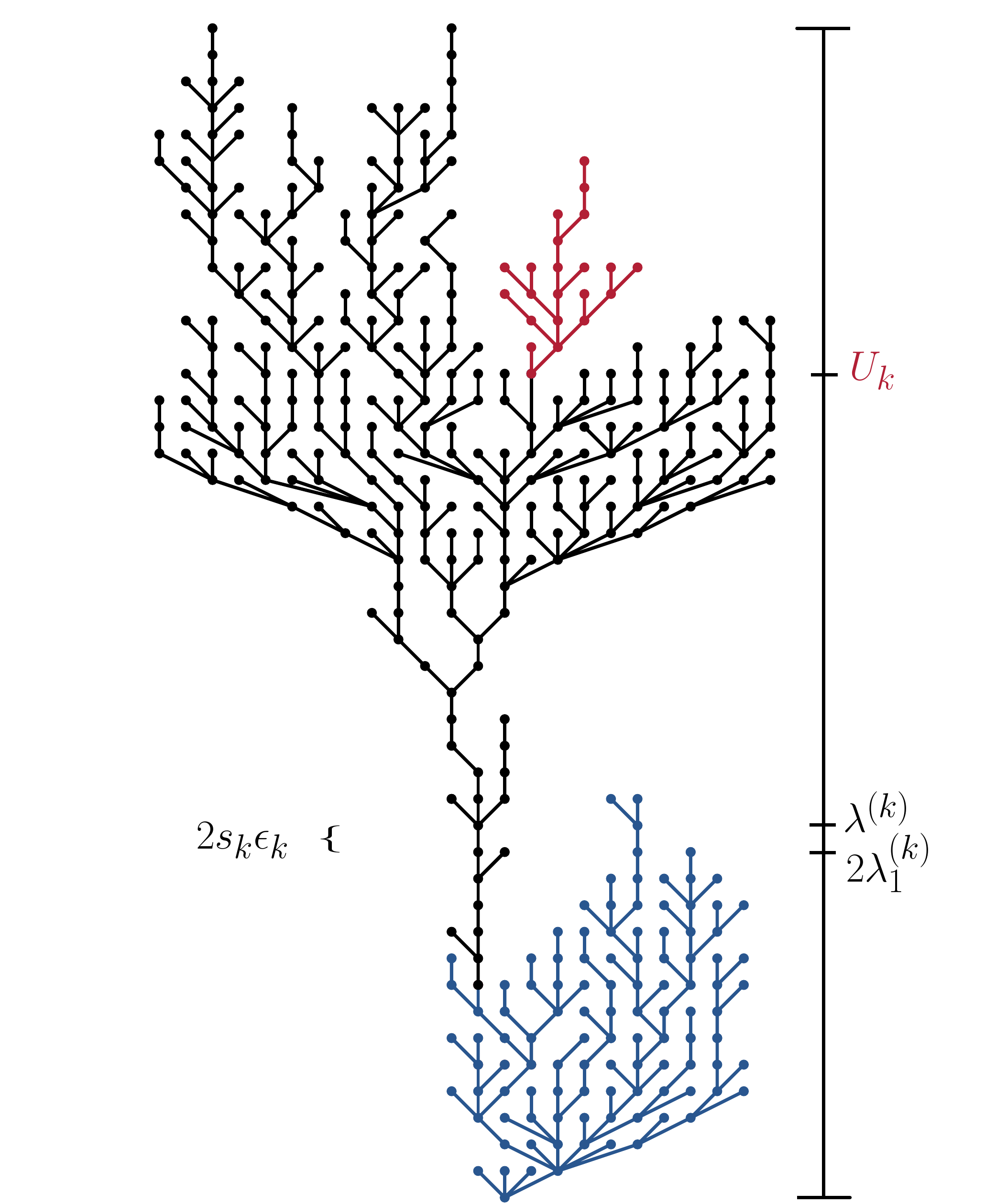}}
		\quad\quad\quad\quad
		\includegraphics[height=.4\textheight]{figLambdaIsZeroV2Transformed.pdf}
	\caption{\footnotesize On the left, the original tree. The giant subtree has root $v_k$, and is at distance $\lambda_1^{(k)}$ of height $\Lambda_{k,1-\epsilon_k}^{(k)}$. The uniform variable $U_k$ is in such subtree, between heights $h_{k,1/2}$ and $h_{k,1/4}$. The ancestor $N_k$ of $U_k$ at distance $\lambda_1^{(k)}$, is an ancestor of $v_k$.
		After applying the 213 transformation (figure on the right), the new tree is thick before height $(2\lambda_1)^{(k)}$, thin between heights $\{(2\lambda_1)^{(k)},\ldots, \lambda^{(k)} \}$, and again thick after $\lambda^{(k)}$.}
	\label{figLambdaIsZero}
\end{figure}

To prove Theorem \ref{teoLambdaIsZeroOrInfinity}, we use the path transformation of Proposition \ref{propoIgualdadDistribucionalIntro}. 
We will use this transformation for individuals \emph{near the top of the tree}. 
To give a formal definition of what \emph{near the top} means, 
we need a lemma regarding the convergence of the heights of the first and last individuals. 

Using Skorohod's theorem, we work on the space where the almost sure convergence $(X^{n_l},C^{n_l}, Z^{n_l})\to (X,C^\Lambda,Z^\Lambda)$ happens, for some deterministic subsequence $(n_l,l\geq 1)$. 
Write only $(X^l,C^l,Z^l)\to (X,C^\Lambda,Z^{\Lambda})$ to avoid cumbersome notation.
Also, note that, on this same probability space, we can let $\Theta_l$ be the only tree whose breadth-first walk is (the unscaled version of) $X^l$. 
We introduce the hitting times of $(\epsilon,\infty)$ by $C^l$. 

\begin{definition}\label{defiLambdaEps}
	For $\epsilon\in (0,1)$ and $l\in \na$, define the first height that the rescaled cumulative profile has more than $\epsilon$ individuals as
\[
		\Lambda_{l,\epsilon}:=\inf\{t>0:C^l(t)>\epsilon \}. 
\]
\end{definition}Since $C^\Lambda$ is a shift by $\Lambda$ of $C^0$, 
the definition of the Lamperti transformation tells us that
\[
	\inf\{t>0: C^\Lambda(t)>\epsilon \}=I^\Lambda(\epsilon)=\Lambda+I(\epsilon)
 \]for any $\epsilon\in (0,1)$. 
However, since $C^\Lambda$ is strictly increasing except when at $0$ or $1$, 
we can apply Proposition \ref{proposition_HittingTimesOfCumulativePopulation} to obtain the first part of the following lemma. 
\begin{lemma}\label{lemmaLambdaLEpsToLambdaEps}
	For every $\epsilon\in (0,1)$, the convergence $\Lambda_{l,\epsilon}\to \Lambda+I(\epsilon)$ holds almost surely as $l\to \infty$. Furthermore, for every $0<a<b<1$, $\Lambda_{l,b}-\Lambda_{l,a}\to I(b)-I(a)$. 
	Consider any sequence $\epsilon_{k}\downarrow 0$. 
	Then, there exists a deterministic subsequence $l_k$ 
	such that
\[
	(\Lambda_{l_k,\eps_k},\Lambda_{l_k,1-\eps_k}-\Lambda_{l_k,1/2})	
	\to_{l\to\infty}  (\Lambda,I_1-I_{1/2})
	\]almost surely. 
\end{lemma}

Indeed, from the proof of this lemma, we see that the previous convergence takes place together with $X^k$ and $C^k$. 
Also, we remark that the choice of 1/2 is arbitrary, and the lemma works for any $a\in (0,1)$.

\begin{proof}
	For the second part of the statement, recall our remark at the end of Subsection \ref{subsection_SequentialCompactnessOfProfiles} %
	and that 
 $I(b)-I(a):=\int_a^bds/X(s)<\infty$ for $0<a<b<1$. 
Thus, even if $\Lambda_{l,a}\to \infty$, the sequence
$\Lambda_{l,b}-\Lambda_{l,a}$ has a finite limit. 

	Note that $\Lambda+I(\epsilon)\to \Lambda$ as $\epsilon\downarrow 0$, 
	since either both sides are finite or infinite. 
	Also, $I_{1-\eps_k}-I_{1/2}\to I_1-I_{1/2}$ as $k\to\infty$. 
	Let $\delta_k$ be a summable sequence decreasing to zero and consider a distance $d$ which generates the topology of  the one-point compactification of $[0,\infty)$. 
	From the (almost sure) convergences 
	$\Lambda_{l,\eps_k}\to \Lambda_{\eps}$ and $\Lambda_{l,1-\eps_k}-\Lambda_{l,1/2}\to I_{1-\eps_k}-I_{1/2}$ 
	as $l\to\infty$, 
	there exists $l_k$ such that for $l\geq l_k$
	\[
		\proba{\imf{d}{\Lambda_{l,\eps_k}, \Lambda_{\eps_k}}\text{ or } \imf{d}{\Lambda_{l,1-\eps_k}-\Lambda_{l,1/2}, I_{1-\eps_k}-I_{1/2}}>\delta_k}<\delta_k. 
	\]Without loss of generality, we assume that $l_k$ increases, 
	and deduce from the Borel-Canteli Lemma that $\Lambda_{l_k,\eps_k}\to \Lambda$ and $\Lambda_{l_k,1-\eps_k}-\Lambda_{l_k,1/2}\to I_1-I_{1/2}$ almost surely. 
\end{proof}

\begin{remark}\label{remarkDefinitionLambdaToTheParenk}
	In the following, we will use the temporal rescaling $a_{k}=s_k/b_{k}$, which comes from the Lamperti transform in Subsection \ref{subsection_SequentialCompactnessOfProfiles}. 
	For ease of notation, we sometimes put the superscript $(k)$ to refer of such rescaling, for example, for $\lambda>0$, we write $\lambda^{(k)}:=\floor{\lambda a_{s_k}}$, and also $\Lambda_{k,\epsilon_{k}}^{(k)}=\floor{\Lambda_{k,\epsilon_{k}}a_{s_k}}$.
\end{remark}

To prove Theorem \ref{teoLambdaIsZeroOrInfinity}, we fix any $\lambda>0 $ and show that
\begin{equation}\label{eqnLambda0AndLambda1FiniteHasProbabilityZero}
\proba{\Lambda\in (\lambda,\infty]}=0
\end{equation}in the following manner.
Recall the 213 transformation 
from Proposition \ref{propoIgualdadDistribucionalIntro}. 
Let  $\delta_n\downarrow 0$ and $\lambda_1\in (0,\lambda/2)$. 
In Subsection \ref{subsecGiantSubtree} we prove that for $n$ big enough and under the event of Equation \eqref{eqnLambda0AndLambda1FiniteHasProbabilityZero}, \emph{near the top of the tree}, that is, at height $\Lambda^{(k)}_{k,1-\epsilon_k}-\lambda_1^{(k)}$, there is one vertex $v_k$ having at least $\floor{\delta_n s_{k}}$ descendants, for every $k$ big enough. 
The subtree rooted at $v_k$ will be referred to as a \emph{giant subtree}. 
It is a subtree with a positive proportion of individuals, almost all of them between its first $\lambda_1^{(k)}$ generations.
Using such a giant subtree, 
we apply the 213 transformation to obtain a tree which is thick near the root, thin in the middle and thick after that. 
This contradicts the fact that solutions to $D_+C=X\circ C$ are strictly increasing unless taking the values $0$ or $1$. 
The steps in the argument are:
\begin{enumerate}\label{enumerateStepsToMoveGiantSubtree}
	\item Show that one can intersect with the set $\Lambda_{k,1-\epsilon_k}-\lambda_1>\Lambda_{k,1/2}$. 
	\item Consider two heights $H_{k,1/2}<H_{k,1/4}$ where the giant subtree is thick: 
	\emph{up to} $H_{k,1/2}$ there is at least half of the size of the giant subtree, \emph{between } $H_{k,1/2}$ and $H_{k,1/4}$ there is at least quarter of the size of the giant subtree. 
	\item The probability that $U_k$ (the uniform variable used for the 213 transformation) be in the giant subtree, \emph{between} heights $H_{k,1/2}$ and $H_{k,1/4}$, is at least $\delta_n/4$. Intersect with this set. 
	\item Apply the 213 transformation cutting at $U_k$ and $N_k$, the latter being the ancestor of $U_k$ at distance $\lambda_1^{(k)}$ from $U_k$. 
	Note that $N_k$ is an ancestor of $v_k$. 
	\item\label{itemPropertiesOfTransformedTree} On such events, the transformed tree $\tilde\Theta_k$ has 
	\begin{enumerate}
		\item at least $\delta_n s_k/2$ individuals up to height $(2\lambda_1)^{(k)}$, 
		\item less than $2\epsilon_ks_k$ individuals between heights $\{(2\lambda_1)^{(k)},\ldots, \lambda^{(k)} \}$,
		\item at least $s_k/2-\epsilon_ks_k$ individuals after $\lambda^{(k)}$.
	\end{enumerate} 
	\item The transformed tree has the same distribution as the original, 
	but its cumulative Lamperti transform is constant between two intervals where it increases. 
	This has probability zero.
\end{enumerate}

We now carry out this program. 
Since $I(1)-I(1/2)\in (0,\infty)$ almost surely,  note that it is enough to establish
\begin{equation}\tag{A}\label{eqnLambda_0IsZeroOrInfinity}
\proba{\Lambda\in (\lambda, \infty), I(1)-I(1/2)\in (\lambda_2,\lambda_2')}=0,
\end{equation}where $0<\lambda<\infty$ and $0<\lambda_2<\lambda_2'<\infty$ are continuity points for the distributions of $\Lambda$ and  $I(1)-I(1/2)$, respectively. 
We decompose \ref{eqnLambda_0IsZeroOrInfinity}{\color{DarkOrchid}} using the set where there is a giant subtree at the top of the tree. 
For a plane tree $\tau_k$ with $s_k$ vertices, let $u\in [s_k]$ be any vertex on $\tau_k$.
We denote by $\tau_k(u)$ the \emph{subtree rooted at $u$} in the tree $\tau_k$ 
(a tree with root $u$  and all its descendants). 
For $v\in [0,1]$ with $\floor{vs_k}\in [s_k]$, we also refer to $v$ as a vertex in the tree (and write $\tau_k(v)$ instead of $\tau_k(\floor{v s_k})$).
Let $\delta_m\downarrow 0$. 
For any $m,k\in \na$ and $\lambda_1>0$ consider
\[
	A(m,k)=\{\mbox{there exists a vertex $V_k$ at height $\Lambda_{k,1-\epsilon_k}^{(k)}-\lambda^{(k)}_1$ such that}  \,|\Theta_k(V_k)|> \delta_m s_k \},
\]where we defined
$\lambda^{(k)}:=\floor{\lambda a_{k}}$ in Remark \ref{remarkDefinitionLambdaToTheParenk}. 
Recall that Lemma \ref{lemmaLambdaLEpsToLambdaEps} implies $\Lambda_{k,\epsilon_k}\to \Lambda$, and $\Lambda_{k,1-\epsilon_k}-\Lambda_{k,1/2}\to I(1)-I(1/2)$ almost surely.
Then, we bound \ref{eqnLambda_0IsZeroOrInfinity}{\color{DarkOrchid}} as
\begin{equation}\tag{B}\label{eqnIntersectingAnk}
\begin{split}
& \proba{\Lambda\in (\lambda,\infty],I(1)-I(1/2)\in (\lambda_2,\lambda_2')}\\
&\leq \proba{\lim_k\,\{\Lambda_{k,\epsilon_k}>\lambda, \Lambda_{k,1-\epsilon_k}-\Lambda_{k,1/2}\in (\lambda_2,\lambda_2')\}}\\
& \leq \varlimsup_m \varlimsup_k\, \p\left(\left\{\Lambda_{k,\epsilon_k}>\lambda ,\Lambda_{k,1-\epsilon_k}-\Lambda_{k,1/2}\in (\lambda_2,\lambda_2')\right\}\cap A(m,k) \right)\\
& + \varlimsup_m \varlimsup_k\p\left(\left\{I(1)-I(1/2)\in (\lambda_2,\lambda_2')\right\}\cap A(m,k)^c \right).
\end{split}
\end{equation}(The last equality follows since
\[
	\proba{
	\set{\Lambda_{k,1-\eps_k}-\Lambda_{k,1/2}\in (\lambda_2,\lambda'_2)}
	\Delta \set{I(1)-I(1/2)\in (\lambda_2,\lambda_2')}
	}
	\to_{k\to\infty} 0. )
\]
The summands on the right-side will be denoted by \ref{eqnIntersectingAnk}{\color{DarkOrchid}1} and \ref{eqnIntersectingAnk}{\color{DarkOrchid}2}, respectively.
First we prove that under \ref{eqnIntersectingAnk}{\color{DarkOrchid}1}, 
a giant subtree is moved to the base of the tree by applying the 213 transformation, 
causing a zero probability event. 
After that, we prove \ref{eqnIntersectingAnk}{\color{DarkOrchid}2} is zero, 
which is interpreted as the existence of a (vanishingly small) giant component.

\subsection{\ref{eqnIntersectingAnk}{\color{DarkOrchid}1} equals zero: moving a giant subtree to the base of the tree}
Now we formalize the steps given in page \pageref{enumerateStepsToMoveGiantSubtree} and Figure \ref{figLambdaIsZero}, proving that \ref{eqnIntersectingAnk}{\color{DarkOrchid}1} can be bounded for $n$ big enough and $2\lambda_1<\lambda$ by 
\[
\proba{C^\Lambda(2\lambda_1)\geq \delta_n/2,C^\Lambda(\lambda)-C^\Lambda(2\lambda_1)=0,1-C^\Lambda(\lambda)\geq 1/2},
\]which is zero since $C^\Lambda$ cannot be constant inside an interval where it is strictly increasing.  
First we need some definitions. 

\begin{definition}
	Let $\tau$ be a tree and $v$ a non-leaf vertex in $\tau$. 
	For any $h_1,h_2\in \na\cup\{\infty\}$ with $|v|\leq h_1\leq h_2$, let
	\[
		|\tau(v)|(h_1,h_2)=\{\mbox{number of individuals in $\tau(v)$ having height $h\in \{h_1,\ldots, h_2\}$ in the tree $\tau$ }\}.
	\]Let $|\tau(v)|(h_1):=|\tau(v)|(h(v),h_1)$, the number of individuals in $\tau(v)$ up to height $h_1$ in the tree $\tau$.
Define the first height where the subtree $\tau(v)$ has at least half of its size:
\[
	h_{\tau(v),1/2}=\min\{h\geq h(v):|\tau(v)|(h)\geq |\tau(v)|/2 \},
\]and the first height $h$ after $h_{\tau(v),1/2}$, where the subtree $\tau(v)$ accumulates a quarter of its size between $\{h_{\tau(v),1/2}+1,\ldots, h \}$:
\[
	h_{\tau(v),1/4}=\min\{h\geq h_{\tau(v),1/2}+1:|\tau(v)|(h_{\tau(v),1/2}+1,h)\geq |\tau(v)|/4 \}.
\]
\end{definition}
We now prove the our random tree cannot have a thin base if there is a giant subtree. 

\begin{proposition}
	Let $2\lambda_1<\lambda_2\wedge \lambda$ and $\lambda_2<\lambda_2'$. 
	Then for every $m$ big enough
\begin{linenomath}
\begin{align*}
&  \varlimsup_k\, \p\left(\Lambda_{k,\epsilon_k}>\lambda ,\Lambda_{k,1-\epsilon_k}-\Lambda_{k,1/2}\in (\lambda_2,\lambda_2'),\right.\\
& \hspace{2cm} \left.\mbox{exists a vertex $V_k$ at height $\Lambda_{k,1-\epsilon_k}^{(k)}-\lambda_1^{(k)}$ with }|\Theta_k(V_k)|> \delta_m s_k \right)=0.
\end{align*}\end{linenomath}
\end{proposition}
\begin{proof}
Denote the event inside the probability as $B(m,k)$ and refer to Figure \ref{figLambdaIsZero} for visual assistance. 
Consider a uniform random variable $U_k$ on $\{2,\ldots, s_k \}$, independent of $\Theta_k$ (recall that the subsequence $(n(l(k)),k\in \na)$ obtained in Lemma \ref{lemmaLambdaLEpsToLambdaEps} was deterministic). 
We apply the 213 transformation, cutting the tree $\Theta_k$ at vertices $U_k$ and $N_k$ (at distance $\lambda_1^{(k)}$), obtaining the subtrees 1, 2 and 3, and rearranging them as subtrees 2, 1 and 3. 
For the tree $\Theta_k$ and the root of the giant component $V_k$, define the quantities $H_{k,1/2}:=h_{\Theta_k(V_k),1/2}$ and $H_{k,1/4}:=h_{\Theta_k(V_k),1/4}$. 
Let us assume that $H_{k,1/4}<\Lambda_{k,1-\epsilon_k}^{(k)}$ with probability going to one as $k\to \infty$, and postpone the proof to the end. 


	The probability for $U_k$ to fall in $\Theta_k(V_k)$, between heights $\{H_{k,1/2}+1,\ldots,H_{k,1/4}\}$, is at least $\delta_m/4$, by definition of $H_{k,1/4}$. 
	It follows that
	\begin{linenomath}
	\begin{align*}
	& \frac{\delta_m}{4}\varlimsup_k\,\proba{B(m,k)}\\
	& \leq \varlimsup_k\,\proba{B(m,k)\cap \{U_k\in \Theta_k(V_k), h(U_k)\in \{H_{k,1/2}+1,\ldots,H_{k,1/4}\} ,H_{k,1/4}<\Lambda_{k,1-\epsilon_k}^{(k)}\}}.
	\end{align*}\end{linenomath}Denote the event on the right-hand side by $B_1(m,k)$. 
	Note that the ancestor $N_k$ at distance $\lambda_1^{(k)}$ of $U_k$, is also an ancestor of $V_k$. 
	Indeed, we have
	\[
		h(N_k)=h(U_k)-\lambda_1^{(k)}\leq H_{k,1/4}-\lambda_1^{(k)}\leq \Lambda_{k,1-\epsilon_k}^{(k)}-\lambda_1^{(k)}=h(V_k).
	\]
	%
	%
	%
	%
	%
	
	On the set $B_1(m,k)$ the transformed tree $\tilde\Theta_k$ satisfies
	\begin{enumerate}
		\item There are at least $\delta_m s_k/2$ individuals up to height $\lambda_1^{(k)}$. 
		This holds true since,  
		before height $\lambda_1^{(k)} $, 
		$\tilde\Theta_k$ 
		contains the band of the subtree $\Theta_k(V_k)$ between heights  $\{h(V_k),\ldots, H_{k,1/2}\}$  and thus, contains at least $\delta_m s_k/2$ individuals. 
		Indeed
		\[
			H_{k,1/2}-h(N_k) \leq  h(U_k)-h(N_k)=\lambda_1^{(k)}
			\qquad \mbox{and}\qquad
			|\Theta_k(N_k)|(H_{k,1/2})\geq |\Theta_k(V_k)|(H_{k,1/2})\geq \delta_m s_k/2.
		\]
		\item Between heights $\{(2\lambda_1)^{(k)},\ldots, \lambda^{(k)}\}$ there are at most $2\epsilon_ks_k$ individuals. 
		To prove this, note that the transformation lifts the height of the subtree 1. 
		Hence, the maximum number of individuals that $\tilde\Theta_k$ can have in the band between heights $\{(2\lambda_1)^{(k)},\ldots, \lambda^{(k)}\}$, 
		is the sum of the cardinality of such a band of subtree 2 and the band between $\{1,\ldots, \lambda^{(k)}\}$ of subtree 1. 
		The latter has at most $\epsilon_ks_k$ individuals since $\Lambda_{k,1-\epsilon_k}>\lambda$, and the former has at most $\epsilon_k s_k$ individuals since it takes at most $(2\lambda_1^{(k)})$ generations to reach height $\Lambda_{k,1-\epsilon_k}^{(k)}$. 
		Indeed
	\[
	\Lambda_{k,1-\epsilon_k}^{(k)}-h(N_k)\leq h(U_k)+\lambda_1^{(k)}-h(N_k)=(2\lambda_1)^{(k)}
	\]and\[
	|\Theta_k(N_k)|(\Lambda_{k,1-\epsilon_k}^{(k)},\infty)\leq |\Theta_k|(\Lambda_{k,1-\epsilon_k}^{(k)},\infty)\leq \epsilon_ks_k.
	\]
		\item After height $\lambda^{(k)}$ there are at least $s_k/2-\epsilon_ks_k$ individuals.
		This holds since $\tilde\Theta_k$ after height $\lambda^{(k)}$, contains the individuals of subtree 1 between heights $\{\lambda^{(k)},\ldots, h(N_k) \}$. 
		Hence
		\[
			h(N_k)/a_{s_k}\geq \Lambda_{k,1-\epsilon_k}-2\lambda_1 \geq \Lambda_{k,1/2}+\lambda_2-2\lambda_1\geq \Lambda_{k,1/2},
		\]and $\lambda_1^{(k)}\leq \Lambda_{k,\epsilon_k}^{(k)}$ implies
		\[
			C^k(h(N_k)/a_{s_k})-C^k(\lambda)\geq 1/2-\epsilon_k.
		\]
	\end{enumerate}Using the equality in distribution $\Theta_k\stackrel{d}{=}\tilde\Theta_k$ of Proposition \ref{propoIgualdadDistribucionalIntro}, we bound
	\begin{linenomath}
	\begin{align*}
	\frac{\delta_m}{4}\varlimsup_k\,\proba{B_1(m,k)}&\leq \varlimsup_k\,\proba{\tilde{C}^k(2\lambda_1)\geq \delta_m/2,\tilde{C}^k(\lambda)-\tilde{C}^k(2\lambda_1)\leq 2\epsilon_k,1-\tilde{C}^k(\lambda)\geq 1/2-\epsilon_k }\\
	&=\varlimsup_k\,\proba{C^k(2\lambda_1)\geq \delta_m/2,C^k(\lambda)-C^k(2\lambda_1)\leq 2\epsilon_k,1-C^k(\lambda)\geq 1/2-\epsilon_k }\\
	& \leq \proba{C^\Lambda(2\lambda_1)\geq \delta_m/2,C^\Lambda(\lambda)-C^\Lambda(2\lambda_1)=0,1-C^\Lambda(\lambda)\geq 1/2},
	\end{align*}\end{linenomath}which is zero, since $C^\Lambda$ cannot be constant inside an interval where it is strictly increasing.
	This implies $\varlimsup_k\,\proba{B(m,k)}=0$. 
	
	To conclude, we prove
	\[
		\lim_{k\to\infty}\proba{B(m,k)\cap \set{H_{k,1/4}<\Lambda_{k,1-\epsilon_k}^{(k)}} }=1. 
	\]
	Since $\Theta_k$ has less than $\epsilon_k s_k$ individuals after height $\Lambda_{k,1-\epsilon_k}^{(k)}$, the same holds for $\Theta_k(V_k)$ on the event $\{H_{k,1/4}\geq \Lambda_{k,1-\epsilon_k}^{(k)} \}$, and thus
	\[
		s_k\epsilon_k\geq |\Theta_k(V_k)|(H_{k,1/4},\infty)\geq |\Theta_k(V_k)|/4-z_k(V_k,1/2)\geq s_k\delta_m/4-\tilde Z^{k}(H_{k,1/2}).
	\]Therefore, on the mentioned event, from the rescaling of $Z^{k}$, we obtain
	\[
		\delta_m/4-\epsilon_k<(b_{k}/s_k)||Z^k||\leq (b_{k}/s_k)||X^k||,
	\]where the last inequality follows by the time-change relating $Z^k$ and $X^k$. 
	Hence, we bound
	\[
	\varlimsup_k\proba{B(m,k)\cap\{H_{k,1/4}\geq \Lambda_{k,1-\epsilon_k}^{(k)} \}}
	\leq \proba{\varlimsup_k\{\delta_m /4-\epsilon_k\leq (b_{k}/s_k)2||X^k|| \}}
	=0
	\]where the computation of the limit follows since $||X^k||\to ||X||$ and $b_k/s_k\to 0$, 
	by hypothesis. 
\end{proof}

\subsection{\ref{eqnIntersectingAnk}{\color{DarkOrchid}2} equals zero: existence of a vanishingly small giant subtree}
\label{subsecGiantSubtree}

For arbitrary $\epsilon\in (0,1/2)$, we split \ref{eqnIntersectingAnk}{\color{DarkOrchid}2} as
\begin{equation}\tag{C}\label{eqnIntersectingCkSmallerThanOneMinusEpsilon}
\begin{split}
&  \varlimsup_m \varlimsup_k\p\left(\left\{I(1)-I(1/2)\in (\lambda_2,\lambda_2')\right\}\cap A(m,k)^c \right)\\
& \leq  \varlimsup_m \varlimsup_k\p\left(\left\{C^k(\Lambda_{k,1-\epsilon_k}-\lambda_1)<1-\epsilon\right\}\cap A(m,k)^c \right)\\
& + \varlimsup_k\p\left(\left\{ I(1)-I(1/2)\in [\lambda_2,\lambda_2']\right\}\cap   \left\{ C^k(\Lambda_{k,1-\epsilon_k}-\lambda_1)\geq 1-\epsilon\right\} \right),
\end{split}
\end{equation}
and denote by \ref{eqnIntersectingCkSmallerThanOneMinusEpsilon}{\color{DarkOrchid}} the first term on the right.
Similarly as in \ref{eqnLambda_0IsZeroOrInfinity}, we shall prove that \ref{eqnIntersectingCkSmallerThanOneMinusEpsilon}{\color{DarkOrchid}} is zero for fixed $\epsilon$.
Thus
\begin{equation}\tag{D}\label{eqnIntersectingCkSmallerThanOneMinusEpsilon1'}
\begin{split}
& \varlimsup_m \varlimsup_k\p\left(\left\{I(1)-I(1/2)\in [\lambda_2,\lambda_2']\right\}\cap A(m,k)^c \right)\\
& \leq \p\left(\left\{I(1)-I(1/2)\in [\lambda_2,\lambda_2']\right\}\cap \varlimsup_{\epsilon\downarrow 0}\varlimsup_k \left\{ C^k(\Lambda_{k,1-\epsilon_k}-\lambda_1)\geq 1-\epsilon \right\} \right).
\end{split}
\end{equation}
Denote the right-hand side by \ref{eqnIntersectingCkSmallerThanOneMinusEpsilon1'}{\color{DarkOrchid}}.

\textit{\ref{eqnIntersectingCkSmallerThanOneMinusEpsilon1'}{\color{DarkOrchid}} equals zero: 
the last $\lambda_1$ generations have many individuals.}

Let $a\in (0,1/2)$ and consider $0<\lambda_1<\lambda_2$. 
We will prove, 
using Lemma \ref{lemmaLambdaLEpsToLambdaEps}, 
that, even when $\Lambda=\infty$, 
the limit of the (shifted) cumulative profile $C^k(\Lambda_{k,1-\epsilon_k})-C^k(\Lambda_{k,1-\epsilon_k}-\lambda_1)$ will be positive. 
If the limit is called $L$, 
by adding  the terms $\pm C^k(\Lambda_{k,1-\epsilon_k}-\lambda_1)$  to $C^k(\Lambda_{k,1-\epsilon_k}-1/a_{s_k})$
we get
\[
	\varlimsup_{\epsilon\downarrow 0}\varlimsup_k \left\{ 1-\epsilon_k\geq 1-\epsilon+C^k(\Lambda_{k,1-\epsilon_k}-1/a_{s_k})-C^k(\Lambda_{k,1-\epsilon_k}-\lambda_1)\right\} \subset \left\{1\geq 1+L \right\}. 
\]We now prove that $L$ exists and that $\proba{L>0}=1$, so that \textcolor{DarkOrchid}{D} equals zero. 

Consider any fixed $a\in (0,1/2)$ (thus, a continuity point of $X$). 
Let $\Lambda^k=\floor{\Lambda_{k,a\cdot a_{k}}}/a_{k}$ and 
\begin{linenomath}
\begin{equation}\label{defiXtolsubaXsuba}
	X^k_a=(X^k(C^k(\Lambda^k)+v),v\in [0,1-C^k(\Lambda^k)])\ \ \ \ \ \mbox{ and }\ \ \ \ \ X_a=(X(a+v),v\in [0,1-a]).
\end{equation}\end{linenomath}We remarked at the end of Subsection \ref{subsection_SequentialCompactnessOfProfiles}, 
$D_+C_a=X_a\circ C_a$ has a unique solution $C_a$, 
with inverse $I_a$ given by
\[
	I_a(b)=\int_0^b\frac{1}{X_a(s)}\, ds. 
\]Consider also the solution $C^k_a$ to the equation
\[
	C^k_a(t)=\int_0^tX^k_a\circ C^k_a(\floor{ua_{k}}/a_{k})du,
\]which is uniquely obtained by recursion. 
The reader can check that $C^k_a=C^k(\Lambda^k+\cdot)-C^k(\Lambda^k)$. 
From the proof of Lemma \ref{lemmaLambdaLEpsToLambdaEps} we deduce $C^k_a\to C_a$ uniformly on the interval $[0,I(1)-I(1/2)]$. 
From the definition of $C^k$, we get
\begin{linenomath}
\begin{align*}
C^k(\Lambda_{k,1-\epsilon_{k}})-C^k(\Lambda_{k,1-\epsilon_{k}}-\lambda_1) & =C^k_a(\Lambda_{k,1-\epsilon_{k}}-\Lambda^k)-C^k_a(\Lambda_{k,1-\epsilon_{k}}-\Lambda^k-\lambda_1)\\
& \to C_a(I(1)-I(a))-C_a(I(1)-I(a)-\lambda_1).
\end{align*}\end{linenomath}
Note that the above expression is well-defined for $k$ big enough since $a\in (0,1/2)$ and
\[
\Lambda_{k,1-\epsilon_{k}}-\Lambda^k\to I(1)-I(a) \geq I(1)-I(1/2)\geq \lambda_2>\lambda_1,
\]
using Lemma \ref{lemmaLambdaLEpsToLambdaEps} and the definition of \ref{eqnIntersectingCkSmallerThanOneMinusEpsilon1'}{\color{DarkOrchid}}. 
Since $C_a$ is the (strictly increasing) solution IVP($X_a$) (on $[0,I(1)-I(a)]$), then $L>0$.

\textit{\ref{eqnIntersectingCkSmallerThanOneMinusEpsilon}{\color{DarkOrchid}} equals zero: a vertex at height $\lambda_1$ from the top of the tree has many descendants. }

On the event \ref{eqnIntersectingCkSmallerThanOneMinusEpsilon}{\color{DarkOrchid}}, 
we have the bounds $|\Theta_k(v)|<\delta_m s_k$ for the subtree rooted at every vertex $v$ at height $\Lambda_{k,1-\epsilon_k}^{(k)}-\lambda_1^{(k)}$ and $\epsilon <C^k(\infty)-C^k(\Lambda_{k,1-\epsilon_k}-\lambda_1)$. 
We might think of the collection of such subtrees as a forest. 
On \ref{eqnIntersectingCkSmallerThanOneMinusEpsilon}{\color{DarkOrchid}}, 
each of the last $\ceil{\epsilon s_k}$ individuals (of our tree, in breadth-first order) have less than $\delta_m s_k$ descendants. 
If we consider the breadth-first walk of these deterministic quantity of individuals, 
it is given by $\tilde X^k_{\ceil{\epsilon s_k}+\cdot}$ and corresponds to the BFW of a forest with a given degree sequence (albeit a non-extremal one). 
Let us now consider the depth-first walk of this forest $\tilde X^{D,\epsilon, k}$, which has the same law; let $X^{D,\epsilon,k}$ denote its scaled version. 
The size of the subtree rooted at $v$ equals the length of the (sub)excursion of $\tilde X^{D,\epsilon,k}$ above its running minimum:
\[
	L^D_k(v):=\min\{u>v: X^{D,\epsilon,k}_u<X^{D,\epsilon,k}_v \}
	=|\tau^\epsilon_k(v)|/s_k\leq \delta_m
\]for every $v\in [0,\epsilon]$, 
and hence
\begin{linenomath}
\begin{align*}
& \varlimsup_m\varlimsup_k \proba{\left\{ C^k(\Lambda_{k,1-\epsilon_{k}})-C^k(\Lambda_{k,1-\epsilon_{k}-\lambda_1})>\epsilon\right\}\cap A(m,k)^c }
\\ & \leq  \varlimsup_m\varlimsup_k \proba{\left\{\mbox{$L^D_k(v)\leq \delta_m$ for all $v\in [0,\epsilon]$} \right\} }
\\ & =  \varlimsup_m\varlimsup_k \proba{\left\{\mbox{$L_k(v)\leq \delta_m$  for all $v\in [0,\epsilon]$} \right\}}, 
\end{align*}\end{linenomath}where $L_k(v)$ denotes the lengths of excursions of $X^{\epsilon,k}$ above its running minimum. 
Since $X$ is continuous at $1-\epsilon$ almost surely, we know that $X^{\epsilon,k}\to X^{\epsilon}$ in the Skorohod topology, where $X^{\epsilon}_t=X_{1-\epsilon+t}$ (cf. by \cite[Lemma 16.1]{MR1700749}). 
Define the length of a subexcursion of $v$ in $X^{\epsilon}$ by $L(v)$. 
If there exists $v\in [0,\epsilon]$ such that $L(v)>2\delta_m$, then for some $\delta>0$ small and $k$ big enough we can take a subinterval $(g',d')\subset (v,v+L(v))$ such that $d'-g'>\delta_m$ and $X^{D,\epsilon,k}_s\geq X^{D,\epsilon}_{v}+\delta$ for every $s\in (g',d')$. 
This implies 
the existence of $v'$ with $L_k(v')\geq d'-g'>\delta_m$.  
Hence, 
\begin{linenomath}
\begin{align*}
& \varlimsup_m\varlimsup_k \proba{\left\{ C^k(\Lambda_{k,1-\epsilon_{k}})-C^k(\Lambda_{k,1-\epsilon_{k}-\lambda_1})>\epsilon\right\}\cap A(m,k)^c }
\\&\leq  \varlimsup_m \proba{ L(v)\leq \delta_m\text{  for all }v\in [0,\epsilon]}
\\&= \proba{L(v)=0\text{  for all }v\in [0,\epsilon]}. 
\end{align*}\end{linenomath}However, a process with no excursions above its running minimum necessarily coincides with it, 
so that it is of bounded variation. 
Since $X$ has sample paths of unbounded variation, 
we see that $\proba{L(v)=0\text{  for all }v\in [0,\epsilon]}$ which concludes the proof that \ref{eqnIntersectingCkSmallerThanOneMinusEpsilon}{\color{DarkOrchid}}=0.


\section{EI processes and the Lamperti transformation}
\label{section_LampertiForEI}
In this section, we analyze non-triviality and finite-time extinction 
of the Lamperti transformation of the Vervaat transform of EI processes 
stated as Proposition \ref{proposition_CompactnessOfICRT}. 

Let $X^b$ be the EI process given in \eqref{eqnEIP} and let $X$ be its Vervaat transform. 
Note that under any one of the conditions in Proposition \ref{proposition_CompactnessOfICRT}, 
$X^b$ is of unbounded variation so that $X^b$ reaches its minimum uniquely and continuously, 
say at a time $\rho$. (See for example \cite[Thm. 2]{2019arXiv190304745A})
Also, note that the trajectories of $X$ close to zero or one coincide 
with the corresponding the corresponding trajectories of $X^b$ 
after or before it reaches its minimum at $\rho$. 
The proof of Proposition \ref{proposition_CompactnessOfICRT} follows from an analysis of the pre and post-minimum processes of $X^b$, in particular, in the obtention of lower envelopes as in the following proposition. 
Let us write $X$ instead of $X^b$ to lighten notation. 

\begin{proposition}
\label{proposition_EnvelopesForEI}
Let $X$ be an extremal EI process of infinite variation and parameters $(0,\sigma^2,\beta)$, 
where $\beta_i\downarrow 0$. 
Let $\rho$ be the unique instant at which $X$ reaches its minimum. 
Then, there exists $\gamma,\gamma'>1$ such that
\begin{equation}\label{eqnLowerEnvelopEIAtLeftOfInfimum}
\lim_{t\to 0+}\frac{X_{\rho-t}-X_\rho}{t^{1/\gamma}}=\infty,
\end{equation}and 
\begin{equation}\label{eqnLowerEnvelopEIAtRightOfInfimum}
\lim_{t\to 0+}\frac{X_{\rho+t}-X_\rho}{t^{1/\gamma'}}=\infty. 
\end{equation}Otherwise, assume there exists $\alpha\in (1,2)$ such that $\lim_{x\to 0}x^\alpha\overline \beta(x)=\infty$. 
Then, there exists $\gamma>1$ such that \eqref{eqnLowerEnvelopEIAtLeftOfInfimum} holds. 
If furthermore, there exists $\tilde \alpha<  1/(2- \alpha)$ such that $\sum_i \beta_i^{\tilde\alpha}<\infty$, there exists $\gamma'>1$ such that \eqref{eqnLowerEnvelopEIAtRightOfInfimum} holds. 
\end{proposition}
Of course, Proposition \ref{proposition_CompactnessOfICRT} follows from the above result if $\sigma^2=0$. 
To prove the latter, however, we will need a path transformation that extends Lemme 4 in \cite{MR1141246} for L\'evy processes to the EI setting. 
We will consider the future minimum process $\dunderline{X}$, 
given by $\dunderline{X}_s=\inf\set{X_t:t\geq s}$, 
as well as the right-endpoint of the excursion straddling $t$ given by $d_t=\inf\{s\geq t: X_s=\overline X_s\}$.  
Also, note that there are no restrictions on the sign of the parameters $\beta_i$. 

\begin{proposition}
\label{proposition_LemmeFromBertoin}
Let $X$ be an extremal EI process of parameters $(0,\sigma,\beta)$ of unbounded variation. 
Let $\rho$ be the place the unique minimum of $X$ is achieved. 
Then $\dunderline{X}_{\rho+\cdot}-\dunderline{X}_\rho\stackrel{d}{=}\overline X\circ d$. 
\end{proposition}
Via the above proposition and the simple inequality $\overline X\circ d \geq \overline X$, 
we recover some comparison results for the sample function growth of \cite{MR606997}, 
since for any increasing $\fun{f}{\re_+}{\re_+}$: 
\[
	\liminf_{t\to 0+} \frac{X_{\rho+t}-X_\rho}{\imf{f}{t}}
	\geq \liminf_{t\to 0+} \frac{\dunderline{X}_{\rho+t}-\dunderline{X}_\rho}{\imf{f}{t}}
	\geq \liminf_{t\to 0+} \frac{\overline X_t}{\imf{f}{t}}. 
\]
On the other hand, since viewing the process at the left of the infimum is the same as viewing the process $-X^R_t=X_{(1-t)-}$ at the right of its infimum, then
\[
\liminf_{t\to 0+} \frac{X_{\rho-t}-X_\rho}{\imf{f}{t}}
\geq \liminf_{t\to 0+} \frac{\sup_{s\leq t} -X^R_s}{\imf{f}{t}}\stackrel{d}{=}\liminf_{t\to 0+} \frac{-\underline X_t}{\imf{f}{t}}.
\]

The importance in translating the problem from the post-minimum process into one near zero is that we can use couplings between EI processes and the more well-known L\'evy processes, as well as the following result for L\'evy processes. 
\begin{proposition}
\label{proposition_LevyEnvelopes}
Let $X$ be a spectrally positive L\'evy process of infinite variation 
with L\'evy measure $\nu$ supported on a compact subset of $\re_+$. 
If $\sigma>0$, then there exists $\gamma,\gamma'>1$ such that
\begin{equation}\label{eqnLowerEnvelopLpAtLeftOfZero}
\lim_{t\to 0+}\frac{\underline X_t}{t^{1/\gamma}}
=-\infty 
\end{equation}
and
\begin{equation}\label{eqnLowerEnvelopLpAtRightOfZero}
\lim_{t\to 0+}\frac{\overline X_t}{t^{1/\gamma'} } 
=\infty.
\end{equation}Otherwise, if there exists $\alpha>1$ such that $\lim_{x\to 0}x^\alpha\imf{\overline\nu}{x}=\infty$ for $\overline \nu(x)=\imf{\nu}{(x,\infty)}$,
then \eqref{eqnLowerEnvelopLpAtLeftOfZero} holds true. 
If furthermore, there exists $\tilde \alpha<1/(2-\alpha)$ such that
	\[
	\int_0^1 x^{\tilde \alpha} \, \imf{\nu}{dx}<\infty,
	\]
then \eqref{eqnLowerEnvelopLpAtRightOfZero} holds true.
\end{proposition}
Let us apply Propositions \ref{proposition_LemmeFromBertoin} and \ref{proposition_LevyEnvelopes} to prove  Proposition \ref{proposition_EnvelopesForEI}  and then finish by proving the former. 

\begin{proof}[Proof of Proposition \ref{proposition_EnvelopesForEI}] 
Assume first that  all $\beta_i$ have the same sign and that $x^\alpha\overline \beta(x)\to \infty$ for some $\alpha>1$. 

Using that $d_t\geq t$, from Proposition \ref{proposition_LemmeFromBertoin} we have
\begin{equation}\label{eqnInequalityPostInfimumRunningSupremum}
\varliminf_{t\to 0+}\frac{X_{\rho+t}-X_\rho}{f(t)}\geq   \varliminf_{t\to 0+} \frac{\overline{X}_t}{f(t)}. 
\end{equation}From Theorem 3.27 in \cite{MR2161313}, 
we can write $X=Y+Z$, where $Y$ is a L\'evy process with characteristics $(0,0,\sum_i \delta_{\beta_i})$
and $Z$ is an exchangeable increment process 
whose jumps $\tilde \beta$ satisfy $\sum_i  |\tilde\beta_i|^{\gamma}<\infty$ for every $\gamma>1$. 
Thanks to Theorem 2.32 item (i) in the same reference, $\abs{Z_t}/t^{1/\gamma}\to 0$ almost surely for every $\gamma>1$. 
Also, our hypothesis on $\alpha$ implies that the L\'evy measure of $Y$ satisfies the hypothesis of Proposition \ref{proposition_LevyEnvelopes} and that therefore, there exists $\gamma>1$ such that $\lim_{t\to 0}\overline{Y}_t/t^{1/\gamma}=\infty$ for some $\gamma>1$. 
We conclude that
\[
	\liminf_{t\to 0}\frac{\overline X_t}{t^{1/\gamma}}
	=\liminf_{t\to 0}\frac{\overline Y_t}{t^{1/\gamma}}
	=\infty.\qedhere
\]
\end{proof}

We continue the proof of Proposition \ref{proposition_LevyEnvelopes} by
dealing with lower envelopes for the pre and post minimum processes 
of a L\'evy process with no negative jumps. 
The tools require fluctuation theory of L\'evy processes 
as presented in \cite{MR1406564}, \cite{MR2320889} and \cite{MR3155252}. 
It also requires the indices of Blumenthal and Getoor for subordinators. 
Given any subordinator $\tau$ with Laplace exponent $\Phi$, 
we define the upper and lower indices according to Blumenthal and Getoor 
(See Theorem 6.1 in \cite{MR0123362}, there called $\beta$ and $\sigma$)  
as follows
\[
	\overline\alpha(\tau)
	=\inf\set{\alpha>0: \lim_{\lambda\to\infty}\lambda^{-\alpha}\imf{\Phi}{\lambda}=0}
	\dispand
	\underline\alpha(\tau)=\sup\set{\alpha>0: \lim_{\lambda\to\infty}\lambda^{-\alpha} \imf{\Phi}{\lambda}=\infty}.
\]and recall that they satisfy 
$0\leq \underline \alpha(\tau)\leq \overline \alpha(\tau)\leq 1$ and that, whenever its dfirt is zero, 
\[
	\lim_{t\to 0} t^{-1/\alpha}\tau_t=\infty 
	\text{ if }
	\alpha<\underline \alpha(\tau)
	\text{ and }
	\lim_{t\to 0} t^{-1/\alpha}  \tau_t=0
	\text{ if }
	\alpha>\overline \alpha(\tau). 
\]
(cf. \cite[Thms. 3.1, 6.1 and 6.2]{MR0123362}).

\begin{proof}[Proof of Proposition \ref{proposition_LevyEnvelopes}]
Let $X$ be a spectrally positive L\'evy process with parameters $(0,0,\nu)$. 
Note that the conclusions we want to establish do not depend on \emph{large} jumps. 
Hence, we assume that  the L\'evy measure $\nu$ is supported on $[0,1]$. 

Our hypothesis with $\alpha$ implies that $X$ is of infinite variation ($\int_0^1 x\, \imf{\nu}{dx}=\infty$); 
thanks to a celebrated theorem of Rogozin, 
we know that $\limsup_{t\to 0+} X_t/t=\limsup_{t\to 0+} -X_t/t=\infty$ 
and therefore $0$ is regular for both half-lines $(0,\infty)$ and $(-\infty,0)$ 
(cf. \cite{MR0242261} or \cite[Thm. 1]{2019arXiv190304745A}). 

The Laplace exponent $\Psi$ of $X$ 
(which satisfies $\esp{e^{-\lambda X_t}}=e^{t\imf{\Psi}{\lambda}}$)  
 is given by
\[
	\imf{\Psi}{\lambda}=\int_0^1 (e^{-\lambda x}-1+\lambda x)\, \imf{\nu}{dx}.
\]In the case $\sigma>0$, for $\gamma\in (1,2)$ we have $\lambda^{-\gamma}\imf{\Psi}{\lambda}\geq \sigma^2\lambda^{2-\gamma}/2\to \infty$. When $\sigma=0$, we now prove the same convergence for $\gamma\in (1,\alpha)$. 
We will denote by $\Phi$ the right-continuous inverse of $\Psi$ given by $\imf{\Phi}{\lambda}=\inf\set{u\geq 0: \imf{\Psi}{u}>\lambda}$.  $\Phi$ appears naturally when considering the running maximum of $-X$. 
Indeed, if $\hat \tau_t=\inf\set{s: -X_s>t}$, then $\hat \tau_t$ is a subordinator with Lapace exponent $\Phi$. 
From our assumption $x^\alpha\overline\nu(x)\to \infty$ as $x\to 0$. 
Given $\gamma\in (1,\alpha)$ and $k>0$, 
choose $\eps>0$ such that $\imf{\overline \nu}{y}\geq k y^{-\gamma}$ if $y\in (0,\eps)$.
For $x>0$ define\[
	\imf{\barbare{\nu}}{x}=\int_x^1 \imf{\overline \nu}{x}\, dx. 
\]Then, for any $x\in (0,\eps)$, 
\[
	\imf{\barbare{\nu}}{x}
	=\int_x^1 \imf{\overline \nu}{x}\, dx
	\geq \int_x^\eps \imf{\overline \nu}{x}\, dx
	\geq\frac{k}{\gamma-1}\bra{x^{-(\gamma-1)}-\eps^{-(\gamma-1)}}. 
\]Use Fubini's theorem to write the Laplace exponent $\Psi$ of $X$ as follows
\[
	\imf{\Psi}{\lambda}
	=\int_0^1 \bra{e^{-\lambda x}-1+\lambda x}\, \imf{\nu}{dx}
	=\lambda \esp{\imf{\barbare{\nu}}{T/\lambda}},
\]where $T$ is a standard exponential random variable independent of $X$. 
Then, 
\[
	\lambda^{-\gamma} \imf{\Psi}{\lambda}
	\geq \lambda^{-(\gamma-1)} \esp{\imf{\barbare{\nu}}{T/\lambda}\indi{T/\lambda\leq \eps}}
	\geq \lambda^{-(\gamma-1)} \frac{k}{\gamma-1}\esp{\bra{\paren{\lambda/T}^{\gamma-1}-\eps^{-(\gamma-1)}}\indi{T/\lambda\leq \eps}}
\]Considering the inferior limit as $\lambda\to\infty$, the second summand disappears, and we get: 
\[
	\liminf_{\lambda\to\infty} \lambda^{-\gamma} \imf{\Psi}{\lambda}
	\geq \frac{k}{\gamma-1}\imf{\Gamma}{2-\gamma}.
\]Since the above is valid for any $k>0$, joining the two cases we deduce that
\begin{equation}
\label{eqn_lowerindexH}
	\lim_{\lambda\to\infty} \lambda^{-\gamma} \imf{\Psi}{\lambda}=\infty \text{ for any }\gamma\in (1,\alpha\wedge 2) . 
\end{equation}It therefore follows that 
\[
	\lim_{\lambda\to\infty}\lambda^{-1/\gamma} \imf{\Phi}{\lambda}=0. 
\]for any $\gamma\in (1,\alpha)\subset (1,2)$. 
In terms of indices, we see that $\imf{\overline \alpha}{\hat\tau}\leq 1/\alpha<1$. 
Now, under our assumptions, the drift of $\hat\tau$ is zero. 
Indeed, if $\tilde \Phi$ is the Laplace exponent of any subordinator, the corresponding drift is given by $\lim_{\lambda\to\infty} \imf{\tilde \Phi}{\lambda}/\lambda$. 
Note that the hypotheses $\underline \alpha(X)>1$ or $\sigma>0$ implies that 
\[
\frac{1}{\lambda}\imf{\Psi}{\lambda}
=\frac{\sigma^2}{2}\lambda+\int_0^1 \bra{1-e^{-\lambda x}} \imf{\overline\nu}{x}\, dx
\to \infty\qquad \mbox{ as }\lambda\to \infty.
\]
Hence, the drift of $\hat\tau$ equals
\[
\lim_{\lambda\to\infty} \frac{\imf{\Phi}{\lambda}}{\lambda}
=\lim_{\lambda\to\infty} \frac{\lambda}{\imf{\Psi}{\lambda}}
=0.
\]

From Theorem 3.1 in \cite{MR0123362}, we deduce that
\[
	\lim_{t\to 0} \frac{\hat\tau_t}{t^{\gamma}}=0. 
\]for $\gamma\in (1,\alpha)\subset (1,2)$. $-\underline X$ and $\hat\tau$ satisfy the following: if $-\underline X_t=y$ then $\hat\tau_y\geq t$. Hence
\[
	\liminf_{t\to 0}\frac{-\underline X_t}{t^{1/\gamma}}
	\geq \liminf_{y\to 0+}\frac{y}{{\hat\tau_y}^{1/\gamma}}=\infty, 
\]which settles the first part of the proposition.

Let us pass to the second part. 
Let $\overline X$ be the running maximum process  of $X$ given by $\overline X_t=\sup_{s\leq t} X_s$. 
It is well known that the reflected process $\overline X-X$ is a Feller process (cf. \cite[Prop. 1, Ch. VI]{MR1406564}); regularity of $0$ for both half-lines implies that $0$ is a regular and instantaneous state for the reflected process  
so that we can define local times at zero, denoted $L$. 
Then, the upward process,  denoted $(\tau,H)$, is defined by 
\[
	(\tau,H)=(L^{-1}, X\circ L^{-1});
\]it allows one to sample the L\'evy process $X$ only at the times at which it achieves a new maximum. 
It is a two-dimensional subordinator. 
Since $X$ is spectrally positive,  
the Laplace exponent of $(\tau,H)$ can be obtained as follows: 
\[
	-\frac{1}{t}\log\esp{e^{-\lambda \tau_t -\mu H_t}}
	=-c\frac{\imf{\Psi}{\mu}-\lambda}{\imf{\Phi}{\lambda}-\mu}
\]where $c$ is a constant depending on the normalization of the local time $L$. (See for example formulae 9.2.8 and 9.2.9 in \cite{MR2320889}.) 

In particular, since $X$ is spectrally positive, 
the Laplace exponents of $L^{-1}$ and $H$ are given by
\begin{linenomath}
\begin{align*}
	-\frac{1}{t}\log\esp{e^{-\lambda \tau}}&= \frac{c\lambda}{\imf{\Phi}{\lambda}} 
	&\text{and}&
	&-\frac{1}{t}\log\esp{e^{-\mu H_t}}&= \frac{c\imf{\Psi}{\mu}}{\mu}. 
\end{align*}\end{linenomath}
To obtain the drifts of $\tau$ and $H$, 
first note that $\imf{\Psi}{\lambda}\to \infty$ as $\lambda\to\infty$, and thus  $\imf{\Phi}{\lambda}\to \infty$ as $\lambda\to\infty$. 
Hence, 
\[
\lim_{\lambda\to\infty} \frac{c\lambda}{\imf{\Phi}{\lambda}\lambda}=0
\]so that the drift of $\tau$ is zero. 
On the other hand, recall that $\lim_{\mu\to\infty}\mu^{-2} \imf{\Psi}{\mu}$ exists and equals the Gaussian coefficient of $X$ (see \cite[Proposition 2, Chapter 1.1]{MR1406564}). 
We deduce that the drift of $H$ equals $\sigma^2/2$.

Let us now obtain bounds on  the upper Blumenthal-Getoor indices for $\tau$. 
The upper index for $X$ has a similar definition to that of subordinators, but in terms of its Laplace exponent $\Psi$; 
it will be denoted by $\overline \alpha(X)$ in terms of our hypotheses, $\overline \alpha(X)\leq \tilde \alpha<1/(2-\alpha)$. 
Note that if $\delta<1$ and $1/(1-\delta)>\overline \alpha(X)$ then
\[
	\lim_{\lambda\to\infty}\frac{\lambda}{\lambda^{\delta}\imf{\Phi}{\lambda}}
	=\lim_{\lambda\to\infty}\bra{\frac{\imf{\Psi}{\lambda}}{\lambda^{1/(1-\delta)}}}^{1-\delta}
	=0.
\]We conclude that $\overline \alpha(\tau)\leq 1-1/\overline\alpha(X)$. 
Finally, for $H$, note that if $\delta\in (0,\alpha-1)$ then we have proved in \eqref{eqn_lowerindexH} that 
\[
	\lim_{\lambda\to\infty}\frac{\imf{\Psi}{\lambda}}{\lambda^{\delta}\lambda}=\infty. 
\]We conclude that $ \alpha-1\leq\underline \alpha(H)$.

Let us now show that under our hypotheses, $\overline\alpha(\tau)<\underline\alpha(H)$. 
Indeed, we have shown that $\overline \alpha(\tau)\leq 1-1/\overline\alpha(X)$ and $ \alpha-1\leq\underline \alpha(H)$. 
However, our hypothesis
\[
	\overline\alpha(X)<\frac{1}{2-\alpha}
\]implies $1-1/\overline \alpha(X)< \alpha-1$. 
Let $L=\tau^{-1}$. 
Recall that, from Lemme 4 in \cite{MR1141246} that the future minimum of the post-minimum process of $X$ has the same law as $H\circ L$ 
and that 
$\overline X=H^{-}\circ L$. 
We can therefore find $\gamma>1$ and $\delta\in (0,1)$ such that $\gamma\overline \alpha(\tau)<\delta<\underline \alpha(H)$. 
Therefore
\[
	\lim_{t\to 0}\frac{H_t}{t^{1/\delta}}=\lim_{t\to 0}\frac{H^-_t}{t^{1/\delta}}=\infty, 
	\quad
	\lim_{t\to 0}\frac{\tau_t}{t^{\gamma/\delta}}=0
\]and then (reasoning as with $-\underline X$ and $\hat \tau$ above)
\[
	\lim_{t\to 0}\frac{\overline X_t}{t^{1/\gamma}}
	=\lim_{t\to 0}\frac{H^-\circ L_t}{t^{1/\gamma}}	
	\geq \lim_{t\to 0}\frac{H^-_t}{\tau_t^{1/\gamma}}
	=\lim_{t\to 0}\frac{H^-_t}{t^{1/\delta}}\bra{\frac{t^{\gamma/\delta}}{\tau_t}}^{1/\gamma}
	=\infty. \qedhere
\]

Finally, in the case $\sigma>0$ we have $\overline X_t\geq \sigma^2 L_t/2$. 
Note that above we obtained the bound $\overline{\alpha}(\tau)\leq 1-1/\overline{\alpha}(X)$ using only the definitions of $\overline{\alpha}(\tau)$ and $\overline{\alpha}(X)$. 
Since $\overline{\alpha}(X) \leq 2$, we have $\overline{\alpha}(\tau)\leq 1/2$. 
Thus, for any $\gamma>1$ such that $\overline{\alpha}(\tau)<1/\gamma$, similarly as before 
\[
\lim_{t\to 0}\frac{\overline X_t}{t^{1/\gamma}}
\geq \frac{\sigma^2}{2} \lim_{t\to 0}\bra{\frac{t^{\gamma}}{\tau_t}}^{1/\gamma}
=\infty. \qedhere
\]
\end{proof}

\subsection{EI processes after their minimum} 
\label{subsection_PathTransformationEIPostMin}
The objective of this subsection is to prove Proposition \ref{proposition_LemmeFromBertoin} and therefore finish the proof of Propositions \ref{proposition_EnvelopesForEI} and \ref{proposition_CompactnessOfICRT}. 
We do this by first showing the result in discrete time and then passing to the limit. 
In discrete time, we first prove that the excursions below the running maximum are reversible, 
which allows us to prove a path transformation that explains Proposition \ref{proposition_LemmeFromBertoin} in the discrete time-case. 
Finally, we tackle some technical results which allow us to pass to the limit. 

Let $W$ be a discrete time EI process on $[n]$ 
as introduced in Section \ref{section_EIProcesses}. 
As in that section, we will mainly work with extremal EI processes $W$; 
recall that these are built from a sequence of increments,  say $x=(x_i)_{i=1}^n$. 
Also, the possible trajectories of $W$ are
\[
	\mc{P}=\set{x^\sigma: \sigma\text{ is a permutation of } [s]}. 
\]In general, the law of $W$ on $\mc{P}$ is uniform, 
except when different permutations give rise to the same trajectory. 
Hence, the law of $W$ is uniform on $\mc{P}$ when $x$ has the \emph{different subset sum property} 
($\sum_{i\in I_1} x_i\neq \sum_{i\in I_2} x_i$ if $I_1$ and $I_2$ are different subsets of $[s]$) 
and in particular when $x$ is a sample from a non-atomic distribution. 
We will use uniformity in establishing the invariance of $W$ under path-transformations by using bijections between the set $\mc{P}$. 
Also, general EI processes can be thought of as extremal ones with the jumps $x$ taken at random. 
In particular, random walks correspond to when $x$ is a sample from a given distribution. 

We will be interested in the \emph{(complete) excursions} of $W$ under its running maximum $\overline W$ 
given by $\overline W_j=\max_{i\leq j} W_i$. 
To define them formally, consider the (random) set $\mc{Z}=\set{i\in [n]_0: X_i=\overline X_i}$, 
where $[n]_0=\{0,1,\ldots, n \}$. 
If $\mc{Z}=\set{0=I_0<I_1<\cdots<I_K\leq n}$, 
we define the excursion intervals of $W$ under $\overline W$ as $[I_{j-1},I_{j})$ for $1\leq j\leq K$. 
($K=0$ if $X_n<\overline X_n=0$, in which case there are no excursion intervals.) 
	
Let $\mc{B}_k$ be the set of paths $w:[n]_0\mapsto \re$ in $\mc{P}$  having $k$ excursions below its maximum, 
for $k\in [n]_0$. 
We now  prove that the excursions of $X$ below its maximum are exchangeable.

\begin{lemma}\label{lemmaBijectionBetweenDiscretePermutationOfEI}
For any $k\in [n]$, let $\sigma$ be a permutation on $[k]$. 
For any $w\in \mc{B}_k$, let $\phi_{\sigma}(w):=\tilde{w}$ be the transformation that permutes the $k$ excursions of $w$ below its maximum by means of $\sigma$. 
Then, the transformation $w\mapsto \tilde w$ is a bijection. 
%
\end{lemma}
\begin{proof}
Note that $\phi_{\sigma^{-1}}\circ \phi_{\sigma}=\phi_{\sigma}\circ \phi_{\sigma^{-1}}=\id$. 
	\end{proof}
	
Now we define a transformation that permutes the excursions below the maximum of $W$, and prove that the distribution remains unchanged.
	
\begin{lemma}\label{lemmaExchangeabilityOfExcursions}
Let $X$ be an extremal EI process based on a deterministic difference sequence $x=(x_i)_{i=1}^n$ satisfying the different subset sum property. 
Fix deterministic permutations $(\sigma_k,k\in [n]_0)$ where $\sigma_k\in S_k$, 
and define $\phi(W)$ as $\phi_{\sigma_k}(W)$  if $W\in\mc{B}_k$. 
%
Then, we have the equality in distribution
\[
	W\stackrel{d}{=}\phi(W).
\]
\end{lemma}
\begin{proof}
Note that $\phi$ is a random bijection of $\mc{P}$, but equals the deterministic bijection $\phi_{\sigma_k}$ if $W\in \mc{B}_k$. 
Since $\phi$ does not change the number of excursions below the maximum, we have that $\phi(W)\in \mc{B}_k$ if and only if $W\in \mc{B}_k$. 
Fix $k\in [n]_0$ and $w\in \mc{B}_k$. 
Then
\begin{linenomath}
\begin{align*}
\proba{\phi(W)=w}
=\proba{\phi_{\sigma_k}(W)=w}
=\proba{W=\phi_{\sigma_k^{-1}}(w)}
=\frac{1}{n!}.
\end{align*}\end{linenomath}The last equality follows because, thanks to the different subset sum property, $W$ is uniform on $\mc{P}$. 
Therefore $\phi(W)$ is also uniform on $\mc{P}$ and hence equal in law to $W$. 

	\end{proof}
	
We can now obtain Proposition \ref{proposition_LemmeFromBertoin}, by applying the above result with the particular permutations $\sigma_k$ that just reverse the order of the excursions below the minimum. 
Let us see how this is true for discrete time EI processes and then how to pass to the limit.

\begin{proposition}\label{propoEqualityDiscretePostInfimumAndXCircD}
Let $X$ be a discrete time extremal EI process on $[0,1]$, having jumps at times $j/n$, for $j\in [n]_0$ constructed from the deterministic jump sequence $x=(x_i)_{i=1}^n$ which satisfies the different subset sum property. 
Then  $X$ reaches its minimum at a unique time $\rho\in [0,1)$,  its supremum at a unique time $\eta\in (0,1]$ and
	\begin{equation}\label{eqnDiscreteEqualityPostInfimumAndXCircD}
	\paren{\underline{\underline{X}}_{\rho+j/n}-\underline{\underline{X}}_{\rho}}_{0\leq j\leq n-n\rho}\stackrel{d}{=}\paren{\overline{X}\circ d(j/n)}_{0\leq j\leq n\eta }.
	\end{equation}
\end{proposition}
\begin{proof}
Uniqueness of the minimum and the maximum are obvious since otherwise there would be two different subsets of $x$ with the same sum. 

The proof equality in law is based on a simple path transformation of $X$, which is the composition of the time-reversal operation $X\mapsto (-X_{1-t})_{t\in [0,1]}$, together with the transformation $\phi$ obtained in Lemma \ref{lemmaExchangeabilityOfExcursions} when $\sigma_k$ is the reversal permutation sending $i$ to $n-i+1$. 
Call the result $\tilde X$. 
Since both transformation leading to $\tilde X$ preserve the law, we see that $\tilde X\stackrel{d}{=} X$. 
To conclude, just note that (with obvious notation) $\overline{\tilde X}\circ \tilde d= \dunderline X_{\rho+\cdot}-\dunderline X_\rho$. 
(See Figure \ref{figureBertoinLemmePathTransformation} for an illustration.)

\end{proof}

\begin{figure}
	\subfloat[Future infimum process and infimum]{	
		\includegraphics[clip, trim=4.5cm 16cm 1.7cm 4cm, width=.3\textwidth]{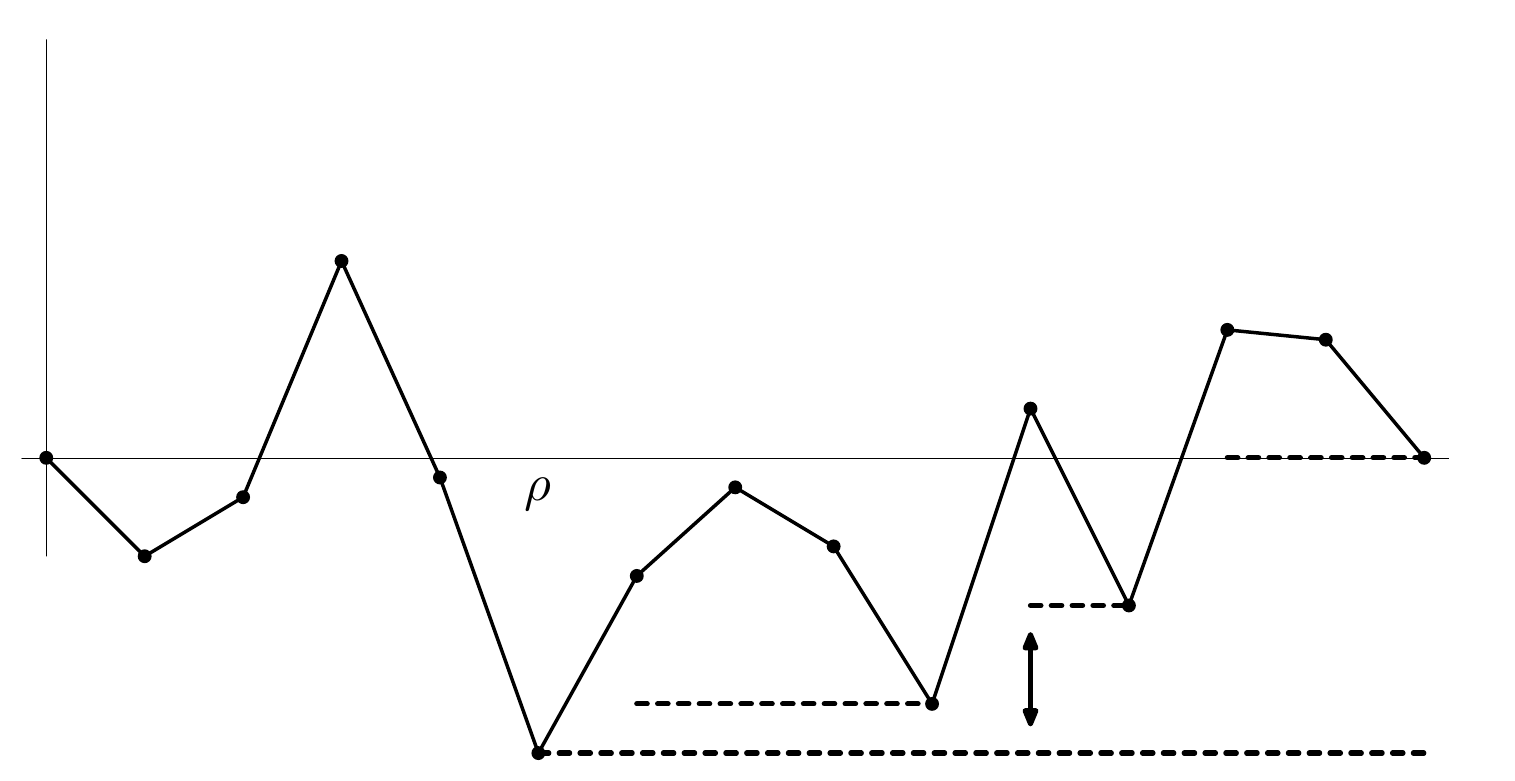}\label{figDiscreteEIWithFutureMinimum}}
	\subfloat[
	Time reversed path 
	]{	
		\includegraphics[clip, trim=4.5cm 16cm 1.7cm 4cm, width=.3\textwidth]{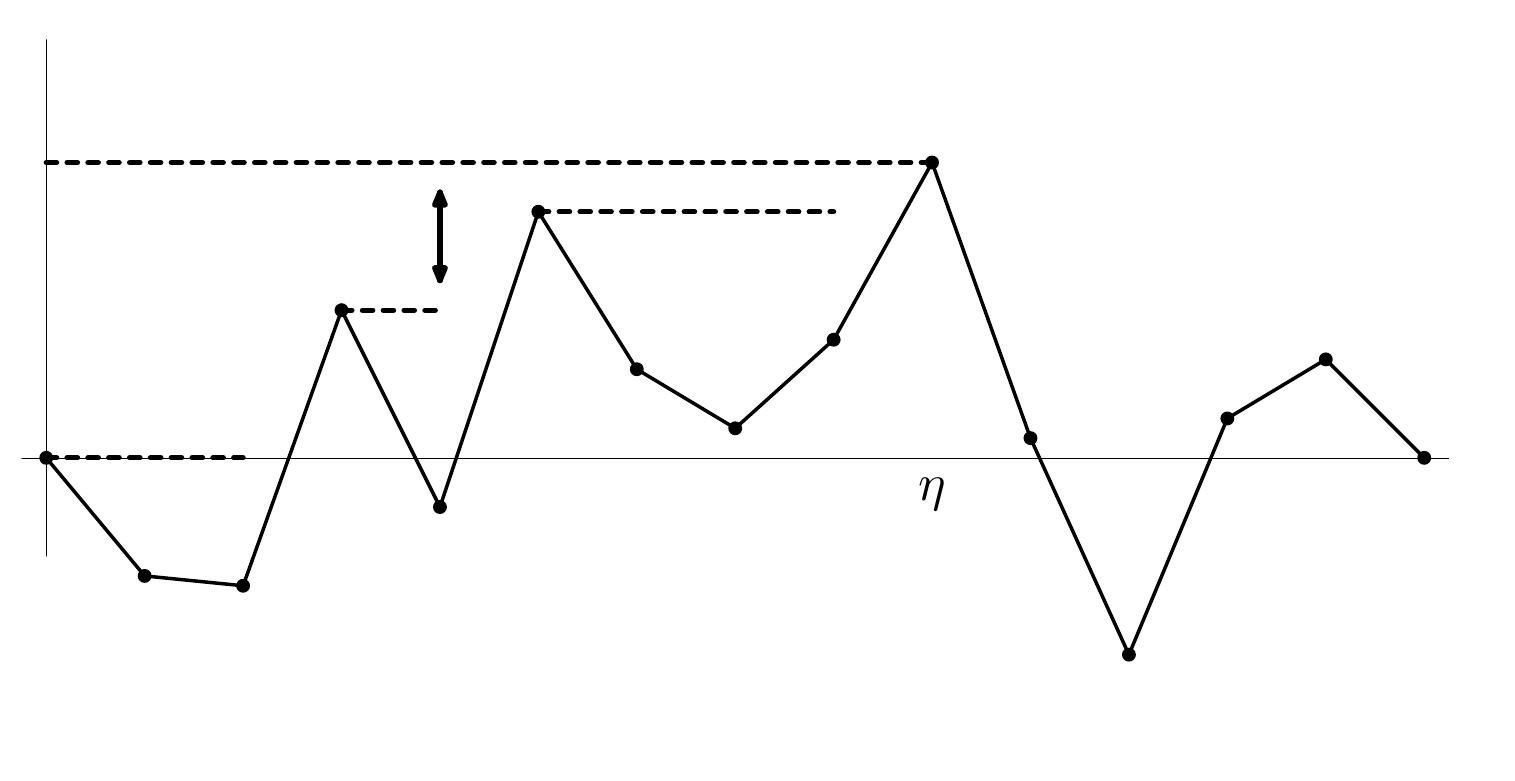}\label{figDiscreteEIWithFutureMinimumReversed}}
	\subfloat[Reordering of the excursions]{
		\includegraphics[clip, trim=4.5cm 16cm 1.7cm 4cm, width=.3\textwidth]{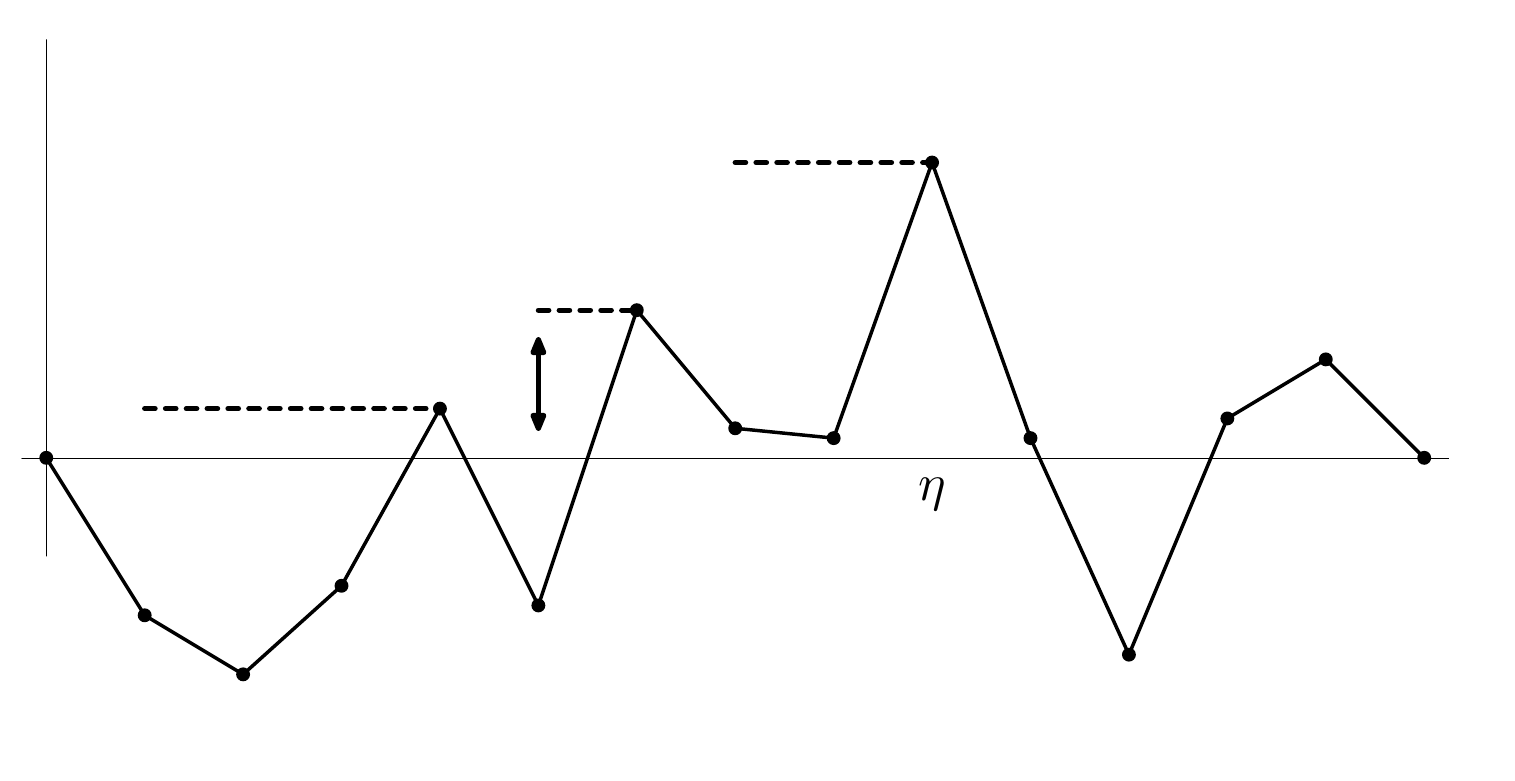}\label{figDiscreteEIWithFutureMinimumExchangeExcursions}}
	\caption{Path transformation illustrating Proposition \ref{propoEqualityDiscretePostInfimumAndXCircD}.}
	\label{figureBertoinLemmePathTransformation}
\end{figure}

Consider now an EI process $X$ as in Proposition \ref{proposition_LemmeFromBertoin} 
and let $X^n$ be such that
\[
	X^n_{t}=X_{(k+1)/n}\text{ if }t\in [k/n,(k+1)/n). 
\]Then, $X^n$ is an EI process on $\set{k/n: 0\leq k\leq n}$ and, by right continuity of $X$, 
$X^n\to X$ almost surely (as random elements of Skorohod space). 
Recall the definition of $d$ and define $d^n$ analogously to $d$ but for the process $X^n$. 
The following lemma will allow us to prove Proposition \ref{proposition_LemmeFromBertoin} by passing to the limit in Proposition \ref{propoEqualityDiscretePostInfimumAndXCircD}. 

\begin{lemma}\label{lemmaConvegenceOfX_nCircD_n}
For any fixed $t\in [0,1]$ we have $d^n(t)\to d(t)$ almost surely, 
and $X^n\circ d^n\to X\circ d$ in the sense of finite-dimensional distributions. 
\end{lemma}

For the proof, we need an additional lemma. 

The following lemma will be used to prove that 0 is instantaneous for $\overline{X}-X$ (meaning that, for every $\eps>0$ there exists $t_\eps\in (0,\eps)$ such that $\overline{X}_{t_\eps}-X_{t_\eps}>0$ ), 
which is needed to prove that $d^n_t\to d_t$. 
It is similar to the strong Markov property and can be proven quite simply by a path transformation due to Kallenberg. 

\begin{lemma}
	\label{lemmaOnExchangeabilityAtStoppingTime}
	Let $X$ be an extremal EI process on $[0,1]$ and let $T$ be a stopping time on [0,1]%
	. 
	Then the process $\tilde X$ given by
	\[
	\tilde X_t=
	\begin{cases}
	X_{t+T}-X_T	& 0\leq t\leq 1-T\\
	X_1-X_{(1-t)-}	& 1-T<t\leq 1
	\end{cases}
	\]has the same law as $X$. 
\end{lemma}
\begin{proof}
	The assertion will follow from the predictable mapping theorem for EI processes (cf. \cite[Theorem 4.7]{MR2161313}), once we show that random bijection $V$ of $[0,1]$ given by 
	\[
	V_t=\begin{cases}
	1-t	& 0\leq t\leq T\\
	t-T	& T<t\leq 1
	\end{cases}
	\]preserves Lebesgue measure and is predictable (in the sense that for any $t\geq 0$, the process $\indi{V_\cdot\leq t}$ is predictable). Indeed, the predictable mapping theorem tells us that the process $I:t\mapsto \int_0^1 \indi{V_s\leq t}\, dX_s$ has the same law as $X$ and we will now see that $\tilde X=I$. 
	
	That $V$ preserves Lebesgue measure can be seen simply as follows: the effect of $V$ is to interchange the intervals $[0,T]$ and $[T,1]$, reflecting the former. 
	
	On the other hand, since $V$ is c\`ag, predictability follows if $V$ is adapted. 
	However, note that
	\[
	\set{V_s\leq t}
	=\set{T\leq 1-t, T\leq s-t}\cup\set{T>1-t, 1-t\leq s}. 
	\]The first set on the union belongs to $\F_{s-t}\subset \F_s$, while the second one belongs to $\F_{1-t}\subset \F_s$  if $1-t\leq s$ and equals $\emptyset$ otherwise. In any case, we see that $\set{V_s\leq t}\in \F_s$. 
	
	Also, note that $\set{s\geq 0: V_s\leq t}$ is one of two stochastic intervals: 
	either $[T,T+t]$ if $t\leq 1-T$ or $[1-t,1]$ if $t>1-T$. 
	Hence, $\int_0^1 \indi{V_s\leq t}\,  dX_s=\tilde X_t$. 
	We deduce from the predictable mapping theorem that $\tilde X\stackrel{d}{=}X$. 
\end{proof}

Define $\mathcal{Z}:=\{t\in [0,1]:\overline{X}_t=X_t \}$.

\begin{proof}[Proof of Lemma \ref{lemmaConvegenceOfX_nCircD_n}]
We will use Lemma 8 in \cite{2019arXiv191009501M}, 
which gives conditions for the continuity (on Skorohod's space) of the hitting times and hitting positions of an open set. 
Fixing $t\in (0,1)$, we will consider the (random) open set $O=(\overline X_t,\infty)$. 
The aforementioned lemma 8 tells us that if $X$ cannot approach $\overline O$ without reaching $O$ and if there are no jumps from $\partial O$ into $O$ then $d^n_t\to d_t$ and $X^n_{d^n_t}\to X_{d_t}$. 
To prove this, let us note that $X$ does not jump when it approaches a level of its previous maximum. 
In other words, that if $R=\inf\set{s\geq t: X_{s-}>\overline X_t}$, then $\Delta X_R=0$. 
This follows since otherwise the reflected process $\hat X$ features a jump time $U$ such that $\hat X_{U+\cdot}-\hat X_U$ is negative on a right neighborhood of $0$. However, by Lemma \ref{lemmaOnExchangeabilityAtStoppingTime} tells us that the latter has the same behavior as $\hat X$ on a neighborhood of zero Theorem 1 in \cite{2019arXiv190304745A} tells us that infinite variation EI processes such as $\hat X$ achieve positive values immediately. 
But then, if $A$ is the approach time of $O$ given by $\inf\set{s\geq t: X_{s-} \text{ or }X_s\geq \overline X_t}$, we see that $X_A\geq \overline X_t$. Also, Lemma \ref{lemmaOnExchangeabilityAtStoppingTime} tells us that right after $X_A$, $X$ reaches higher levels. Hence, we see that $A$ coincides with the hitting time of $O$. 
We have proved the two conditions implying that, for every fixed $t\in (0,1)$, almost surely, $d^n_t\to d_t$ and $X^n_{d^n_t}\to X_{d_t}$. The finite-dimensional convergence now follows. 
\end{proof}

\begin{proof}[Proof of Proposition \ref{proposition_LemmeFromBertoin}] 
By Proposition \ref{propoEqualityDiscretePostInfimumAndXCircD}, 
we have that
\[
	\underline{\underline{X}}^n_{\rho_n+\cdot}-\underline{\underline{X}}^n_{\rho_n}
	\stackrel{d}{=} \overline{X}^n\circ d^n
\]
Regarding the left hand side, recall that the minimum is a continuous operation on Skorohod space (cf. \cite[\S 6]{MR561155}) and that $\rho_n\to\rho$ (by uniqueness of the minimum). 
Hence, we see that
\[
	\underline{\underline{X}}^n_{\rho_n+\cdot}-\underline{\underline{X}}^n_{\rho_n}
	\to \underline{\underline{X}}_{\rho+\cdot}-\underline{\underline{X}}_{\rho}.
\]
But then, Proposition \ref{propoEqualityDiscretePostInfimumAndXCircD} and Lemma \ref{lemmaConvegenceOfX_nCircD_n} tell us that the \cadlag\ processes 
\[
 	\underline{\underline{X}}_{\rho+\cdot}-\underline{\underline{X}}_{\rho}
	\quad\text{and}\quad  
	X\circ d 
\]have the same finite-dimensional distributions, so forcibly the same law. 
\end{proof}

We have just finished providing the proofs of the main results of the paper, 
contained in the statements of the Propositions and the Theorem in  Section \ref{sectionIntroProfileOfTGDS}. 
All that remains is to give further remarks on the examples, by proving Corollary \ref{corolarioConvergenciaCGWVarFinIntro} and showing how one can construct degree sequences satisfying the hypotheses of Theorem \ref{teoUnicidadEnFnalGralIntro}.

\section{Some remarks on the examples}
\label{section_ExampleSection}

This section has two objectives. 
First, to prove Corollary \ref{corolarioConvergenciaCGWVarFinIntro}. 
Secondly, to construct degree sequences showing the general applicability of Theorem \ref{teoUnicidadEnFnalGralIntro}.

\subsection{Application to Conditioned Galton-Watson trees}
Let us turn to the proof of Corollary \ref{corolarioConvergenciaCGWVarFinIntro}. 
Part of the interest of this proposition lies in showing how one can apply Theorem \ref{teoUnicidadEnFnalGralIntro} to non-extremal EI processes. 

Let $\mu$ be a critical and aperiodic offspring distribution 
which either has finite variance or regularly varying tails 
as stated in Corollary \ref{corolarioConvergenciaCGWVarFinIntro}. 
Consider a random walk $W$ with jump distribution $\tilde \mu$ given by $\tilde \mu_k=\mu_{k+1}$. 
$W$ oscillates as it has mean zero and therefore its hitting time $T$ of $-1$ is finite almost surely. 
From Section 1.2 in \cite{MR2203728} (see also the introduction to \cite{MR3449255}), 
we know that the BFW of the $\mu$-Galton-Watson tree has the same law 
as the random length sequence $1+W_0,1+W_1,\ldots, 1+W_T$. 
Therefore the size of of the tree has the same law as $T$ 
and aperiodicity implies that 
$\proba{T=n}>0$ for large enough $n$, so that the $\text{CGW}(n)$ tree is well defined. 
Also, one sees that the BFW $X^n$ of the $\text{CGW}(n)$ tree 
has the same law as $(1+W_i,i\leq n)$ conditionally on $T=n$. 
This process can also be obtained by means of the Vervaat transformation 
of the random walk bridge $W^{b}$ of length $n$ from $1$ to $0$ 
 by interchanging the pre first minimum and post first minimum parts of the trajectory. 
The random walk bridge has the law of $(1+W_i,i\leq n)$ conditionally on $W_n=-1$. 
Since, by Kemperman's formula 
(related to the preceeding assertion on the Vervaat transformation in \cite{MR1785529}, \cite[Ch. 6]{MR2245368}), 
$\proba{T=n}=\proba{W_n=-1}/n$, 
then the bridge is also well defined for large $n$. 
As is well known, the random walk $W$ has a scaling limit. 
If $\mu$ has finite variance $\sigma^2>0$, 
$(W_{\floor{nt}}/\sqrt{n},t\geq 0)$ converges weakly in Skorohod space 
to $Z=\sigma B$ where $B$ is a Brownian motion. 
If, on the other hand, $\overline \mu_n\sim n^{-\alpha}\imf{L}{n}$ 
for some regularly varying function $L$ and some $\alpha\in (1,2)$  
then $(W_{\floor{nt}}/b_n,t\geq 0)$ converges weakly 
to an $\alpha$-stable L\'evy process $Z$ with no negative jumps, where $b_n$ is a regularly varying sequence of index $1/\alpha<1$. Note that in both cases the size of the tree is $n$ and $n/b_n\to \infty$. 
There are also convergence results for the corresponding bridge and for the BFWs. 
Indeed, the latter follows from the former by continuity of the Vervaat transformation on functions that achieve their minimum uniquely and continuously. 
To get convergence of the bridges, one can use \cite[Proposition 13]{kersting2011height}; 
and for the convergence of 
BFWs of a $\text{CGW}(n)$, recall that 
the classical result from \cite{MR0415706} implies that the BFW of the CGW$(n)$ $\tilde W$ is such that $(\tilde W_{nt}, 0\leq t\leq 1)$ converges weakly to $X=\sigma e$ where $e$ is the so-called normalized Brownian excursion. 
When $\mu$ has regularly varying tails, Proposition 13 in \cite{kersting2011height} or 
Lemma 4.5 from \cite{MR1964956} give us the same result for $\tilde W$ under the spatial normalization $b_n\sim n^{1/\alpha}\tilde L(n)$ for some slowly varying function $\tilde L$, where now $X$ is the normalized stable excursion. These stable excursions can be obtained from the bridge by applying the Vervaat transformation as in \cite{MR515820} and \cite{MR1465814}.

\begin{lemma}
Let $X$ be the normalized excursion of an $\alpha$-stable process with no negative jumps and index $\alpha\in (1,2]$. 
Then $\int_0^1 1/X_s\, ds<\infty$. 
\end{lemma}
\begin{proof}
The lemma follows simply from results in the literature (for $\alpha=2$ several ways to prove the integrability of $1/\bo{e}$ on $[0,1]$ are given in \cite[Remark 5.2]{MR2245498}). 
However, it also follows from the analysis in Section \ref{section_LampertiForEI}, 
which is why we briefly sketch both approaches. 

In the finite variance case when $\alpha=2$, $X$ and $(X_{1-t},0\leq t\leq 1)$ have the same law and $(X_t,t\leq 1/2)$ is absolutely continuous with respect to the three-dimensional Bessel process $R$. 
On the other hand, the Brownian escape process $F_t=\inf_{s\geq t} R_s$ is the inverse of a stable subordinator $\tau$ of index $1/2$ as shown in \cite{MR542136}. 
By \cite[Thm. 3.1]{MR0123362}, we see that $\limsup_{t\to 0}\tau_t/t^2=0$. 
Hence, $\liminf_{t\to 0}R_s/\sqrt{s}\geq \liminf_{t\to 0}F_s/\sqrt{s}=\infty$. 
Therefore, for some random $\eps>0$, $\int_{[0,\eps]\cup[1-\eps,1]} 1/X_s\, ds<\infty$. 
Since $X_t=0$ only for $t=0,1$, we deduce the result in this case. 

When $\alpha<2$, the statement is proved using continuous state branching processes with immigration in Lemma 9 and the final comment of \cite{MR2018924}. 
It is also proved in \cite{angtuncioHernandezThesis} by applying the stretched exponential bounds on the tails of the distribution of the height of CGW$(n)$ trees of \cite{MR3185928}. 
However, we can give a proof as in the preceeding paragraph. 
First let us note that the statement follows if we find power law bounds for the pre and post minimum processes of $X^{b}$, by the Vervaat transformation. 
Let $\rho_t$ be the unique place that $X^{b}$ reaches its minimum on $[0,t]$; 
it is known that the minimum is achieved uniquely and continuously under our hypothesis. 
Therefore $\rho_t=\rho$ as long as $t$ is close enough to $1$ so that for positive $\beta$ we have
\[
\proba{\liminf_{s\to 0}\frac{X^{b}_{\rho+s}-X^{b}_\rho}{s^{\beta}}=\infty}
=
\lim_{t\to 1} \proba{\liminf_{s\to 0}\frac{X^{b}_{\rho_t+s}-X^{b}_{\rho_t}}{s^{\beta}}=\infty}. 
\]However, recall that $X^{b}$ is locally absolutely continuous with respect to $X$ (cf. \cite{MR2789508}): 
if $t<1$ then
\[
\proba{\liminf_{s\to 0}\frac{X^{b}_{\rho_t+s}-X^{b}_{\rho_t}}{s^{\beta}}=\infty}
=\esp{\indi{\liminf_{s\to 0}\frac{X_{\rho_t+s}}{s^{\beta}}=\infty}\frac{\hat f_{1-t}(X_s)}{\hat f_1(t)}}=0,
\]where $\hat f_t$ is the density of $-X_t$ and the last equality follows from Proposition \ref{proposition_LevyEnvelopes} whenever $\beta\in (1/\alpha, 1)$ for some $\beta<1$. 
An analogous result holds for the pre minimum of the bridge. 
Since the minimum is achieved uniquely, we see that, almost surely
\[
	\int_0^1 \frac{1}{X^{b}_s-X^{b}_\rho}\, ds<\infty. \qedhere
\]
\end{proof}

\begin{proof}[Proof of Corollary \ref{corolarioConvergenciaCGWVarFinIntro}]
We have already argued that, under the conditions of Corollary \ref{corolarioConvergenciaCGWVarFinIntro}, 
the breadth-first walks have scaling limits equal to the Vervaat transform of an exchangeable increment process of unbounded variation: either the Brownian bridge or the stable bridge of index $\alpha\in (1,2)$. 
Thanks to Skorohod's representation result, 
we can place ourselves in a probability space where the convergence $W^{\text{b}}_{\floor{n\cdot}}/b_n$ holds almost surely. 

Let $\beta^n_1\geq \beta^n_2\geq\cdots>0$ be the ranked jumps of $W^{b}/b_n$ and $\beta^1\geq \beta_2\cdots$ be the corresponding ones for $X$. If we add one to the former, we get the ordered child sequence  associated to the random degree sequence  of the conditioned Galton-Watson trees. 

Since $W^{b}$ and $X^{b}$ are exchangeable increment processes, 
the convergence of $W^{b}_{\floor{n\cdot}}/a_n\to X^{b}$ is equivalent (thanks to Theorem 2.2 in \cite{MR0394842} as recalled in Subsection \ref{subsection_ConvergengeOfBFWs}) to that of
\[
	\beta^n_i\to \beta_i
	\quad\text{and}\quad
	\sum_i (\beta^n_i)^2\to \sum_i \beta_i^2. 
\](in the regularly varying case) and to
\[
	\beta^n_i\to 0
	\quad\text{and}\quad
	\sum_i (\beta^n_i)^2\to \sigma^2. 
\]in the finite variance case. 

Also, our discussion in Section \ref{sectionIntroProfileOfTGDS} on the fact that conditioned Galton-Watson trees, conditioned on their degree sequence being $s$ have law $\p_s$, implies that\[
	\proba{W^{b}=w}=\esp{\imf{\p_{\beta^n}}{W^n=w}}
\]where $\p_{\beta^n}$ is the law of an exchangeable increment process constructed from the sequence of jumps $\beta^n$ as in Section \ref{section_EIProcesses}. The similar statement
\[
	\proba{X^{b}\in A}=\esp{\imf{\p_\beta}{X^{b}\in A}}
\]is true, where $\p_{\beta}$ is the law of the continuous EI process obtained through Kallenberg's representation in Equation \eqref{eqnEIP} out of the jumps $\beta$. 
Since the sequences $\beta^n$ and $\beta$ almost surely satisfy the conditions of Theorem \ref{teoUnicidadEnFnalGralIntro}, we obtain
\[
	\imf{\se_{\beta^n}}{\imf{H}{X^n,C^n,Z^n}}
	\to \imf{\se_{\beta}}{\imf{H}{X,C,Z}}
\]for continuous and bounded functionals $H$. 
The bounded convergence theorem let's us conclude that, 
\[
	\imf{\se}{\imf{H}{X^n,C^n,Z^n}}
	\esp{\imf{\se_{\beta^n}}{\imf{H}{X^n,C^n,Z^n}}}
	\to \esp{\imf{\se_{\beta}}{\imf{H}{X,C,Z}}}
	=\imf{\se}{\imf{H}{X,C,Z}}, 
\]which concludes the proof of Corollary \ref{corolarioConvergenciaCGWVarFinIntro}. 
\end{proof}

\subsection{Examples of degree sequences satisfying the assumptions of Theorem \ref{teoUnicidadEnFnalGralIntro}}
\begin{example}[Infinite variation EI process with $\sigma=0$]
Let $(\beta_j,j\in \na)\subset \re_+$ be a deterministic sequence such that $\sum \beta_j^2<\infty$ and $\sum \beta_j=\infty$. 
For simplicity assume $\beta_j> \beta_{j+1}$ for every $j$.

Consider a sequence $(b_n,n\in \na)\subset \na$ such that $b_n\to \infty$, and assume that for $M_n:=\sup\{j:\beta_ja_n\geq 1 \}$, we have $\sum_1^{M_n}\beta_j/b_n\to 0$ as $n\to \infty$.
For the child sequence $d^n_j=\floor{\beta_jb_n}$, define for $n$ big enough the degree sequence 
\[
N^n_0=1+\sum_1^{M_n} d^n_j-M_n\qquad\mbox{and}\qquad N^n_{d^n_j} =1 \quad \mbox{for }j\in [M_n].
\]
To obtain the scaling of the BFW given in hypothesis \defin{Hubs}, we compute
\[
\sum_{0}^{d^n_1}\frac{(j-1)^2}{b_n^2}N^n_j
=\sum_1^{M_n} \frac{\floor{\beta_j b_n}^2}{b_n^2}-\frac{s_n}{b_n^2}+\frac{2}{b_n^2}
=\sum_1^{M_n} \frac{\floor{\beta_j b_n}^2}{b_n^2}-\sum_1^{M_n}\frac{\floor{\beta_jb_n}}{b_n^2}+\frac{1}{b_n^2}.
\]The first term is bounded by $\sum \beta_j^2$ and the negative of the second by $\sum_1^{M_n}\beta_j/b_n$. 
Hence, by the hypothesis on $M_n$ and $b_n$, the above display converges to $\sum \beta_j^2$.  
Finally, for every $m\in \na$ we obtain $\varliminf s_n/b_n\geq \sum_1^m \beta_j$, implying the hypothesis of Theorem \ref{teoUnicidadEnFnalGralIntro}. 

An explicit sequence satisfying the previous assumptions is $\beta_j=j^{-\alpha}$ and $b_n=\floor{1/\beta_n}$ for $j,n\in \na$ and $\alpha\in (1/2,1)$.
To see this, note that for $n$ big enough we have for some positive $c$ that $\sum_1^n\beta_j^\gamma\approx c \paren{n^{1-\alpha\gamma}-1}$ which is finite or infinite, for $1/\alpha<\gamma$ or $1/\alpha>\gamma$ respectively.
Therefore, the hypothesis on $\alpha$ implies $\sup\left\{\gamma:\sum \beta_j^\gamma=\infty \right\}=\inf\left\{\gamma:\sum \beta_j^\gamma<\infty \right\}=1/\alpha\in (1,2)$. 

Now, let $\{x \}$ denote the fractional part of $x\in\re$.  
Since $n^{-\alpha}\floor{n^\alpha}=1-n^{-\alpha}\{n^\alpha \}\leq 1$, then $M_n\leq n$. 
It follows that for some positive constant $c$
\[
\frac{1}{b_n}\sum_1^{M_n}\beta_j\leq \frac{1}{\floor{n^\alpha}}\sum_1^{n}\beta_j\leq c\frac{1}{\floor{n^\alpha}}\paren{n^{1-\alpha}-1}\sim \paren{n^{1-2\alpha}-n^{-\alpha}}  \to 0.
\]
\end{example}

\begin{example}[EI process with $\sigma>0$ and jumps]
	Let $X$ be an extremal EI process with parameters $(0,\sigma,\beta)$, where $\sigma>0$, having jumps $(\beta_j,j\in [M])$, with $M\in \na\cup \{\infty\}$ and $\sum_1^M \beta_j^2<\infty$.
	For simplicity assume $\beta_j>\beta_{j+1}>0$ for every $j\in [M-1]$. 
	Following \cite{MR3188597}, consider a degree sequence $\bo{s_n}=(N^n_j)$ of size $s_n=\sum N^n_j$ such that $\Delta_n=\max\{j:N^n_j>0 \}=o(\sqrt{s_n})$ and 
	\[
	\sigma^2=\lim_n\sum_1^{\Delta_n}\frac{j^2}{s_n}N^n_j-1. 
	\]We know from \cite[Lemma 7]{MR3188597} that in this case $b_n=\sqrt{s_n}$, and the rescaled BFW's converge to the Brownian excursion. 
	We add jumps to the degree sequence. 
	
	Consider $M_n\uparrow M$ satisfying $\sum_1^{M_n}\beta_j/\sqrt{s_n}\to 0$ as $n\to \infty$, and define $d^n_j=\floor{\beta_j\sqrt{s_n}}$ for $j\leq M_n$. 
	Construct the degree sequence $\bo{\tilde s_n}=(\tilde N^n_j)$ as
	\[\tilde N^n_i= \begin{cases} 
	N^n_0+\sum_1^{M_n}[d^n_j-1] & i=0 \\
	N^n_i & 1\leq i\leq \Delta_n\\
	1 & d^n_{M_n}\leq i\leq d^n_1,
	\end{cases}
	\]noting that $\Delta_n<d^n_{M_n}$ for $n$ big enough. 
	If $\tilde  s_n$ is the size of $\bo{\tilde s_n}$, then 
	\[
	\tilde s_n/s_n-1=\sum_1^{M_n}d^n_j/s_n\leq \frac{1}{\sqrt{s_n}}\sum_1^{M_n} \beta_j \to 0
	\]by hypothesis.
	Finally, since for the child sequence $(\tilde d^n_j)$ of $\bo{\tilde s_n}$ we have
	\[
	\sum \frac{[j-1]^2}{\tilde s_n}\tilde N^n_j
	=\sum_1^{\Delta_n} \frac{j^2}{\tilde s_n}N^n_j-1+\sum_1^{M_n} \frac{\floor{\beta_j \sqrt{s_n}}^2}{\tilde s_n}+\frac{2}{\tilde s_n}
	\leq \frac{s_n}{\tilde s_n}\sum_1^{\Delta_n} \frac{j^2}{s_n}N^n_j-1+\frac{s_n}{\tilde s_n}\sum_1^{\infty} \beta_j^2+1,
	\]using dominated convergence the left-hand side converges to $\sigma^2+\sum_1^M \beta_j^2$, and thus, the scaling is $\tilde b_n=\sqrt{\tilde s_n}$.
\end{example}

\section*{Acknowledgements}
\small{Research supported by CoNaCyT grant FC 2016 1946 and UNAM-DGAPA-PAPIIT grants IN115217 and IN114720.}
\bibliography{GenBib}
\bibliographystyle{amsalpha}
\end{document}